%% file: al_learn_TAC_arXiv_v4.tex
\documentclass[10pt,final,twocolumn]{IEEEtran}
\usepackage{graphicx}
\usepackage{tabularx}
\input{al_TAC_shared}
\ifCLASSINFOpdf
\else
\fi
\begin{document}
%
% paper title
% Titles are generally capitalized except for words such as a, an, and, as,
% at, but, by, for, in, nor, of, on, or, the, to and up, which are usually
% not capitalized unless they are the first or last word of the title.
% Linebreaks \\ can be used within to get better formatting as desired.
% Do not put math or special symbols in the title.
\title{On the analysis of inexact augmented Lagrangian schemes for misspecified conic convex programs}

% author names and affiliations
% use a multiple column layout for up to three different
% affiliations

\author{\IEEEauthorblockN{
N. S. Aybat,  H.~Ahmadi, and U.~V.~Shanbhag\IEEEauthorrefmark{1}}\thanks{All the authors are at Ind. \& Manuf. Engg. at The Penn.
	State University, PA, USA and are reachable at  {\tt nsa10,udaybag@psu.edu} and {\tt
		ahmadi.hesam@gmail.com}. Their
		research was partially supported by NSF CMMI-1400217 and NSF CMMI-1246887 (CAREER, Shanbhag).}}

\maketitle
% As a general rule, do not put math, special symbols or citations
% in the abstract
\begin{abstract}
We consider {the} misspecified optimization problem \sa{of minimizing a} %that requires of minimizing a
convex function $f(x;\theta^*)$ in $x$ over a {conic} constraint set
represented by $h(x;\theta^*) \in \cK$, where $\theta^*$ is an unknown (or
misspecified) vector of parameters, $\cK$ is a %proper 
\sa{closed convex} cone and
$h$ is affine in $x$. Suppose $\theta^*$ is unavailable but may be learnt by a
separate process that generates a sequence of estimators $\theta_k$, each of
which is an increasingly accurate approximation of $\theta^*$.  We develop a
first-order inexact augmented Lagrangian~(AL) scheme for computing an optimal
solution $x^*$ corresponding to $\theta^*$ while simultaneously learning
$\theta^*$. In particular, we derive rate statements for such schemes when the
penalty parameter {sequence} is either constant or increasing and derive bounds
on the overall complexity in terms of proximal-gradient steps when
AL~subproblems are \sa{inexactly} solved via an accelerated proximal-gradient scheme.
Numerical results for a portfolio optimization problem with a misspecified
covariance matrix suggest that these schemes perform well in practice {while} 
naive sequential schemes  may perform poorly in comparison.
\end{abstract}

% no keywords

% For peer review papers, you can put extra information on the cover
% page as needed:
% \ifCLASSOPTIONpeerreview
% \begin{center} \bfseries EDICS Category: 3-BBND \end{center}
% \fi
%
% For peerreview papers, this IEEEtran command inserts a page break and
% creates the second title. It will be ignored for other modes.
\IEEEpeerreviewmaketitle

\section{Introduction}
Consider an
optimization problem in $n$-dimensional {Euclidean} space %given by (${\cal C}(\theta^*)$)
defined as follows: \vspace*{-1mm}
\begin{equation}
{\cal C}(\theta^*):\qquad X^*(\theta^*)\triangleq\argmin_{x \in X \cap {\cal H}(\theta^*)} \ f(x;\theta^*),
\vspace*{-1mm}
\end{equation}
where $\theta^*\in\R^d$ denotes the \sa{unknown} parametrization of the objective and
constraints.  Traditionally, optimization {research} has considered
settings where $\theta^*$ is available a priori, {{in contrast with} robust and stochastic optimization approaches.}
%a singular exception being robust optimization approaches.

%there have also been concerted efforts
%to consider least two distinct generalizations when
%for the case ${\cal H}(\theta^*)= \Real^n$, namely \us{through} {\it robust}  optimization approaches.\\
\noindent {\bf Robust {optimization}:} %For instance,
Suppose $\theta^*$ is unavailable, but one has access to an uncertainty set ${\cal T}$.
%corresponding to $\theta^*$.
%For simplicity assume ${\cal H}(\theta^*) = \Real^n$, in this scenario,
{\emph{Robust optimization} \sa{(RO)} %approaches minimize
considers {minimizing} $f(\cdot;\theta)$ {over} the worst-case realization of
%the worst-case {value} {that $\min_{x\in X}f(x;\theta)$ takes for any
$\theta\in{\cal T}$, i.e.,}
%as captured by the following} {formulation}:
\vspace*{-1mm}
\begin{equation}
%\qquad\qquad
\min_{x \in X } \{\max_{ \theta \in {\cal T}} \ f(x;\theta):\ x\in \cH(\theta),~\forall\theta\in\cT\}. %\tag{Robust Optimization}
%\vspace*{-1mm}
\end{equation}
%\noindent {\bf Stochastic approaches~\cite{shapiro09lectures}:} An alternate approach considers
%an uncertain regime where $\theta: \Omega \to \Real^d$ is an
%$d-$dimensional random variable defined on a suitable probability space.
%The resulting stochastic optimization schemes consider the minimization
%of an expectation:
%\begin{align}\min_{x \in X } \ \mathbb{E}[f(x;\theta)].
%	\tag{Stochastic Optimization}
%\end{align}
{%Robust optimization
This has proven {to be} useful %to be an enormously useful technique
	in resolving %resolution of
problems in design, control~\cite{burger14polyhedral,bertsimas07constrained}, and optimization~\cite{bental09robust}. \sa{That said, {RO} is known to produce conservative solutions and is {intractable} in general, e.g., \emph{not tractable} when $\cH(\theta) { \ \triangleq \ } \{x:\ h(x;\theta)\preceq_{\mathcal{K}}\mathbf{0},~\forall\theta\in\mathcal{T}\}$,  {where} $\mathcal{K}$ is a general closed and convex cone~\cite{Bertsimas2006}.}

\noindent {\bf Stochastic optimization} \sa{(SO)} techniques~\cite{shapiro09lectures,jie18stochastic,wilson18adaptive} assume %that
the parameter $\theta$ %is uncertain and
follows \sa{a \emph{prescribed} probability
distribution $\mathcal{D}$, i.e., $\theta\sim\mathcal{D}$}. The 
risk-neutral focus \sa{in (SO)} lies in computing %computing an $\bar{x}$ that solves
%{\small
%\begin{subequations}
\begin{align}
 &\min_{x\in X}\{\sa{\mathbb{E}_{\cD}}[f(x;\theta)]: \sa{\mathbb{E}_{\cD}}[h(x;\theta)]\preceq_{\mathcal{K}}\mathbf{0}\}. \label{eq:expectation}
% &\min_{x\in X}\{\mathbb{E}_{\theta}[f(x;\theta)]: \mathbb{P}_{\theta}\big(\cS(x)\big)\geq 1-\tau\}, \label{eq:chance}
 \end{align}
% \end{subequations}}%
Suppose $X$ is convex, $f(.;\theta)$ and $h(.;\theta)$ are convex differentiable for all $\theta\in\Theta$ and an unbiased oracle is given that produces $\nabla_x f(x;\theta)$ and $\nabla_x h(x;\theta)$ {under suitable bias and moment assumptions}, then an $\epsilon$-solution to \eqref{eq:expectation} can be computed in $\mathcal{O}(1/\epsilon^2)$ projected gradient steps~\cite{jiang16variational}.
 %for the problem with expectation constraint.
\sa{Alternatively, the constraint in \eqref{eq:expectation} can be {approximated} by $\sa{\mathbb{P}_{\cD}}\big(\cS %_{\vs{_{\theta}}}
	(x) %{ \ \neq \ \emptyset} 
	\big)\geq 1-\tau$, where $S %_{\vs{\theta}}
	(x)\triangleq\{\theta:\ h(x;\theta)\preceq_{\mathcal{K}}\mathbf{0}\}$ and $\tau\in(0,1)$; the resulting chance-constrained variant is intractable in general, and even {for} simple distributions ${\cD}$. ~\cite{shapiro09lectures}.}
	%\vs{[\bf uvs: Did we define $\mathcal{D}$ or should we just drop it and say simple distributions.]}. %Usually the challenging assumption is the existence of an oracle that produces $\nabla_x f(x;\theta)$ and $\nabla_x h(x;\theta)$ with suitable bias and moment assumptions.}
		
Motivated by the increasing accessibility to data, we consider an alternate approach where $\theta$ has a
nominal or true value $\theta^*$ {computable} by solving a
suitably defined learning problem: \vspace*{-2.5mm}
{%\small
\begin{equation}
\label{eq:learning-problem}
({\cal E}):\qquad \min_{\theta \in \Theta} \  \ell(\theta).
\vspace*{-0.5mm}
\end{equation}}%
\sa{$\theta^*$ may simply be defined as an optimal solution to some learning problem as in~\eqref{eq:learning-problem} due to its desirable statistical guarantees.}
Such problems routinely arise when $\theta^*$ is idiosyncratic to
the problem and may be learnt by the aggregation of data; instances arise when
attempting to learn covariance matrices associated with a collection of stocks,
efficiency parameters associated with machines on a supply line, or demand
parameters associated with a supply chain. A natural approach %in this case
is to first estimate $\theta^*$ with high accuracy and then solve the parametrized
problem. Yet, in many
instances, this {\it sequential} approach cannot be adopted for at
	least two reasons:
% uvs -- dropped this since we do not account for streaming
% observations. \noindent (i) Observations are not available a priori and appear in a streaming fashion;\\
\noindent \textbf{(i)} The learning problem can be {very large, \emph{precluding} a highly
accurate a priori parameter resolution in a reasonable time; hence, making it impractical to solve
%to \us{allow for} solving
the original problem, i.e., the decision maker may have to wait for a long time during the learning stage with \emph{no} availability of an estimate solution to $\cC(\theta^*)$;}
%\uss{with no availability of a solution to the original problem during this learning process};\\
\noindent \textbf{(ii)} Unless the learning problem can be solved {\it exactly}
in finite time, sequential schemes are not asymptotically convergent and
can, at best, provide approximate solutions. Indeed, \sa{the error due to inexactly solving \eqref{eq:learning-problem}}
%the lack of exactness arising from the error in learning 
cascades into the resolution of the subsequent optimization problem.

{Accordingly, we consider %the development of
designing schemes that generate %sequences
$\{x_k,\theta_k\}$ such that
%\vspace*{-2mm}
%\begin{equation*}
$\sa{\lim_k}\|\theta_k - \theta^*\|=0$, $\sa{\liminf_k} d_{X^*(\theta^*)}(x_k) =0$, %as $k \to \infty$,
%\vspace*{-1mm}
%\end{equation*}
where $\theta^*$ is the unique solution to (${\cal E}$), 
%$X^*(\theta^*)$ is the set of optimal solutions for ${\cal C}(\theta^*)$, 
and $d_{\mathcal{X}}(x)\triangleq\min_{s\in \mathcal{X}} \|x-s\|$
%denotes
is the distance %function
of $x$ to a %given closed
convex set $\cal X$.}

This work {draws inspiration from misspecified game-theoretic problems~\cite{Bischi08,Szidarovszky04} and builds on} {our} prior work that develops \emph{coupled} stochastic approximation schemes %with rate statements
for stochastic optimization/variational inequality problems~\cite{jiang13solution,jiang16variational} and
Nash games~\cite{jiang16distributed}. Subsequently, these
statements were refined for deterministic optimization
problems~\cite{ahmadi14data} and extended to
misspecified Markov Decision Processes (MDPs)~\cite{jiang15data}. {This
	paper is inspired by the challenges arising from
	misspecification in the \sa{\it constraint set}. Such concerns are addressed
		in~\cite{ahmadi15misspecified,jiang13solution,jiang16distributed}
by considering the associated variational inequality problem, {while here}
we consider a more general form of
	%misspecification in the constraints
misspecified constraints (namely convex conic).}
{Rather than standard gradient-type approaches}, we develop an
%misspecified analog of an
augmented Lagrangian~(AL) scheme for misspecified
convex problems {where} both the objective and constraints are
misspecified, {and} the subproblems are solved with increasing exactness. {Such
models have particular relevance in the context of iterative learning control
(ILC), which has its roots in the research by Uchiyama~\cite{Uchiyama78} and Arimoto et
al.~\cite{Arimoto84}}.

Throughout, our focus will be on the problem ${\cal C}(\theta^*)$ %when 
with ${\cal H}(\theta^*) \triangleq \{ x: h(x;\theta^*) {\preceq_{\Kscr} {\mathbf{0}}}
\}$, {where $\preceq_{\Kscr}$ denotes the partial order induced by the
	%proper 
\sa{closed convex cone} $\cK$ in $\reals^m$, i.e., given $a,b\in\reals^m$,
		   $a\preceq_{\Kscr} b$ implies $b-a\in\cK$. Hence, ${\cal
			   H}(\theta^*)\equiv\{x:\ -h(x;\theta^*)\in \Kscr\}$.}
Consequently, we may redefine the problem in parametric form, ${\cal C}(\theta)$, %\vspace*{-2mm} %more explicitly as follows:
%and
%\sout{$f(x;\theta^*) = p(x;\theta^*) + q(x;\theta^*), $ as given by}
{
\begin{equation}
\label{eq:parametric-opt-problem}
{\cal C}(\theta):\quad
\min_x\{f(x;\theta):\quad h(x;\theta)\preceq_{\Kscr} 0,\quad  x \in X\}, \vspace*{-1mm}
\end{equation}
}%
where $f:\R^n\times \Theta \to \R\cup\{+\infty\}$,
%\sout{$q:{\cal D}_q\times\Theta\to \R$},
and $h:\R^n\times \Theta\to \R^m$ {such that $f(\cdot;\theta)$ is convex and $h(\cdot;\theta)$ is affine for all $\theta\in\Theta$}, {$\cal K$ is a %proper 
\sa{closed convex cone} in {$\R^m$},
	%i.e., closed, convex, pointed, with a nonempty interior}, 
and $\theta \in \Theta \subseteq \R^d$ denotes the misspecified parameter.
Throughout, we assume that %$X\subset\R^n$ is a compact convex set %\sout{and ${\cal D}_q$ are subset of $\R^n$.}
% \vs{an}\sout{for purposes of well-posedness}
${\cal C}(\theta^*)$ has
a finite optimal value, denoted by $f^*$, {i.e., $f^*\triangleq f(x^*;\theta^*)$ for all $x^*\in X^*(\theta^*)$. Moreover, we also assume that} the corresponding Lagrangian dual problem has a solution, denoted by $\lambda^*$, and there is {\it no} duality gap. Before proceeding further, we recall that %Augmented Lagrangian
AL~schemes are rooted in the seminal work by Hestenes~\cite{hestenes1969multiplier}
and Powell~\cite{Powell69_1J}, and their relation to proximal-point methods was
established by Rockafellar~\cite{Rockafellar73_1J,rockafellar1973dual}.
{Based on the increasing role} of constrained optimization models in
control {systems and
engineering (cf.~\cite{bouyarmane18weight,freundlich18active,gu18explicit}),} there
has been a renewed examination of such techniques, particularly in convex
regimes, with an emphasis on deriving rate estimates,
e.g.,\cite{aybat2013augmented,lan15_1J,1506.05320 }.

%Next, we %outline
%state our main contributions:
{\bf Contributions:}
{\textbf{(i)} We provide a verifiable sufficient condition for the solution
set of AL subproblems to be upper-Lipschitz; this is of importance {in}
deriving rate statements. \textbf{(ii)} We proposed a parametric AL
method~(IPALM) and {derive} its convergence rate for constant and increasing
penalty parameter {sequences} -- these rates are {derived} for the first time; our
results subsume the {classical} results on AL method and extend them to the
parametric setting introduced in this paper. For the constant penalty case,
at most ${\cal O}(\epsilon^{-1})$ and ${\cal O}(\epsilon^{-4})$
proximal-gradient steps are required to obtain an $\epsilon$-feasible and
$\epsilon$-optimal solution without and with learning, respectively; while for
the increasing penalty case, we show the worst-case complexity reduces to
${\cal O}(\epsilon^{-1}\log(\epsilon^{-1}))$ regardless of whether a parallel
learning process is employed or not.} After having independently proven
iteration complexity statements for a constant penalty AL scheme in our
preliminary work~\cite{ahmadi16_ACC}, we became aware of related
work~\cite{1506.05320} considering inexact augmented Lagrangian schemes {\em
without} learning. {When $\theta^*$ is available, the complexity statements
provided in this manuscript and%
				 %in our preliminary work
~\cite{ahmadi16_ACC} are related to those %the statements provided
in~\cite{1506.05320}; though, our results are both novel and distinct from~\cite{1506.05320}.}

{{\bf Application:} {Before {providing} technical details, {we motivate} our problem setup with {an} example arising in multi-agent consensus optimization setting~\cite{aybat17_1J}.} Suppose $\cG=(\cN,\cE)$ denotes an undirected connected network of agents where $\cN$ denotes the set of agents with limited processing capability, and $\cE\subset\cN\times\cN$ %denotes the edges representing connectivity among agents, i.e.,
such that agents $i,j\in\cN$ can directly communicate if and only if
$(i,j)\in\cE$. {The objective is to minimize the sum of %locally known
(agent-specific) convex functions,} $\min_{x\in X}\sum_{i\in\cN}f_i(x)$ where
$X\subset\reals^n$ is a convex compact set, and each $f_i$ is a smooth convex
function \sa{that is %known 
only available to agent $i\in\cN$}. In this setting the agents are
collaborative, i.e., willing to cooperate to reach a consensus decision $x^*$;
but without sharing their private data, e.g., when $f_i(x)=\norm{A_ix-b_i}^2$,
{agents do not want to communicate $A_i$ or $b_i$ defining $f_i$ %due to
%The reason behind could be privacy restrictions or high communication overhead
-- collecting large-scale $\{A_i\}_{i\in\cN}$ at a %central node,
fusion center can be very costly and/or may violate privacy.} One way to deal with this problem is to reformulate it %through
using a local variable %vector
$x_i\in\reals^n$ for each $i\in\cN$, and impose consensus among these local variables, i.e., $x_i=x_j$ for all $(i,j)\in\cE$:
\begin{equation}
\label{eq:consensus_formulation}
\cC(W):\ \min_{\substack{\bx=[x_i]_{i\in\cN}\\ x_i\in X\ \forall i\in\cN}}\{\sum_{i\in\cN} f_i(x_i):\ (W\otimes I_n)\bx=\mathbf{0}\}.
\end{equation}
The communication matrix $W$ is a design parameter. %To design algorithms that
For algorithmic computations to rely only on local communications over $\cE$, one should choose $W\in\reals^{|\cN|\times|\cN|}$ such that $W_{ij}=0$ if $(i,j)\not\in\cE$. Moreover, to impose consensus one also needs the null space of $W$ %is one dimensional and
to be spanned by $\mathbf{1}\in\reals^{|\cN|}$ vector of ones. Thus, $(W\otimes I_n)\bar{\bx}=\mathbf{0}$ implies that $\bar{\bx}=[\bar{x}_i]_{i\in\cN}$ satisfies $\bar{x}_i=\bar{x}$ for all $i\in\cN$ for some $\bar{x}\in\reals^n$. Choosing
$W_{ij}<0$ for $(i,j)\in\cE$ and $W_{ii}=-\sum_{j\in\cN_i}W_{ij}$ satisfies this requirement, where $\cN_i=\{j\in\cN:\ (i,j)\in\cE\ \hbox{or}\ (j,i)\in\cE\}$ denotes the neighbors of $i\in\cN$ on $\cG$. Note that one trivial choice for $W$ is the graph Laplacian, i.e., $\cL_{ii}=d_i$ is the degree of node $i\in\cN$, $\cL_{ij}=-1$ if $(i,j)\in\cE$, and $\cL_{ij}=0$ if $(i,j)\not\in\cE$. However, this choice is not necessarily the best one for quickly reaching consensus, and a better one may be learnt. Indeed, a better choice for $W$ on a {barbell} graph is to use a communication matrix $W^*$ defined by effective resistances~\cite{aybat17_globalsip}, and the distributed method in~\cite{aybat17_globalsip} generates $\{W_k\}$ linearly converging to $W^*$.}

%The outline of the paper is as follows:
{\bf Outline:} (i) After proving some technical results in Section~\ref{sec:prelim}, we {derive}
some rate statements in Section~\ref{sec:rate} for dual suboptimality, primal infeasibility and %primal
suboptimality for the prescribed coupled first-order scheme %with a quantification of
{and quantify} the impact of misspecification. %In particular,
In Section~\ref{sec:rate-constant} we consider a setting with a constant
penalty parameter, and in Section~\ref{sec:rate-increasing} we derive
analogous rate statements in a setting where the penalty parameter is
increased after each dual update. %iteration. 
\noindent (ii) An overall iteration complexity analysis of the scheme is provided in Section~\ref{sec:complexity}.
\begin{comment}
In Section~\ref{sec:complexity-constant} we consider the constant
penalty case, and in Section~\ref{sec:complexity-increasing} we analyze the method using a suitably defined %sequence of
increasing penalty %parameters
sequence.
\end{comment}
%and prove that at most ${\cal O}(\epsilon^{-1})$ and ${\cal O}(\epsilon^{-4})$
%proximal-gradient steps are required to obtain an
%$\epsilon$-feasible and $\epsilon$-optimal solution without and
%with learning, respectively.
%Utilizing a suitably defined %sequence of
%increasing penalty %parameters
%sequence, in Section~\ref{sec:complexity-increasing}, we note that this
%worst-case complexity reduces to ${\cal O}(\epsilon^{-1}\log(\epsilon^{-1}))$ regardless of
%whether a parallel learning process is employed or not.
%%It is worth emphasizing
%\comt{We should provide our results without learning as they are new, just recently shown in Necoara's paper.}
(iii) Finally, in Section~\ref{sec:numerical}, we demonstrate the utility of the prescribed
scheme through a portfolio optimization problem with a
\emph{misspecified} covariance matrix.
\begin{comment}
where several insights emerge: \emph{a}) The misspecified variants of the AL %augmented Lagrangian
schemes perform well in practice;
\emph{b}) the complexity bound
in the constant penalty case {does not appear to be tight}, leaving room for improvement;
and \emph{c})
%naive
sequential schemes may perform %quite
poorly compared to the %their
proposed schemes.
\end{comment}

\noindent{\bf Notation:} Let $\norm{.}$ denote the Euclidean norm, {and $\cB(\mathbf{0},1) = \{z : \|z\| \leq 1\}$}. Given a closed convex set $\cK\subset\reals^m$ and $y\in\reals^m$, %define
{$d_{\mathcal{K}}(y)\triangleq\min_{s\in \mathcal{K}} \|y-s\|$,
$d^2_{\mathcal{K}}(y)\triangleq\left(d_{\mathcal{K}}(y)\right)^2$, and %We denote the Euclidean projection
 $\Pi_{\mathcal{K}}(y)\triangleq\argmin_{s\in \cK} \|y-s\|$; hence,
	$d_{\mathcal{K}}(y)=\norm{y-\Pi_{\mathcal{K}}(y)}$. Moreover, it is
		easy to verify that $d^2_{\mathcal{K}}(.)$ is differentiable and
		its gradient $\grad
		d^2_{\mathcal{K}}(y)=2(y-\Pi_{\mathcal{K}}(y))$.} Let
		$\cB(\bar{y},r)\triangleq\{y:\ \|y-\bar{y}\|\leq r\}$. Given a
		cone $\cK\in\reals^m$, let $\cK^*$ denote its dual cone, i.e.,
		$\cK^*=\{y'\in\reals^m:\ \fprod{y',~y}\geq 0,\
			   \forall~y\in\cK\}$. {Given $A\in\reals^{m\times n}$, $\sigma_{\max}(A)$ denotes the largest singular value of
		$A$, and $\norm{A}\triangleq\sigma_{\max}(A)$ denotes the spectral norm of $A$.} For $x,y\in\reals^n$, $\fprod{x,y}\triangleq x^\top y$.
\vspace*{-3mm}

\begin{table*}[t]
{\scriptsize
\vspace{1mm}
\begin{center}
\begin{tabular}{p{1.5in}|p{1.5in}|p{1.5in}|p{1.5in}} \hline
   Formulation  &  Motivation and applications & Tractability &  Challenges  \\ \hline
   \vspace*{0.25mm}
   {\bf Robust Optimization:}
  \begin{flalign*}
  \min_{x \in X}\{\max_{\theta \in \mathcal{T}} f(x;\theta): x\in \cH(\theta), \forall \theta\in\cT\},\\
  \cH(\theta)\triangleq\{x:\ h(x;\theta)\preceq_{\mathcal{K}}\mathbf{0}\}
  \end{flalign*}
  &
  \vspace*{0.25mm}
  (i) $\theta$ is {a random variable} with no distributional information; or (ii) $\theta$ has {an unknown} nominal value.
  \vspace*{1mm}

  In both instances, $\theta$ is assumed to lie in $\mathcal{T}$, a {user-defined} uncertainty set.
 \vspace*{0.25mm}

%\textbf{Applications:} Financial planning, production planning, {microgrid control}~\cite{burger14polyhedral}.

  &
  \vspace*{0.25mm}
  For a general convex uncertainty set $\mathcal{T}$, robust optimization is \emph{intractable}. For some {special structured sets} $\cT$, {this problem} is equivalent to a finite-dimensional convex program.
  &
  \vspace*{0.25mm}
  (i) Conservative solutions; (ii) Needs user-specified uncertainty set $\mathcal{T}$; (iii) Tractability relies on the problem structure -- %introducing a requirement on $x$ such that
  \sa{for general closed convex cone $\mathcal{K}$, %is a general closed and convex cone, 
  imposing $h(x;\theta)\preceq_{\mathcal{K}}\mathbf{0}$ for all
$\theta\in\mathcal{T}$ is \emph{not tractable}~\cite{Bertsimas2006}}.
  \\ \hline
  \vspace*{0.25mm}
 {\bf Stochastic Optimization:}
 \begin{align*}
    \min_{x\in X}\{\mathbb{E}_{\theta}[f(x;\theta)]: \mathbb{E}_{\theta}[h(x;\theta)]\preceq_{\mathcal{K}}\mathbf{0}\}\\
\mbox{{or}} \hspace{1in} \\
    \min_{x\in X}\{\mathbb{E}_{\theta}[f(x;\theta)]: \mathbb{P}_{\theta}\big(\cS(x)\big)\geq 1-\tau\}\\
 S(x)   \triangleq\{\theta:\ h(x;\theta)\preceq_{\mathcal{K}}\mathbf{0}\}.
 \end{align*}
 &
 \vspace*{0.25mm}
 $\theta$ is a random variable: (i) either the distribution of $\theta$ is available or (ii) an oracle exists for generating gradients $\nabla_x f(x;\theta)$ and $\nabla_x h(x;\theta)$.

%\textbf{Applications:} Stochastic production planning, stochastic resource allocation problems, stochastic traffic control~\cite{jie18stochastic}, sequential stochastic optimization for tracking~\cite{wilson18adaptive}.

 &
 \vspace*{0.25mm}
 Suppose an unbiased \emph{oracle} is given that produces $\nabla_x f(x;\theta)$ and $\nabla_x h(x;\theta)$ when $X$ is convex, $f(.;\theta)$ and $h(.;\theta)$ are convex differentiable \sa{for any $\theta$}. Then an $\epsilon$-solution can be computed in {$\mathcal{O}(1/\epsilon^2)$} projected gradient type steps {%for problem 
  (under some conditions)~\cite{jiang16variational}. The chance constrained problem is, in general, intractable.}
 &
 \vspace*{0.25mm}
 {Solving this problem} necessitates {either (i) the availability of tractable expectations (integrals) (and their derivatives) defined over a distribution $\mathcal{D}$}; or (ii) an oracle that produces $\nabla_x f(x;\theta)$ and $\nabla_x h(x;\theta)$ with suitable bias and moment assumptions ({facilitating the use of stochastic approximation}. \\
\hline
\vspace*{0.25mm}
{\bf Learning-facilitated Opt.:}
 \begin{align*}
 \cC(\theta^*):\quad \min_{x \in X}\{ f(x;\theta^*):\ x\in\cH(\theta^*)\}\\
 \theta^* \in \argmin_{\theta \in \Theta} \ell(\theta).
\end{align*}
\vspace*{-2mm}
&
\vspace*{0.25mm}
$\theta$ has a true nominal value $\theta^*$ that corresponds to a solution of a learning problem (such as a regression problem). This requires an apriori collection of data.\vspace*{-2mm}

%{\bf Applications.} Misspecified production planning; portfolio optimization; multi-agent consensus optimization
&
\vspace*{0.25mm}
Under differentiability and convexity assumptions, \sa{an $\epsilon$-solution to $\mathcal{C}(\theta^*)$ can be computed via first-order schemes in $\mathcal{O}(1/\epsilon)$ projected gradient type steps~\cite{nesterov2013introductory}}.\vspace*{-2mm}
&
\vspace*{0.25mm}
Need the availability of data to construct a suitable learning problem $\ell(\theta)$. \\ \hline
\end{tabular}
\vspace*{0.5mm}
\caption{\sa{Three related approaches}}
\label{tab:difference}
\end{center}
}
\vspace*{-7mm}
\end{table*}
%----------------------------------------
\section{Preliminaries}\label{sec:prelim}
%----------------------------------------
{Given $\theta\in\Theta$,} the problem ${\cal C}(\theta)$ is equivalent to the following reformulation: %reformulated problem:}
\vspace*{-1mm}
\begin{equation}
\min_{x,z} \left\{f(x;\theta):%+{\rho\over 2}\|h(x;\theta))+z\|^2:
h(x;\theta)+z={\mathbf{0}}, \quad x\in X,\quad z\in \Kscr\right\}.\label{formul:1}
\vspace*{-1mm}
\end{equation}
Let $\lambda\in \R^m$ denote a dual variable corresponding to the
	equality constraints in \eqref{formul:1}. For any given $\rho>0$,
			 {we define} the augmented Lagrangian function for
				 \eqref{formul:1} as
%${\cal L}_{\rho}(x,\lambda;\theta)$ {and define it as}
{\small
\begin{align*}
{\cal L}_{\rho}(x,\lambda;\theta)\triangleq \min_{z \in \Kscr}
\left[f(x;\theta)+\lambda^\top (h(x;\theta)+z)+\tfrac{\rho}{2}\|h(x;\theta)+z\|^2\right],
\end{align*}}%
where $\dom\,{\cal L}_\rho=X\times\R^m\times\Theta$. %Through a
%By rearranging the terms,
From the definition of $d_{\cK}(\cdot)$, %we get
\vspace*{-2mm} %it can be shown that
\begin{align*}
{\cal L}_\rho(x,\lambda;\theta)
%& =f(x;\theta)+\tfrac{\rho}{2}\min_{z\in \Kscr} \left\|h(x;\theta)+z+{\lambda \over \rho}\right\|^2-{\|\lambda\|^2\over 2\rho} \notag \\
		%& =f(x;\theta)+\tfrac{\rho}{2}\min_{z\in \Kscr} \left\|-\left(h(x;\theta)+{\lambda \over \rho}\right)-z\right\|^2-{\|\lambda\|^2\over 2\rho} \notag \\
=f(x;\theta)+\tfrac{\rho}{2}
d^2_{\Kscr}\big(-\big(h(x;\theta)+\tfrac{\lambda}{\rho}\big)\big) -\tfrac{\|\lambda\|^2}{2\rho}.%\label{Aug_L_forrho1}	
\vspace*{-2mm}
\end{align*}
{Note that {for any $\bar{y}\in\reals^m$}, we have
\begin{align}\label{proj-prop} \Pi_{\Kscr}(-\bar y) = %\argmin_{y \in \Kscr} \|y + \bar y\| =
-\argmin_{\hat y \in -\Kscr} \|\hat y - \bar y\| = -\Pi_{-\Kscr}(\bar y);
\end{align}}%
and consequently, by invoking \eqref{proj-prop}, we also have that
{\small
\begin{align}\label{proj-prop2}
{d_{\Kscr}(-\bar{y})
%= \min_{y \in \Kscr} \|y + \bar x \| = \min_{y \in \Kscr} \|-y - \bar x \|
= \|-\Pi_{\Kscr}(-\bar{y})-\bar{y}\|=\|\Pi_{-\Kscr}(\bar{y})-\bar{y}\|=d_{-\Kscr}(\bar{y}).}
\end{align}}%
{It follows from \eqref{proj-prop2} that  ${\cal L}_\rho(x,\lambda;\theta)$} %, given by \eqref{Aug_L_forrho1},
can be rewritten as
\begin{align}
{\cal L}_\rho(x,\lambda;\theta) =f(x;\theta)+\tfrac{\rho}{2} d^2_{-\Kscr}\big(h(x;\theta)+\tfrac{\lambda}{\rho}\big) -{\|\lambda\|^2\over 2\rho}.\label{Aug_L_forrho}	
\end{align}
Next, we derive the gradient $\nabla_{\lambda} {\cal L}_{\rho} (x,\lambda;\theta)$. %{To simplify notation, let $\psi(\lambda)\triangleq \tfrac{\lambda}{\rho}+h(x;\theta)$. Then,
{Since $\grad d^2_{\mathcal{K}}(x)=2(x-\Pi_{\mathcal{K}}(x))$ for any $x\in\reals^m$}, $\nabla_{\lambda} {\cal L}_{\rho}$ can be computed as %follows:
{\small
\begin{align}
\nabla_{\lambda} {\cal L}_{\rho}(x,\lambda;\theta)
%& = \frac{\rho}{2} . {2\over \rho} \left(h(x;\theta)+\tfrac{\lambda}{\rho}-\Pi_{-\Kscr} \left ( h(x;\theta)+\tfrac{\lambda}{\rho}\right)     \right) -\frac{\lambda}{\rho} \notag\\
= h(x;\theta)-\Pi_{-\Kscr}\Big({\lambda \over \rho}
		+h(x;\theta)\Big)\notag \\
=\Pi_{\Kscr^*}\Big({\lambda \over \rho} +h(x;\theta)\Big) -{\lambda\over \rho},\label{Grad_l}
\end{align}}%
where in the last  equality, we use the property that $\bar{x}=\Pi_{-\Kscr} (\bar{x}) + \Pi_{\Kscr^*}(\bar{x})$ {for all $\bar{x}\in\reals^m$}.  {Clearly, ${\cal L}_0(x,\lambda;\theta)$, i.e., ${\cal L}_\rho(x,\lambda;\theta)|_{\rho=0}$ is
the Lagrangian function:} %defined as
{\small
\begin{align}
\label{eq:lagrangian}
 {\cal L}_0(x,\lambda;\theta)\triangleq \begin{cases}
 f(x;\theta)+\lambda^\top h(x;\theta), & \mbox{ if } \lambda\in \Kscr^* \\
		  							  -\infty, & \mbox{otherwise.}
						\end{cases} 	
\end{align}}%
Moreover, let
\begin{align}
\label{lag-prob}
X^*(\lambda;\theta)\triangleq\argmin_{x \in X} {\cal L}_0(x,\lambda;\theta).
\end{align}
For $\rho\geq 0$, the \emph{augmented} dual problem %corresponding to
of ${\cal C}(\theta)$
is defined as %follows:
\begin{align*}
\left(D_{\rho}\right):\qquad
\max_{\lambda \in \Kscr^*}\quad \Big\{g_\rho(\lambda;\theta)\triangleq\inf_{x\in X} {\cal L}_\rho(x,\lambda;\theta)\Big\}.
\end{align*}
%{$X^*(\lambda;\theta)$ denotes the solution set of the Lagrangian problem:}
%where %$g_{\rho}$ is defined as follows:%
%$g_\rho(\lambda;\theta)\triangleq\inf_{x\in X} {\cal L}_\rho(x,\lambda;\theta).$
%The following results
Clearly, $\dom g_\rho \subseteq \Kscr^*$ for $\rho\geq 0$. Let $\Lambda^*\triangleq\argmax_{\lambda\in\cK^*} g_0(\lambda;\theta^*)$.
%Throughout, we make the following %additional
%assumptions.
\begin{assumption}\label{Assump: Aug_1}
Consider problems $\cE$ in~\eqref{eq:learning-problem} and $\cC(\theta)$ in \eqref{eq:parametric-opt-problem}. Suppose $X \subseteq\Real^n$ and $\Theta$ are convex compact sets.
\begin{enumerate}
\item[i.] Let {$h(\cdot;\theta)$ be an affine map
%in $x$
for every $\theta \in \Theta$, i.e.,
	$h(x;\theta)=A(\theta)x+b(\theta)$ for some $A(\theta)\in\R^{m\times
		n}$ and $b(\theta)\in\R^m$. Suppose $A(\cdot)$ and $b(\cdot)$
		are {globally Lipschitz continuous at $\theta^*$}, i.e., there exist
			constants $L_A, L_b$ such that for all $\theta\in \Theta$,
{{$\|A(\theta)-A(\theta^*)\|  \leq L_A \|\theta-\theta^*\|$ and $\|b(\theta)-b(\theta^*)\| \leq L_b \|\theta-\theta^*\|.$
}}}%
\item[ii.] {Let %In addition,
%The function %\sout{$p$, $q$}
$f(\cdot;\theta)$ be convex on $X$ {for all $\theta\in\Theta$} and $f(x;\cdot)$ be globally Lipschitz continuous {at $\theta^*$} uniformly for all $x\in X$ with constant {$L_{f,\theta}$}; i.e., ${|f(x;\theta)-f(x;\theta^*)|}\leq {L_{f,\theta}}\|\theta-\theta^*\|$ for all $\theta\in \Theta$ and $x \in X$.}
\item[iii.] For each $\lambda\in\cK^*$, {$X^*(\lambda;\cdot)$} is {globally upper-Lipschitz
at $\theta^*$} with some constant {$\kappa(\lambda)>0$}, i.e.,
%there exists a constant $\kappa^*$ such that
%for any $\lambda\in\cK^*$, %the following inclusion holds
for all $\theta\in\Theta$,
{\small
\begin{align}
\label{eq:main-assumption}
X^*(\lambda;\theta)\subseteq X^*(\lambda;\theta^*)+{\kappa(\lambda)} \|\theta-\theta^*\|~{\cB(\mathbf{0},1)}.
%,\ \forall~\theta\in\Theta.
\end{align}}%
{Moreover, $\kappa:\cK^*\rightarrow\reals$ is continuous.}\qed
\end{enumerate}
\end{assumption}
\begin{remark}
\label{rem:Lipschitz-h}
{Given $x\in X$, by Assumption~\ref{Assump: Aug_1}.i,
\begin{align*}
\|h(x;\theta)-h(x;\theta^*)\|&\leq
	\|A(\theta)-A(\theta^*)\|\|x\|+\|b(\theta)-b(\theta^*)\|\\
	& \leq (L_A\|x\|+L_b)\|\theta-\theta^*\|,\quad \forall~\theta\in\Theta.
\end{align*}
%for all $\theta\in \Theta$. 
Since $X$ is assumed to be compact, there exists a finite
	constant $D_x$ such that $D_x \triangleq\max_{x\in X} \|x\|$.
Hence, $h(x;\cdot)$ is {globally Lipschitz continuous at $\theta^*$} with uniform constant
	$L_{h,\theta}$ for all $x\in X$, where $L_{h,\theta}\triangleq (L_A D_x+L_b)$.}
{Moreover, given $\theta\in\Theta$,
	$$\norm{h(x;\theta)-h(x';\theta)}\leq\norm{A(\theta)}\norm{x-x'},\quad \forall~x,x'\in X.$$ 
%for all $x,x'\in X$. 
Since
		$\norm{A(\theta)}=\sigma_{\max}(A(\theta))$ is continuous in
		$\theta$ and $\Theta$ is compact,
	$L_{h,x}\triangleq\max_{\theta\in\Theta}\norm{A(\theta)}$ exists;
	therefore, $h(\cdot;\theta)$ is Lipschitz continuous with a uniform constant $L_{h,x}$ for all $\theta \in \Theta$.}
{Clearly, for all $\theta \in \Theta$, $h(\cB(\mathbf{0},1);\theta) \subseteq \cB\big(b(\theta), \norm{A(\theta)}\big)$; hence, $h(\cB(\mathbf{0},1);\theta) \subseteq b(\theta)+L_{h,x} \cB(\mathbf{0}, 1)$.}\qed

{The non-Lipschitzian regime is significantly more challenging. {When gradient maps are either not Lipschitz
continuous or such a constant is not immediately
available, one approach has been to utilize either convolution-based smoothing techniques~\cite{yousefian12stochastic,yousefian16selftuned} or deterministic techniques such as Moreau smoothing~\cite{beck2017first} in which a diminishing smoothing parameter (decaying to zero at a suitable rate) is employed~\cite{yousefian17smoothing,jalilzadeh2018vssa}. Using these techniques, one may be able to recover rate statements for misspecified problems, a focus of future work.}}
\end{remark}
Rather than focusing on the nature of the algorithm employed for resolving the
learning problem, {we assume} that the adopted scheme
produces a sequence $\{\theta_k\}$ that converges to the optimal solution {$\theta^*$}
at a non-asymptotic {\it linear} rate.
%(Assumption~\ref{Assump: learn_lin_rate}).
\begin{assumption}
\label{Assump: learn_lin_rate}
There exists a learning scheme that generates %a sequence
$\{\theta_k\}$
such that $\theta_k \to \theta^*$ at a linear rate as $k\to \infty$, i.e., there
exists %a constant
$\tau\in(0,1)$ such that $\|\theta_k-\theta^*\|\leq \tau^k\|\theta_0-\theta^*\|$ for $k \geq 0$ and $\theta_0
\in \Theta$. In addition, at iteration $k$ of the optimization {method for ${\cal C}$}, only $\theta_1,\dots,\theta_k$ are revealed.\qed
\end{assumption}

\sa{There are a host of schemes that produce sequences that converge to a solution of a strongly convex (possibly linearly constrained) learning problem  at a linear rate, including projected gradient schemes, accelerated gradient methods, etc.~\cite{beck2017first}. %It is worth noting that
In case the learning problem is merely convex (not strictly convex), if it satisfies certain growth conditions around the set of optimal solutions, then there are efficient methods that can also compute an optimal solution with linear convergence rate~\cite{drusvyatskiy2018error}. {In fact, linear rates can also be proven in nonconvex regimes under the Polyak-Lojasiewicz (PL) inequality~\cite{karimi17linear}.}
In some other cases, $\theta^*$ can be computed by solving a large-scale linear system, and Kaczmarz method can produce a linearly convergent sequence as well~\cite{aybat17_globalsip}.}%
\begin{remark}
\sa{For $\theta\in\Theta$, let $X^*(\theta)\triangleq\argmin_{x\in X}\{f(x;\theta):\ A(\theta)x+b(\theta)\in -\cK\}$ and $f^*(\theta)\triangleq f(x^*(\theta);\theta)$ for $x^*(\theta)\in X^*(\theta)$. Assuming $X$ and $\Theta$ {are} compact convex sets, $\cK$ {is} closed, and $f(\cdot;\cdot)$ is continuous on $X\times \Theta$, one can show that $f^*(\cdot)$ is continuous and $X^*(\cdot)$ is upper-hemicontinuous at $\theta^*$ under weak assumptions on the feasible parameter space $\Theta$ -- see Appendix~\ref{sec:app-set_continuity}. Therefore, if the learning scheme produces a parameter estimate sequence converging to $\tilde{\theta}$ close to $\theta^*$, then {the set} $X^*(\tilde{\theta})$ {obtained} using our method will be ``close" to the true solution set $X^*(\theta^*)$.}
\end{remark}
Lemma~\ref{Lemma: Aug_Rock_Lemmas} {provides} various properties of the
gradient of the dual function $\nabla_{\lambda}g_\rho$ {and} will be used in
our analysis. Its proof {may} be found in~\cite{rockafellar1973dual} and is
omitted.
\begin{lemma}\label{Lemma: Aug_Rock_Lemmas}
Suppose {Assumptions}~\ref{Assump: Aug_1}.i and \ref{Assump: Aug_1}.ii hold. %Then, the following hold.
\begin{enumerate}
\item[i.] For any $\rho>0$ and $\theta \in \Theta$, the dual function $g_\rho (\lambda;\theta)$ is {\it everywhere} finite, continuously differentiable concave function over $\R^m$; more precisely, $g_\rho(\lambda;\theta)=\max_{w\in\R^m}\{g_0(w;\theta)-\tfrac{1}{2\rho}\|w-\lambda\|^2\}$, i.e., $g_\rho(\cdot,\theta)$ is the Moreau envelope of $g_0(\cdot,\theta)$ for all $\theta\in\Theta$. Therefore,
	$\nabla_\lambda g_\rho(\lambda;\theta)$ is Lipschitz continuous in
	$\lambda$ with constant $\tfrac{1}{\rho}$ for all $\theta$ in $\Theta$.
\item[ii.] For any given $\lambda\in \Kscr^*$ and $\theta\in \Theta$, $\nabla_\lambda g_\rho$ can be computed as $\nabla_\lambda g_\rho (\lambda;\theta)=\nabla_{\lambda} {\cal L}_\rho (x^*(\lambda;\theta),\lambda;\theta)$, where $x^*(\lambda;\theta)\in \argmin_{x\in X} {\cal L}_\rho
(x,\lambda;\theta)$.
\item[iii.] Given $\lambda\in \Kscr^*$, $\theta\in \Theta$,
	let $\tilde{x}(\lambda;\theta)$ be an {\it inexact} solution to $\min_{x\in X}{\cal L}_\rho(x,\lambda;\theta)$
	with accuracy $\alpha$, i.e.,  $\tilde{x}(\lambda;\theta)\in X$ satisfies
${\cal L}_\rho(\tilde{x}(\lambda;\theta),\lambda;\theta)\leq g_\rho(\lambda;\theta)+\alpha,$
then \vspace*{-2mm}
\begin{equation*}
\|\nabla_{\lambda} {\cal L}_\rho(\tilde{x}(\lambda;\theta),\lambda;\theta)-\nabla_{\lambda} g_\rho(\lambda;\theta)\|^2\leq {2\alpha/ \rho}.
\vspace*{-2mm}
\end{equation*}
\end{enumerate}
\end{lemma}
%Suppose $F_0$ and $F_{\rho}$ are defined as follows:
%\begin{align*}
% F_0(x,u;\theta)\triangleq \begin{cases} f(x;\theta), & \mbox{ if }  u_i\geq h_i(x;\theta),\quad i=1,\hdots,m,\\
%		  							  +\infty, & \mbox{otherwise}
%						\end{cases} 	
%\end{align*}
%and $F_\rho(x,\lambda,u;\theta)\triangleq F_0(x,u;\theta)+\lambda^Tu+{\rho\over 2}\|u\|^2$.
%\begin{IEEEproof}
%\begin{enumerate}
%\item[i)] The proof follows from ~{Rockafellar1973}
%\item[ii)] The proof follows from ~\cite[Lemma 4.3]{Rockafellar1973}.
%\item[iii)] The proof follows from ~\cite[Theorem 3.1]{Rockafellar1973}.
%\end{enumerate}
%\end{IEEEproof}
%\begin{proposition}
%\label{Theorem:Aug_pertuba}
%Given $\rho>0$, we have that
%$${\cal L}_\rho(x,\lambda;\theta)=\min_{u\in \R^m} F_\rho(x,\lambda,u;\theta),$$
%where the minimum is attained at $u^*=\nabla_{\lambda} {\cal L}_\rho(x,\lambda;\theta).$
%\end{proposition}
%\begin{IEEEproof}....
%\end{IEEEproof}
Next, we examine %the Lipschitzian property of
the Lipschitz continuity of $\nabla_{\lambda}
g_{\rho}(\lambda;\theta)$ in $\theta\in\Theta$. %Recall that
%$$ g_{\rho}(\lambda;\theta) = \max_{w} \ \left[g_0(w;\theta) - \frac{1}{2\rho} \|w-\lambda\|^2\right].$$
{By properties of the Moreau envelope~\cite{hiriart2001convex}, in particular from Lemma~\ref{Lemma: Aug_Rock_Lemmas}.i, it follows that}
\begin{align}
\label{def-nabla-g} \nabla_{\lambda} g_{\rho}(\lambda;\theta) =  \tfrac{1}{\rho}
({\pi_{\rho}(\lambda;\theta) -\lambda}), \end{align}
where {${\pi_{\rho}}(.;\theta)$, the Moreau map of $g_\rho(.;\theta)$},
	  is defined as
\begin{align}\label{def-pi}
 \pi_{\rho}(\lambda;\theta) \ \triangleq \ \argmax_w\ g_0(w;\theta) - \tfrac{1}{2\rho} \|w-\lambda\|^2.\end{align}
Therefore, it suffices to show the Lipschitz continuity of
${\pi_{\rho}}(\lambda;\cdot)$. %in $\theta$.
We begin with an
intermediate lemma that proves the Lipschitz continuity %of $\pi_{\rho}$
when the \emph{subdifferential} of $-g_0(\cdot;\theta)$ is {upper-Lipschitz} at $\theta^*$. Let
$\partial_w g_0(w;\theta)\triangleq\{s\in\reals^m:\
	g_0(w;\theta)+s^\top(w'-w)\geq g_0(w;\theta),\ \forall w'\in\cK^*\}$. %
%\us{The proofs of %the following technical results stated in Lemmas~\ref{lip-pi}--~\ref{lip-g0} can be found in the extended version of the manuscript~\cite{ahmadi16_preprint}.}
\begin{lemma}\label{lip-pi}
Suppose {for $\lambda\in\cK^*$, there exists $\kappa_d(\lambda)>0$ such that $g_0(\lambda;\cdot)$ satisfies the
following for all $\theta \in \Theta$: %$w$ and for all $\theta_1, \theta_2 \in \Theta$:
\begin{align}\label{lip-sub}
\partial_\lambda {g_0}(\lambda;\theta) \subseteq
\partial_\lambda {g_0}(\lambda;\theta^*) + \kappa_d(\lambda) \|\theta-\theta^*\|~\cB(\mathbf{0},1),
\end{align}}%
%where $\cB(\mathbf{0},1) = \{z : \|z\| \leq 1\}$.
\sa{Then for any %$\lambda\in\cK^*$ and
$\theta\in\Theta$,}
%the following holds  for all  $\theta_1, \theta_2 \in \Theta$:
\begin{equation}
\label{eq:lipschitz-MoreauMap}
{\|{\pi_{\rho}}(\lambda;{\theta})-{\pi_{\rho}}(\lambda;{\theta^*})\|\leq
{\rho} \kappa_d(\lambda) \|{\theta} -\theta^*\|.}
\end{equation}
\end{lemma}
%\begin{IEEEproof}
%{See the supplementary materials file.}
%\
%\end{IEEEproof}
%\begin{comment}

%\end{comment}
\sa{The condition %under which the above results hold are
 given in \eqref{lip-sub}} is not immediately useful in practice as %ideal in that
it is assumed on $g(\lambda;\theta)$. However, the next result
shows that %by assuming a suitable pseudo-Lipschitzian property on
{\eqref{lip-sub} \sa{indeed} holds when $X^*(\lambda;\theta)$ is upper-Lipschitz at $\theta^*$ -- see %uniformly in $\lambda$
Assumption~\ref{Assump: Aug_1}.iii}.
%we obtain the required property.
%where $X^*(\lambda;\theta)$ is defined as follows:
%$$ X^*(\lambda;\theta) = \mbox{arg}\min_{x \in X} {\cal L}_0(x,\lambda;\theta). $$
{
\begin{lemma}\label{lip-g0}
{For each $\lambda\in\cK^*$, under Assumption~\ref{Assump: Aug_1}, {$g_0$ satisfies \eqref{lip-sub}} for all $\theta\in\Theta$ with % constant $\partial_{\lambda} g_0(\lambda;\theta)$ is pseudo-Lipschitz in $\theta$ with a uniform constant
	$\kappa_d(\lambda) \triangleq L_{h,\theta} + \kappa(\lambda) L_{h,x}$.}
%for all $\lambda$, {i.e., for any $\lambda$, one has $\partial_\lambda {g_0}(\lambda;\theta_1) \subseteq
%\partial_\lambda {g_0}(\lambda;\theta_2) + \kappa \|\theta_1-\theta_2\|~\cB(\mathbf{0},1)$ for all $\theta_1,\theta_2\in\Theta$.}
\end{lemma}}
%\begin{IEEEproof}
%sa{See the supplementary materials file.}
%\end{IEEEproof}
%\begin{comment}
%\end{comment}
%We are %now
%ready to prove
%Our main
\sa{The proofs of Lemmas~\ref{lip-pi} and~\ref{lip-g0} are given in the appendices~\ref{sec:app-LipschitzI} and~\ref{sec:app-LipschitzII} of the online supplement, respectively.} 

{The Lipschitzian %requirements on
{properties of}
	$\pi(\lambda;\cdot)$ and $\nabla_{\lambda} g_{\rho}
	(\lambda;\cdot)$ stated below in Theorem~\ref{Lips_contin_grad_g} follow} from Lemmas~\ref{lip-pi} and \ref{lip-g0}.
\begin{theorem}
\label{Lips_contin_grad_g}
Suppose Assumption~\ref{Assump: Aug_1} holds and let {$\kappa_d(\lambda)
	\triangleq L_{h,\theta} +\kappa(\lambda) L_{h,x}.$} Then, we have
	the following:
\begin{enumerate}
\item[(i)]
$\|{\pi_{\rho}}(\lambda;{\theta})-{\pi_{\rho}}(\lambda;{\theta^*})\|\leq
\kappa_d(\lambda) {\rho} \|{\theta} -\theta^*\|,\ \forall~\theta\in \Theta$.
\item[(ii)]
 $\nabla_{\lambda}
g_\rho(\lambda;\cdot)$ is Lipschitz continuous at $\theta^*$ %over $\Theta$
with {constant $\kappa_d(\lambda)$} for all $\lambda\in \Kscr^*$.
\end{enumerate}
\end{theorem}
We provide a %verifiable
sufficient condition for Assumption~\ref{Assump: Aug_1}.iii to hold under %suitable
convexity and differentiability assumptions on $f$.
%More precisely, we show that for any $\lambda\in\Lambda$, the parametric solution set $X^*(\lambda;\cdot)$ is pseudo-Lipschitz continuous for convex and differentiable $f(\cdot;\theta)$. In particular, we make the following mild assumptions on $f$.
%\sa{In Section~\ref{sec:qp}, we focus on the case where $f(x,\theta)$ is quadratic in $x$ and provide the sufficient conditions for this case.}
%\subsection{}\label{sec:non}
\begin{theorem}\label{thm:pseudoLipschitz}
\sa{Under Assumption~\ref{Assump: Aug_1}.i}, consider $\cC(\theta)$ in~\eqref{eq:parametric-opt-problem} such that $f(\cdot;\theta)$ is {continuously} differentiable on an open set $\tilde{X}\supseteq X$ for all $\theta\in\Theta$; $\grad_x f(x;\cdot)$ is Lipschitz on $\Theta$ with constant $L_{F,\theta}$ for all $x\in X$; $f(\cdot;\theta^*)$ is twice continuously differentiable on $\tilde{X}$; $f(\cdot;\theta^*)$ is strictly convex on $X$. \sa{Then for any compact $\Lambda\subset\cK^*$, there exist $c>0$ such that for any $\lambda\in\Lambda$,
%such that the dual iterate sequence, {generated by IPALM,} satisfies $\{\lambda_k\}\subset\Lambda$; and
%there exists $\kappa(\lambda)>0$ such that
%Moreover, %$\kappa:\Lambda\rightarrow\reals$ is continuous \sa{on $\Lambda$}. More precisely,
\eqref{eq:main-assumption} holds for all $\theta\in\Theta$ with %$\kappa:\Lambda\rightarrow\reals$ such that
$\kappa(\lambda)\triangleq c~(L_{F,\theta}+L_A\norm{\lambda})$.}
\end{theorem}
\begin{IEEEproof}
\sa{The proof is provided in the Appendix~\ref{sec:app-LipschitzThm} of the online supplement.}
%online supplement. %-- also see {the extended version of the manuscript}~\cite{ahmadi16_preprint}.
\end{IEEEproof}
{For the consensus optimization application (discussed in introduction), if $\sum_{i\in\cN}f_i(x)$ is strongly convex and twice continuously differentiable, then Assumption~1.iii holds;
%e.g., $f_i(x)=\half\norm{A_ix-b_i}$ for $i\in\cN$ and vertical concatenation $[A_i]_{i\in\cN}$ has full-column rank;
for this scenario, \eqref{eq:consensus_formulation} can be regularized to be strongly convex without altering the optimal solution set -- see~\cite{aybat2017Allerton} for details.}
\section{%Convergence analysis for
Inexact Parametric Aug. Lagrangian Method}
\label{sec:rate}
%-----------------------------------------------
%\vspace*{-3mm}
\begin{algorithm}
{\
\caption{{:\ $\mathrm{IPALM}$} -- {inexact parametric augmented Lagrangian method}}
\hspace*{-2mm}\textbf{Input:} $\lambda_0\in\cK^*$, $x_0\in X$, $\{\rho_k\}$ and $\{\alpha_k\}$\\%
For all $k\geq 0$, update:
\label{Alg: ALM}
\begin{enumerate}
\item[1.]  {Given $\theta_k$,} compute {$x_{k+1}\in X$} such that\\
\hspace*{1cm}${\cal L}_{\rho_k}(x_{k+1},\lambda_k;\theta_{k})\leq g_{\rho_k}(\lambda_k;\theta_{k})+\alpha_k$;
\item [2.] $\lambda_{k+1}\gets\lambda_k+\rho_k\nabla_{\lambda} {\cal L}_{\rho_k}(x_{k+1},\lambda_k;\theta_{k})$;
\item [3.] $k\gets k+1$;
\end{enumerate}}
\end{algorithm}
\vspace*{-2mm}
{%In this section,
Next, we introduce the misspecified inexact
	augmented Lagrangian scheme, displayed in Algorithm~\ref{Alg: ALM}, {i.e., IPALM.}
%\begin{algorithm}
%\
%\caption{{Misspecified inexact aug. Lag. scheme with constant penalty parameter}}
%Given $\lambda_0=\mathbf{0}\in\R^m$, let $\{\alpha_k\}$,$\{\rho_k\}$ and $\{\theta_k\}$ be given sequences. Then for all $k\geq 0$,
%\label{Alg: Aug_const_rho}
%\begin{enumerate}
%\item[(1)]  Compute $x_k$ such that ${\cal L}_{\rho_k}(x_k,\lambda_k;\theta_{k})\leq g_{\rho_k}(\lambda_k;\theta_{k})+\alpha_k$;
%\item [(2)] $\lambda_{k+1}:=\lambda_k+\rho_k\nabla_{\lambda} {\cal L}_{\rho_k}(x_k,\lambda_k,\theta_{k})$;
%\item [(3)] $k:=k+1$;
%\end{enumerate}
%\end{algorithm}
From Step~2 of %Algorithm~\ref{Alg: ALM}
{IPALM}, and \eqref{Grad_l}, it follows that \vspace*{-1mm}
{
\begin{align}
\lambda_{k+1}
&=\lambda_k+\rho_k\nabla_{\lambda} {\cal
	L}_{\rho_k}(x_{k+1},\lambda_k;\theta_{k})\nonumber \\
%= \lambda_k +\rho_k\left( h(x_{k+1};\theta_k)-\Pi_{-\Kscr}\left({\lambda_k\over \rho_k}+h(x_{k+1};\theta_k)\right)\right) \notag \\
& = \lambda_k +\rho_k\Big( \Pi_{\Kscr^*}\Big({\lambda_k\over\rho_k}+h(x_{k+1};\theta_k)\Big)-{\lambda_k\over \rho_k}\Big) \notag\\
&= \Pi_{\Kscr^*}\left({\lambda_k}+\rho_kh(x_{k+1};\theta_k)\right),\quad \forall~k\geq 0.
\label{eq:lambda-relation}
\end{align}}%
Hence, $\{\lambda_k\}\subseteq \Kscr^*$} for all $k\geq 0$. Notably, if
$\theta_k = \theta^*$ for all $k\geq 0$, this reduces to the traditional
version considered in~\cite{rockafellar1973dual}. {In the remainder of this
	section, %comprises of two subsections, where
we analyze the {rate of
		convergence} of %Algorithm~\ref{Alg: ALM}
{IPALM} for a constant penalty parameter in Section~\ref{sec:rate-constant} and proceed to examine the increasing penalty parameter regime in Section~\ref{sec:rate-increasing}.}
\vspace*{-4mm}
\subsection{Convergence analysis for constant penalty %sequence
$\mathbf{\rho_k=\rho>0}$}
\label{sec:rate-constant}
In this section, we study the convergence behavior of %Algorithm~\ref{Alg: ALM}
{IPALM} when $\rho_k=\rho$ for $k\geq 0$.
We make the following assumption on the sequence of inexactness $\{\alpha_k\}$.
\begin{condition}
\label{Assump:Aug_2} %The inexactness
{\it $\{\alpha_k\}$ sequence satisfies
{$\sum_{k=0}^{\infty} \sqrt{\alpha_k} <\infty$}.}
\end{condition}
%\vspace*{-1mm}
Define $\bar{x}_k\in X$ and $\bar{\lambda}_k\in\cK^*$ as \vspace*{-2mm}
{%\
$$\bar{x}_k\triangleq {1\over k}{\sum_{i=1}^k x_i},\quad \hbox{\normalsize and}\quad \bar{\lambda}_k\triangleq {1\over k}{\sum_{i=1}^k \lambda_i}\quad \hbox{\normalsize for}\quad k\geq 1.\vspace*{-2mm}$$}%
Under {Condition}~\ref{Assump:Aug_2}, we show %the following:
\emph{(i)} $0\leq f^*-g_\rho(\bar{\lambda}_k)\leq\mathcal{O}(1/k)$;
 \emph{(ii)} $d_{-\Kscr}\big(h(\bar{x}_k;\theta^*)\big)\leq
 \mathcal{O}(1/\sqrt{k})$,\quad
\emph{(iii)} $-\mathcal{O}(1/\sqrt{k})\leq f(\bar{x}_k;\theta^*)-f^*\leq
		\mathcal{O}(1/k)$.
These statements are then utilized in deriving the overall
computational complexity in Section~\ref{sec:complexity-constant}.
After independently proving these bounds, we became aware of related recent work~\cite{1506.05320 }, where
    %the perfectly specified variant of
%Algorithm~\ref{Alg: ALM}
{a special case of {IPALM} is studied} with $\alpha_k=\alpha>0$ for all $k\geq 0$ for some fixed $\alpha>0$, and
			under the \emph{perfect information} assumption, i.e., $\theta_k=\theta^*$ for all $k\geq 0$.
In~\cite{1506.05320 }, it is shown that \emph{(i)}
	$f^*-g_\rho(\bar{\lambda}_k)\leq\mathcal{O}(1/k)+\alpha$,
		\emph{(ii)} $d_{-\Kscr}(h(\bar{x}_k;\theta^*))\leq
			\mathcal{O}(1/\sqrt{k})$, and \emph{(iii)}
	$-\mathcal{O}(1/\sqrt{k})\leq f(\bar{x}_k;\theta^*)-f^*\leq
		\mathcal{O}(1/k)+\alpha$. Therefore, according to~\cite{1506.05320 }, $\alpha$ should be fixed as a small
		constant as it appears in \emph{both} primal and dual
		suboptimality bounds. Moreover, since $\alpha$ is fixed
		in~\cite{1506.05320 }, suboptimality of the iterate sequence will
		\emph{stall} after certain number of iterations. In contrast, our
		method may  start with large $\alpha_0$ and gradually
		decrease it, ensuring \emph{both} numerical stability
		and asymptotic convergence to optimality in contrast
		with~\cite{1506.05320 } where the scheme provides approximate
		solutions at best.}
%-----------------------------------------------
%\subsection{Rate of convergence analysis}\label{sec:3.2}
%-----------------------------------------------

Next, we %{begin by showing}
show that dual variables stay bounded by using the following
supporting lemma whose proof follows from
%Lemma~\ref{Lemma: Aug_Rock_Lemmas} and
the properties of proximal maps (cf.~\cite{hiriart2001convex}). %\red{in~\cite{Rockafellar73_1J}}.
\begin{lemma}
\label{Lemma:Non_exp_pi}
Let $\pi_\rho(\lambda;\theta)$ be the proximal map of $g_0(\cdot;\theta)$ defined in \eqref{def-pi}. Then for all $\lambda_1 ,\lambda_2\in \Kscr^*$ and $\theta \in \Theta$,
$$\|\pi_\rho(\lambda_1;\theta)-\pi_\rho(\lambda_2;\theta)\|^2+ \|\pi^c_\rho(\lambda_1;\theta)-\pi^c_\rho(\lambda_2;\theta)\|^2 \leq \|\lambda_1-\lambda_2\|^2$$
where $\pi^c_\rho(\lambda;\theta)\triangleq \lambda- \pi_\rho(\lambda;\theta)$. Hence, $\pi_\rho$ is {\it nonexpansive} in $\lambda\in \Kscr^*$ for all $\theta\in \Theta$.
\end{lemma}
\begin{IEEEproof}
The proof is given in~\cite{citeulike:9472207}.\
\end{IEEEproof}
To show $\{\lambda_k\}$ stays bounded, {we mainly use the same proof technique
		in~\cite{Rockafellar73_1J,rockafellar1976augmented} together with the Lipschitz continuity result in
		%Proposition
Theorem~\ref{Lips_contin_grad_g}.}\begin{theorem}[\textbf{Boundedness of $\{\lambda_k\}$}]
\label{Theorem: Aug_lambda_bnd_cont_rho}
Let {Assumptions~\ref{Assump: Aug_1} and~\ref{Assump: learn_lin_rate}, and Condition~\ref{Assump:Aug_2} hold,} and $\lambda^*$ be an arbitrary solution to the Lagrangian dual of ${\cal C}(\theta^*)$, i.e., $\lambda^*\in\argmax_\lambda g_0(\lambda;\theta^*)$. Then for all $k\geq 1$, $\|\lambda_k-\lambda^*\|\leq C_\lambda$ and {$\bar{\kappa} \triangleq\sup_{k\geq 1}\kappa_d(\lambda_k)< \infty$}, where $C_\lambda$ is defined as follows:
\begin{equation}
C_\lambda \triangleq\sqrt{2\rho}\sum_{i=0}^\infty\sqrt{\alpha_i}+\rho {\kappa_d(\lambda^*)}{\|\theta_0-\theta^*\|\over 1-\tau}+\|\lambda_0-\lambda^*\|.\label{bd-C-lambda}
\end{equation} %+\|\lambda^*\|.$$
\end{theorem}
\begin{IEEEproof}
We begin by {bounding} %deriving a bound on
$\|\lambda_{k+1}-\pi_{\rho}(\lambda_k;\theta_k)\|$. From \eqref{def-nabla-g} and the definition of $\lambda_{k+1}$, given in Step~2 of %Algorithm~\ref{Alg: ALM}
{IPALM}, %it follows that
\begin{eqnarray}
\lefteqn{\|\lambda_{k+1}-\pi_\rho(\lambda_k;\theta_{k})\|}\label{proof:Aug_bnd_lambda_1}\\
%&=\|\lambda_{k}+\rho\nabla_{\lambda} {\cal L}_\rho(x_{k+1},\lambda_k;\theta_{k})-\lambda_k-\rho\nabla_\lambda g_{\rho}(\lambda_k;\theta_{k})\|\notag\\
& & = \rho\|\nabla_{\lambda} {\cal
	L}_\rho(x_{k+1},\lambda_k;\theta_{k})-\nabla_\lambda
	g_{\rho}(\lambda_k;\theta_{k})\| \leq \sqrt{2\rho\alpha_k}, \nonumber
\end{eqnarray}
where the last inequality follows from Lemma~\ref{Lemma: Aug_Rock_Lemmas}.iii. Since $g_\rho(\cdot;\theta^*)$ is the Moreau regularization of $g_0(\cdot;\theta^*)$, we have  that $\lambda^*\in \argmax_{\lambda} g_{\rho}(\lambda;\theta^*)$ for all $\rho>0$; hence, $\nabla_\lambda g_\rho(\lambda^*;\theta^*)=\mathbf{0}$ and $\lambda^*=\pi_{\rho}(\lambda^*;\theta^*)$. Thus,\vspace*{-2mm}
%From this observation, we obtain %the bound: %below:
%\begin{eqnarray}
\begin{align}
& \|\pi_\rho(\lambda_k;\theta_{k})-\lambda^*\| \nonumber \\
%&= \|\pi_\rho(\lambda_k;\theta_k)-\pi_{\rho}(\lambda^*;\theta^*)\|\notag\\
\leq & {\|\pi_\rho(\lambda_k;\theta_{k})-\pi_\rho(\lambda^*;\theta_k)\| +\|\pi_\rho(\lambda^*;\theta_k)-\pi_\rho(\lambda^*;\theta^*)\|}
\notag\\
\sa{\leq} & {\|\lambda_k-\lambda^*\|+\rho\|\nabla_{\lambda} g_{\rho}(\lambda^*;\theta_{k})-\nabla_{\lambda} g_{\rho}(\lambda^*;\theta^*)\|} \nonumber\\ %&+\|\pi_\rho(\lambda_k;\theta^*)-\pi_\rho(\lambda^*;\theta^*)\notag \|\\
\leq &\|\lambda_k-\lambda^*\|+\rho {\kappa_d(\lambda^*)}\|\theta_{k}-\theta^*\|,\label{proof:Aug_bnd_lambda_2}
\end{align}
%\end{eqnarray}
%where the first inequality follows from the triangle
%inequality, the second equality is a result of invoking the definition of
%$\pi_{\rho}(\lambda, \theta)$,   and the last inequality
{which follows \sa{from \eqref{def-nabla-g}}, the Lipschitz continuity of $\nabla_{\lambda} g_{\rho}(\lambda^*;\cdot)$ %in $\theta$ {(%Proposition
(Theorem~\ref{Lips_contin_grad_g}.ii), %uniformly in $\lambda$,
and the nonexpansivity of $\pi_\rho$ in
$\lambda$ (Lemma~\ref{Lemma:Non_exp_pi})}. Hence, %from~
\eqref{proof:Aug_bnd_lambda_1} and~\eqref{proof:Aug_bnd_lambda_2} imply %, we obtain for all $i\geq 0$ that
\begin{equation}
\label{eq:SM-convergence}
\begin{array}{c}
\|\lambda_{i+1}-\lambda^*\| %&\leq\|\lambda_{k+1}-\pi_\rho(\lambda_k,\theta_{k})\|+\|\pi_\rho(\lambda_{k},\theta_k)-\lambda^*\|\\
\leq \sqrt{2\rho\alpha_i}+\rho {\kappa_d(\lambda^*)}\|\theta_{i}-\theta^*\|+\|\lambda_i-\lambda^*\|. %\quad \forall\ i\geq 0.
\end{array}
\end{equation}
for all $i\geq 0$. Therefore, for any $k\geq 1$, by summing the above inequality over $i=0,\ldots,k-1$, %and using the fact that $\lambda_0=\mathbf{0}$,
 we get \vspace*{-2mm}
{\small
\begin{align}
\|\lambda_{k}-\lambda^*\|\leq &	\sum_{i=0}^{k-1}\left(\sqrt{2\rho\alpha_i} +\rho {\kappa_d(\lambda^*)}\|\theta_{i}-\theta^*\|\right)+\|\lambda_0-\lambda^*\|, \label{eq:aux-ineq-lambda-bound}
%&\leq \sqrt{2\rho}\sum_{i=0}^\infty\sqrt{\alpha_i}+\rho \kappa{\|\theta_0-\theta^*\|\over 1-\tau}+\|\lambda_0-\lambda^*\|. \nonumber
\end{align}}%
and \eqref{bd-C-lambda} immediately follows as $\sum_{i=0}^\infty\norm{\theta_i-\theta^*}=\norm{\theta_0-\theta^*}/(1-\tau)$. {Finally, $\bar{\kappa}<\infty$ follows from continuity of $\kappa_d$ (Assumption~1.iii) and boundedness of $\{\lambda_k\}$.}
%Finally, for all $k$, we have that the following holds:
%\begin{align} \notag
%& \quad \|\lambda_{k+1}\|\leq \|\lambda_{k+1}-\lambda^*\|+\|\lambda^*\|\\
%		\notag
%&\leq \sqrt{2\rho}\sum_{i=0}^\infty\sqrt{\alpha_i}+\rho M_{h,\theta}{\|\theta_0-\theta^*\|\over 1-\tau}+\|\lambda_0-\lambda^*\|+\|\lambda^*\|\\
%&\triangleq C_\lambda.
%\end{align}
\end{IEEEproof}

\begin{remark}
	It is worth emphasizing that the bound $C_{\lambda}$ can be tightened when $\theta^*$ is known. {Indeed, when $\theta_0 = \theta^*$}, the second term disappears.
\end{remark}
\begin{remark}
\sa{Under the premise of Theorem~\ref{thm:pseudoLipschitz}, Assumption~\ref{Assump: learn_lin_rate} and Condition~\ref{Assump:Aug_2} are sufficient to show Theorem~\ref{Theorem: Aug_lambda_bnd_cont_rho}. Indeed, for any $\theta\in\Theta$ and $\bar{x}\in X^*(\lambda^*;\theta)$, the first-order conditions imply
{\small
	\begin{equation*}
	%\label{eq:perturbed-sol}
	\begin{array}{rl}
	\bar{x}\in\argmin_{x\in X}\{&\L_0(x,\lambda^*;\theta^*)\\
	 &+(\grad_x \cL_0 (\bar{x},\lambda^*,\theta)-\grad_x \cL_0(\bar{x},\lambda^*,\theta^*))^\top x\}.
	\end{array}
	\end{equation*}}%
%Therefore, 
Lemma~5.1.7 in~\cite{facchinei2007finite} implies that
there exists $c^*>0$ such that
{\small
	\begin{eqnarray}
		\lefteqn{\sup\{\norm{x-x^*(\lambda^*;\theta^*)}:\ x\in X^*(\lambda^*;\theta)\}}\nonumber\\
		&&\leq c^* \sup_{x\in X}\norm{\grad_x f(x;\theta)-\grad_x f(x;\theta^*)+(A(\theta)-A(\theta^*)){^\top}\lambda^*}\nonumber\\
		&&\leq c^*~(L_{F,\theta}+L_A\norm{\lambda^*})\norm{\theta-\theta^*},\quad \forall\theta\in\Theta.\label{eq:assumption1iii_alternative}
\end{eqnarray}}%
Thus, \eqref{eq:main-assumption} is satisfied for
$\lambda=\lambda^*$. Therefore, {by resorting to Assumption~\ref{Assump: Aug_1}.i (and not to Assumption~\ref{Assump: Aug_1}.ii or \ref{Assump: Aug_1}.iii)} and using \eqref{eq:assumption1iii_alternative} one can show $\{\lambda_k\}_{k\in\integers_+}\subset\cK^*$ is a bounded sequence; hence, there exists a compact set $\Lambda\subset\cK^*$ such that $\{\lambda_k\}\subset\Lambda$. Moreover, Theorem~\ref{thm:pseudoLipschitz} implies that for each $\lambda\in\Lambda$, \eqref{eq:main-assumption} holds for all $\theta\in\Theta$ with $\kappa(\lambda)\triangleq c~(L_{F,\theta}+L_A\norm{\lambda})$	for some $c>0$, and clearly $\kappa(\cdot)$ is continuous on $\Lambda$.}
\end{remark}

Next, we prove that the augmented Lagrangian scheme generates %a sequence
$\{\lambda_k\}$ such that $\bar{\lambda}_k \to \lambda^*\in\Lambda^*\triangleq \argmax_\lambda g_0(\lambda;\theta^*)$ as $k \to \infty$ by
deriving a rate statement first. %on the averaged sequence.
\begin{theorem}[\textbf{Bound on dual suboptimality}]
\label{Theorem:Aug_bnd_dual_sub}
%Let Assumptions~\ref{Assump: Aug_1} -- \ref{Assump:Aug_2} hold and
{Under the premise of %Proposition
Theorem~\ref{Theorem: Aug_lambda_bnd_cont_rho}}, let $\{\lambda_k\}_{k\geq 1}$ denote the sequence generated by %Algorithm~\ref{Alg: ALM}
{IPALM}, %In addition,
and let $\bar{\lambda}_k\triangleq {1\over k}{\sum_{i=1}^k \lambda_i}$. {Then $\lambda^*\triangleq\lim_{k\rightarrow\infty}\bar{\lambda}_k$ exists; moreover, $\lambda^*\in\argmax_\lambda g_0(\lambda;\theta^*)$ and} \vspace*{-2mm}
\begin{equation}
0\leq f^*-g_\rho(\bar{\lambda}_{k};\theta^*)
% = \sup_{\lambda} \ g_\rho(\lambda;\theta^*)-g_\rho(\bar{\lambda}_{k};\theta^*)
\leq{B_g\over k}\label{Bound_dual_cons_rho}
\end{equation}
for all $k\geq 1$, where %$\lambda^*\in \argmax g_0(\lambda,\theta^*)$,
$C_\lambda$ is defined in %Proposition
Theorem~\ref{Theorem: Aug_lambda_bnd_cont_rho}, and %$B_g$ is defined as %follows:
\vspace*{-2mm}
$$B_{g}\triangleq \tfrac{1}{2\rho}\|\lambda_0-\lambda^*\|^2+C_{\lambda}\Big(\sqrt{\tfrac{2}{\rho}}\sum_{k=0}^{\infty} \sqrt{\alpha_k}+\frac{{\bar{\kappa}}\|\theta_0-\theta^*\|}{1-\tau}\Big).$$
\vspace*{-2mm}
\end{theorem}
\begin{IEEEproof}
{Let $\lambda^*\in\Lambda^*$ {denote} an arbitrary dual solution.} From Lemma~\ref{Lemma: Aug_Rock_Lemmas} and by recalling that the duality gap for $\mathcal{C}(\theta^*)$ is {zero}, it follows that $f^*=\max_\lambda g_\rho(\lambda;\theta^*)$ for all $\rho>0$. {By invoking} the Lipschitz continuity of $\nabla_{\lambda} g_\rho(\lambda;\theta^*)$ in $\lambda$ with constant ${1/\rho}$, {the following holds for $i\geq 0$:}
\begin{align}
-g_\rho(\lambda_{i+1};\theta^*)\leq &
-g_\rho(\lambda_i;\theta^*)-\nabla_{\lambda}
g_\rho(\lambda_i;\theta^*)^\top(\lambda_{i+1}-\lambda_i) \nonumber \\
		& +\tfrac{1}{2\rho}\|\lambda_{i+1}-\lambda_i\|^2.\label{proof:ineq_decent}
\end{align}
Under the concavity of $g_\rho(\lambda;\theta^*)$ in $\lambda$, we have that
$-g_\rho(\lambda^*;\theta^*)\geq -g_\rho(\lambda_i;\theta^*)-\nabla_{\lambda} g_\rho(\lambda_i;\theta^*)^\top(\lambda^*-\lambda_i)$.
By combining this inequality and \eqref{proof:ineq_decent}, we get
{\small
\begin{align}
-g_\rho&(\lambda_{i+1};\theta^*) \nonumber \\
\leq &
-g_\rho(\lambda_i;\theta^*)-\nabla_{\lambda}
g_\rho(\lambda_i;\theta^*)^\top(\lambda_{i+1}-\lambda_i)+\tfrac{1}{ 2\rho}\|\lambda_{i+1}-\lambda_i\|^2 \notag\\
\leq &-g_\rho(\lambda^*;\theta^*)-\nabla_{\lambda} g_\rho(\lambda_i;\theta^*)^\top(\lambda_{i+1}-\lambda^*)+\tfrac{1}{2\rho}\|\lambda_{i+1}-\lambda_i\|^2  \notag \\
		%& -\nabla_{\lambda} g_\rho(\lambda_i;\theta^*)^\top(\lambda_{i+1}-\lambda_i)+{1\over 2\rho}\|\lambda_{k+1}-\lambda_i\|^2  \notag \\
%& = -g_\rho(\lambda^*;\theta^*)-\nabla_{\lambda} g_\rho(\lambda_i;\theta^*)^\top(\lambda_{i+1}-\lambda^*)+{1\over 2\rho}\|\lambda_{i+1}-\lambda_i\|^2 \notag \\
\notag = &-g_\rho(\lambda^*;\theta^*)-\nabla_{\lambda} {\cal
	L}_\rho(x_{i+1},\lambda_i;\theta_{i})^\top(\lambda_{i+1}-\lambda^*)\\
& +\delta_i^\top(\lambda_{i+1}-\lambda^*)+s_i^\top(\lambda_{i+1}-\lambda^*)+\tfrac{1}{2\rho}\|\lambda_{i+1}-\lambda_i\|^2
	\notag \\
\leq &-g_\rho(\lambda^*;\theta^*)-\tfrac{1}{\rho}
(\lambda_{i+1}-\lambda_{i})^\top(\lambda_{i+1}-\lambda^*)+\tfrac{1}{2\rho}\|\lambda_{i+1}-\lambda_i\|^2\notag \\
& +\|\delta_i\|\|\lambda_{i+1}-\lambda^*\|+\|s_i\|\|\lambda_{i+1}-\lambda^*\|,\label{ineq-bd-lamb}
\end{align}}%
where {$\delta_i\triangleq \nabla_{\lambda} g_\rho(\lambda_i;\theta_{i})-\nabla_{\lambda} g_\rho(\lambda_i;\theta^*)$} and {$s_i\triangleq \nabla_{\lambda} {\cal L}_\rho(x_{i+1},\lambda_i;\theta_{i})-\nabla_{\lambda} g_\rho(\lambda_i;\theta_{i})$}.
Note that
$\|\lambda_{i+1}-\lambda_i\|^2+2(\lambda_{i+1}-\lambda_{i})^\top(\lambda^*-\lambda_{i+1})=\|\lambda_i-\lambda^*\|^2-\|\lambda_{i+1}-\lambda^*\|^2$;
	hence, we can rewrite \eqref{ineq-bd-lamb} as
\begin{align}
-g_\rho(\lambda_{i+1};\theta^*) \leq &-g_\rho(\lambda^*;\theta^*)+\left(\|\delta_i\|+\|s_i\|\right)\|\lambda_{i+1}-\lambda^*\|
\notag \\
		& +\tfrac{1}{2\rho}\left(\|\lambda_i-\lambda^*\|^2-\|\lambda_{i+1}-\lambda^*\|^2\right).\label{bd-lamb-ineq2}
\end{align}
By summing \eqref{bd-lamb-ineq2} over $i=0,\ldots,k-1$, replacing
$g_\rho(\lambda^*;\theta^*)$ by $f^*=\sup_\lambda g_\rho(\lambda;\theta^*)$, %and setting $\lambda_0=\mathbf{0}$,
we obtain \vspace*{-1mm}
{%\
\begin{eqnarray}
\lefteqn{-\sum_{i=0}^{k-1} (g_\rho(\lambda_{i+1};\theta^*)-%\sup_\lambda g_\rho(\lambda;\theta^*)
f^*)+\tfrac{1}{2\rho} \|\lambda_{k}-\lambda^*\|^2} \label{proof: bnd_dual_sub_opt} \\
&  \quad \leq \tfrac{1}{2\rho} \|\lambda_0-\lambda^*\|^2 +\sum_{i=0}^{k-1}\left(\|\delta_i\|+\|s_i\|\right)\|\lambda_{i+1}-\lambda^*\|.
\nonumber
\end{eqnarray}}%
Under concavity of $g_\rho(\lambda;\theta^*)$ in $\lambda$, the following holds:
%\begin{align*}
$- \Big(g_\rho(\bar{\lambda}_{k};\theta^*)-%\sup_{\lambda} g(\lambda,\theta^*)
f^*\Big)\leq -{1\over k}\sum_{i=0}^{k-1} \Big(g_\rho(\lambda_{i+1};\theta^*)-%\sup_\lambda g_\rho(\lambda,\theta^*)
f^*\Big)$.
%\end{align*}
By dividing both sides of \eqref{proof: bnd_dual_sub_opt} by $k$, dropping the positive term on the left hand side, and using the fact that $\bar{\lambda}_k\in\cK^*$, we get \vspace*{-2mm}
\begin{align} \notag
0 & \leq f^* %\sup_\lambda g_\rho(\lambda;\theta^*)
-g_\rho(\bar{\lambda}_{k};\theta^*)\\
&\leq\frac{1}{k}\big(\tfrac{1}{2\rho}\|\lambda^*\|^2+\sum_{i=0}^{k-1}\left(\|\delta_i\|+\|s_i\|\right)\|\lambda_{i+1}-\lambda^*\|\big).\notag \vspace*{-2mm}
%& \leq\frac{1}{k+1}\left(\tfrac{1}{2\rho}\|\lambda_{0}-\lambda^*\|^2 +\sum_{i=0}^\infty\left(\|\delta_i\|+\|s_i\|\right)\|\lambda_{i+1}-\lambda^*\|\right) \notag
%\label{proof:averg_bnd}
\end{align}
Lemma~\ref{Lemma: Aug_Rock_Lemmas}.iii and %Proposition
Theorem~\ref{Lips_contin_grad_g} imply that $\|s_i\|\leq\sqrt{
	{2\alpha_i\over \rho}}$ and
$\|\delta_i\|\leq {\bar{\kappa}}\|\theta_{i}-\theta^*\|$, respectively for $i\geq 0$. In addition, from %Proposition
Theorem~\ref{Theorem: Aug_lambda_bnd_cont_rho}, we have $\|\lambda_i-\lambda^*
\|\leq C_{\lambda}$ for all $i\geq 1$. Then by the summability of $\sqrt{\alpha_i}$, we have that \vspace*{-2mm}
\begin{eqnarray}
\lefteqn{{\sum_{i=0}^{\infty}
	(\|\delta_i\|+\|s_i\|)\|\lambda_{i+1}-\lambda^*\|}} \notag \\
& & \leq  C_{\lambda}\Big({\bar{\kappa}} \sum_{i=0}^{\infty} \|\theta_{i}-\theta^*\|+\sqrt{\tfrac{2}{\rho}}\sum_{i=0}^\infty\sqrt{\alpha_i}\Big).\label{proof:bnd_subopt_1}
 \vspace*{-2mm}
\end{eqnarray}
Furthermore, substituting $\sum_{i=0}^{\infty} \|\theta_{i}-\theta^*\|= \|\theta_0-\theta^*\|/(1-\tau)$ into \eqref{proof:bnd_subopt_1} gives the desired bound in \eqref{Bound_dual_cons_rho}. %and completes the proof.\

{Since $\{\bar{\lambda}_k\}$ is a bounded sequence, there exists a convergent subsequence $\cS\subset\integers_+$, i.e., $\lim_{k\in\cS}\bar{\lambda}_k$ exists. Note $g_\rho(\cdot;\theta^*)$ is smooth; hence, taking the limit of \eqref{Bound_dual_cons_rho} along $k\in\cS$, we get $g_\rho(\lim_{k\in\cS}\bar{\lambda}_k;\theta^*)=f^*$. Thus, $\lim_{k\in\cS}\bar{\lambda}_k\in\Lambda^*$. Finally, since \eqref{eq:SM-convergence} is true for any dual optimal point, setting $\lambda^*\triangleq\lim_{k\in\cS}\bar{\lambda}_k$ within \eqref{eq:SM-convergence} implies that $\lambda^*=\lim_{k\rightarrow\infty}\bar{\lambda}_k$ -- for details, see dual convergence proof of Theorem~\ref{Theorem:Aug_increas_rho_linear_rate} provided \sa{in Appendix~\ref{sec:app-dual} of the online supplement.}} %and this completes the proof.}
\end{IEEEproof}
Next, we derive a bound on the primal {\it infeasibility} \sa{for the IPALM iterate sequence.}
%, where the primal iterate sequence is computed such that Step~1 in %Algorithm~\ref{Alg: ALM} {IPALM} is satisfied. 
%Prior to proving our main result,
{First, we provide two supporting technical lemmas; the first one follows directly
	from Thm.~2.1.5 in~\cite{nesterov2013introductory} {while the proof of the second is elementary and} is omitted.}
\begin{lemma}
\label{Lemma:Aug_grad_bnd}
Let $\phi(\lambda):\R^m\to \R$ be a concave function %whose supremum is finite and is attained at $\lambda_{\phi}^*$.
%In addition, assume that
such that $\lambda^*\in\argmax_{\lambda\in\R^m}\phi(\lambda)$ attains the maximum.
If $\nabla \phi$ is Lipschitz continuous with constant $L_{\phi}$, then, %we have that for all $\lambda\in \R^m$
{\small
$$\|\nabla \phi(\lambda)\|\leq \sqrt{2L_{\phi}\big(\phi(\lambda^*)-\phi(\lambda)\big)},\ \forall \lambda\in \R^m.$$}
\end{lemma}
%{\bf Should ommit the proof and cite??}
%{The Lipschitz} continuity of $\nabla \phi$ implies that
%\begin{align*}
%\phi(\lambda')\geq \phi(\lambda)+\nabla_{\lambda} \phi(\lambda)^\top(\lambda'-\lambda)-{L_{\Phi}\over 2}\|\lambda'-\lambda\|^2.
%\end{align*}
%for all $\lambda,\lambda'\in \R^m$. Taking the supremum over {$\lambda'$} {on} both sides of above inequality proves the claim after rearranging the terms, i.e.,
%\begin{align*}
%\phi(\lambda^*_{\phi})&\geq \phi(\lambda)+\sup_{\lambda' \in \R^m}\{\nabla \phi(\lambda)^\top(\lambda'-\lambda)-{L_{\phi}\over 2}\|\lambda'-\lambda\|^2\}
%%&=\phi(\lambda)+\sup_{u \in \R^m}\{\nabla \phi(\lambda)^\topu-{L_{\phi}\over 2}\|u\|^2\}\\
%=\phi(\lambda)+{\|\nabla \phi(\lambda)\|^2\over 2L_{\phi}}.
%\end{align*}
%{Next, for any $y, y' \in \R^m$, we provide a bound on $d_{\Kscr}\big(y+y'\big)$ -- {the proof is elementary; hence, omitted.}
\begin{lemma}
\label{Lemma:d_triangle}
{Given a closed convex cone $\Kscr$, $d_{\Kscr}\big(y+y'\big)\leq d_{\Kscr}\big(y\big)+\|y'\|$ for all $y,y' {\in \Real^m}$.}
\end{lemma}
\begin{comment}
\begin{IEEEproof}
{From the definition of $d_X(u)$ and the triangle inequality, we obtain the following bound:}
%the triangle
%inequality, we obtain the following inequality:
\begin{align*}
d_{\Kscr}\big(y+y'\big) & =\|y+y'-\Pi_{\Kscr}(y+y')\|\\
		& =\left\|y-\Pi_{\Kscr}(y) +y'+y-y+\Pi_{\Kscr}(y)-\Pi_{\Kscr}(y'+y)\right\|\\
&\leq \|y-\Pi_{\Kscr}(y)\| +\|y'+y-\Pi_{\Kscr}(y'+y)-(y-\Pi_{\Kscr}(y))\|.
\end{align*}
Define $\Pi^c_{\Kscr}(x)\triangleq x-\Pi_{\Kscr}(x)$, which is a nonexpansive operator (cf.~\cite[Chapter 1.]{facchinei02finite}). Using this operator, we can rewrite the above inequality as follows:
\begin{align*}
d_{\Kscr}\big(y+y'\big)& \leq \|y-\Pi_{\Kscr}(y)\|+\|\Pi^c_{\Kscr}(y+y')-\Pi^c_{\Kscr}(y)\|
\leq d_{\Kscr}\big(y\big)+\|y'\|,
\end{align*}
where the last inequality follows from nonexpansivity of $\Pi^c_{\Kscr}$.
\end{IEEEproof}}
\end{comment}
We now derive the bound on the primal infeasibility. %after $k$ iterations.
\begin{theorem}[\textbf{Bound on primal infeasibility}]
\label{Theorem:Aug_prim_infeas_cons_rho}
%Let Assumptions~\ref{Assump: Aug_1}--\ref{Assump:Aug_2} hold and
{Under the premise of %Proposition
Theorem~\ref{Theorem: Aug_lambda_bnd_cont_rho}}, let
$\{\lambda_k\}_{k\geq 0}$ and $\{x_k\}_{k\geq 0}$ denote the sequences generated by
%Algorithm~\ref{Alg: ALM}
{IPALM}. Furthermore, let {$\xbar_k\triangleq{1\over k}{\sum_{i=1}^{k} x_i}$}. Then, {for all $k\geq 1$}, it follows that
\begin{align}
d_{-\Kscr}\Big(h(\xbar_k,\theta^*)\Big)\leq{\cal V}(k)\triangleq \frac{C_1}{\sqrt{k}}+\frac{C_2}{k},\label{Def:V(k)}
\end{align}
where %${\cal V}(k)$ is defined as
$C_1 \triangleq
\sqrt{\frac{2B_g}{\rho}+\left(\frac{C_\lambda}{\rho}\right)^2}$ and $C_2
\triangleq \sqrt{\tfrac{2}{\rho}}~\sum_{i=0}^\infty
	\sqrt{\alpha_i} + (L_{h,\theta}+\bar{\kappa}) \|\theta_0-\theta^*\|/ (1-\tau)$.
\end{theorem}
\begin{IEEEproof}
%For  $i=1,\hdots,k$, $F_0(x_i,u_i;\theta_{i})$ is finite at
Let {$u_i\triangleq\nabla_\lambda{\cal L}_\rho(x_{i+1},\lambda_i;\theta_i)$
	for all $i\geq 0$. Note %that {trivially} %$h(x_{i+1};\theta^*)$ can be rewritten as
\begin{align*}
h(x_{i+1};\theta^*)=&h(x_{i+1};\theta^*)\\
	& +\underbrace{u_i-h\left(x_{i+1};\theta_i\right)+\Pi_{-\Kscr}\left(\tfrac{\lambda_i}{\rho}+h\left(x_{i+1};\theta_i\right)\right)}_{\mbox{\scriptsize Term
				1}},\end{align*}
{since}~\eqref{Grad_l} implies that Term 1 $= 0$.
Hence, using Lemma~\ref{Lemma:d_triangle}, %we get %the following inequality for all $i\geq 0$:
\begin{align}
\label{eq:feasibility-bound-1}
& d_{-\cK}\big(h(x_{i+1};\theta^*)\big) \notag \\
\leq & d_{-\cK}\left(\Pi_{-\Kscr}\big(\tfrac{\lambda_i}{\rho}+h\left(x_{i+1};\theta_i\right)\big)\right)\notag \\
	& +\norm{u_i+h(x_{i+1};\theta^*)-h\left(x_{i+1};\theta_i\right)},\quad \forall\ i\geq 0.
\end{align}
%it trivially follows that
%\begin{align}h_j(x_i,\theta_{i})\leq [u_i]_j, \hbox{ for } j=1,\hdots,m. \label{bd-u} \end{align}
Under Assumption~\ref{Assump: Aug_1}.i (from
		Remark~\ref{rem:Lipschitz-h}), we have
$\norm{ h(x_{i+1},\theta_i) - h(x_{i+1},\theta^*)} \leq {L_{h,\theta}} \|\theta_i - \theta^*\|$ for all $i\geq 0$. Moreover, since $\Pi_{-\cK}(y)\in-\cK$ for all $y\in\reals^m$ and $d_{-\cK}(y)=0$ for all $y\in -\cK$, it follows from \eqref{eq:feasibility-bound-1} that
{%\
\begin{equation*}
d_{-\cK}\big(h(x_{i+1};\theta^*)\big)\leq \norm{u_i}+L_{h,\theta} \norm{\theta^*-\theta_i},\ \forall i\geq 0.
\end{equation*}}%
%%implying that $h_j(x_i,\theta_{i})\geq h_j(x_i,\theta^*)-L^j_{h}\|\theta_{i}-\theta^*\|$.
%Combining this with \eqref{bd-u}, we obtain
% \begin{align}
% \label{bd-u2}
% h_j(x_i,\theta^*)\leq
% [u_i]_j+L^j_{h}\|\theta_{i}-\theta^*\|.\end{align}
Since $h(\cdot;\theta^*)$ is an affine function in $x$ and $d_{-\cK}(\cdot)$ is a convex function, their composition $d_{-\cK}\big(h(\cdot;\theta^*)\big)$ is a convex function as well; therefore, from Jensen's inequality we get \vspace*{-2mm}
{\
\begin{align}
d_{-\Kscr}\Big(h(\xbar_k,\theta^*)\Big)&\leq {1\over k} \sum_{i=0}^{k-1}
d_{-\cK}\big(h(x_{i+1};\theta^*)\big) \notag\\
& \leq \frac{1}{k}\sum_{i=0}^{k-1} \left(\|u_i\|+ {L_{h,\theta}} \|\theta_{i}-\theta^*\|\right). \label{proof:bnd_prim_infea_2}
\end{align}}%
}
%\begin{align}
%   \label{bd-h}
%   \sum_{i=0}^k h_j(x_i,\theta^*)\leq \sum_{i=0}^k [u_i]_j+\sum_{i=0}^k
%   L^j_{h}\|\theta_{i}-\theta^*\| .\end{align}
% On the other hand, convexity of $h_j(x,\theta^*)$ in $x$ implies that
%  $$h_j(\xbar_k,\theta^*)\leq {1\over k+1} \sum_{i=0}^k h_j(x_i,\theta^*);$$
% hence, for all $j=1,\dots,m$, we have from \eqref{bd-h},
%\begin{align}
%h_j(\xbar_k,\theta^*) \leq {1\over k+1}\left(\sum_{i=0}^k [u_i]_j+\sum_{i=0}^k L^j_{h}\|\theta_{i}-\theta^*\|\right). \label{bd-h2}
%%\notag &\leq {1\over k+1}\sum_{i=0}^k (u_i)_j+{1\over k+1} \left(\max_{j=1,\hdots,m} L^j_{h,\theta}\right) \sum_{i=0}^k\|\theta_{i}-\theta^*\|\\
%%&\leq {1\over k+1}\left(\sum_{i=0}^k [u_i]_j+L_{h}\sum_{i=0}^k\|\theta_{i}-\theta^*\|\right),
% \end{align}
%Hence, $L_{h} \triangleq \max\{L^j_{h}:\ j=1,\hdots,m\}$, and \eqref{bd-h2} imply that
%\begin{align}
%d_{-\Kscr}\Big(h(\xbar_k,\theta^*)\Big)
%&\leq {1\over k+1}\left(\left\|\sum_{i=0}^k u_i\right\|+ L_{h}\sum_{i=0}^k \|\theta_{i}-\theta^*\|\right)\notag\\
%&\leq \frac{1}{k+1}\left(\sum_{i=0}^k \|u_i\|+ L_{h}\sum_{i=0}^k \|\theta_{i}-\theta^*\|\right).\label{proof:bnd_prim_infea_2}
%\end{align}
Recall that from Lemma~\ref{Lemma: Aug_Rock_Lemmas}.iii, for $i=0,\hdots,k-1,$
{\
$$\Big\|\nabla_{\lambda} {\cal
	L}_\rho(x_{i+1},\lambda_i;\theta_{i})-\nabla_{\lambda}
	g_\rho(\lambda_i;\theta_{i})\Big\|\leq \sqrt{\tfrac{2\alpha_i}{\rho}};$$}%
therefore, it follows that %we obtain %that
$\|u_i\|=\|\nabla_{\lambda} {\cal
	L}_\rho(x_{i+1},\lambda_i;\theta_{i})\|\leq \|\nabla_{\lambda}
	g_\rho(\lambda_i;\theta_{i})\|+\sqrt{2\alpha_i/\rho}$. In addition, since  $\|\nabla_{\lambda} g_\rho(\lambda_i;\theta_{i})\|$ $\leq \|\nabla_{\lambda} g_\rho(\lambda_i;\theta^*)\|+{\bar{\kappa}}\|\theta_{i}-\theta^*\|$, we get %the following bound:
{\
$$\|u_i\|\leq \|\nabla_{\lambda}
g_\rho(\lambda_i;\theta^*)\|+\sqrt{2\alpha_i/\rho}+{\bar{\kappa}}\|\theta_{i}-\theta^*\|.$$}%
On the other hand, by Lemma~\ref{Lemma:Aug_grad_bnd}, we have
$ \|\nabla_{\lambda} g_\rho(\lambda_i;\theta^*)\|\leq \sqrt{ \tfrac{2}{\rho} \left(f^*-g_\rho(\lambda_i;\theta^*)\right)}$; hence,
%Combining this with the previous inequality leads to
{\
\begin{align}
\|u_i\| & \leq \sqrt{\tfrac{2}{\rho}
	\Big(f^*-g_\rho(\lambda_i;\theta^*)\Big)}+\sqrt{\tfrac{2\alpha_i}{\rho}}+{\bar{\kappa}}\|\theta_{i}-\theta^*\|.\notag
\end{align}}%
By substituting this bound into~\eqref{proof:bnd_prim_infea_2} and using the concavity of square-root function $\sqrt{\cdot}$, we get
{\small
\begin{align}
d_{-\Kscr}\Big(h(\xbar_k,\theta^*)\Big) %\notag\\
% \leq & {1\over k}\sum_{i=0}^{k-1} \sqrt{ {2\over \rho}
%	\Big(f^*-g_\rho(\lambda_i;\theta^*)\Big)} \notag \\
% & + {1\over k}\left(\sum_{i=0}^{k-1} \sqrt{\frac{2\alpha_i}{\rho}}+(L_{h}+\kappa)\sum_{i=0}^{k-1} \|\theta_{i}-\theta^*\|\right)\notag\\
\leq & \sqrt{ \frac{2}{\rho} \Big(f^*-{1\over k}\sum_{i=0}^{k-1}
		g_\rho(\lambda_i;\theta^*)\Big)} \notag\\
&+{1\over k}\big(\sum_{i=0}^{k-1} \sqrt{\frac{2\alpha_i}{\rho}}+{(L_{h,\theta}+\bar{\kappa})}\sum_{i=0}^{k-1} \|\theta_{i}-\theta^*\|\big),
\label{proof:bnd_prim_infea_5}
\end{align}}%
%where the last inequality follows from concavity of square-root function $\sqrt{\cdot}$.
The first term in
\eqref{proof:bnd_prim_infea_5} can be bounded using \eqref{proof: bnd_dual_sub_opt} and \eqref{proof:bnd_subopt_1}, %which states that
%\begin{eqnarray}
\begin{align*}
f^*-{1\over k}\sum_{i=0}^{k-1}g_\rho(\lambda_i;\theta^*)\leq  {1 \over k}\big(B_g+g_\rho(\lambda_0;\theta^*)-g_\rho(\lambda_{k};\theta^*)\big).
%\label{proof:bnd_prim_infea_3}
\end{align*}
%\end{eqnarray}
Note $g_\rho(\lambda_0;\theta^*)- f^*\leq 0$, and Lipschitz continuity of $\nabla g_\rho$ implies $f^*-g_\rho(\lambda_{k};\theta^*)\leq\tfrac{1}{2\rho}\|\lambda_{k}-\lambda^*\|^2\leq \tfrac{1}{2\rho} C_\lambda^2$.
The %remaining
other terms in
\eqref{proof:bnd_prim_infea_5} can %also
be bounded as %follows
%\begin{eqnarray}
{\
\begin{eqnarray}
\lefteqn{{1\over k}\Big(\sum_{i=0}^{k-1} \sqrt{\frac{2\alpha_i}{\rho}} +
(L_{h,\theta}+\kappa)\sum_{i=0}^{k-1} \|\theta_{i}-\theta^*\|\Big)} \notag \\
&& \leq
{1\over k}\Big[\sum_{i=0}^\infty \sqrt{\frac{2\alpha_i}{\rho}}
		 +{ {(L_{h,\theta}+\bar{\kappa})}
	\|\theta_0-\theta^*\| \over 1-\tau}\Big].\label{proof:bnd_prim_infea_4}
\end{eqnarray}
}%
%\end{eqnarray}
The result follows by incorporating these bounds %~\eqref{proof:bnd_prim_infea_3} and~\eqref{proof:bnd_prim_infea_4}
into \eqref{proof:bnd_prim_infea_5}.
\end{IEEEproof}
%We now proceed to
Next, we derive %lower and upper
bounds on
$f(\xbar_k;\theta^*)-f^*$. In contrast with %standard
the \emph{unconstrained} %convex
optimization, $f(\xbar_k;\theta^*)$ could be less than
$f^*$, due to infeasibility of $\xbar_k$.

\begin{theorem}[{\bf Bounds on primal suboptimality}]
\label{Theorem: Aug_prim_low_bnd_cons_rho}
%Let Assumption~\ref{Assump: Aug_1}--\ref{Assump:Aug_2} hold and
{Under the premise of %Proposition
Theorem~\ref{Theorem: Aug_lambda_bnd_cont_rho}}, let $\{x_k\}$ and $\{\lambda_k\}$ be the
sequences generated by %Algorithm~\ref{Alg: ALM}
{IPALM}. In addition,
let $\xbar_k={1\over k}\sum_{i=1}^k x_k$. Then {for all $k\geq 1$,} the following {bounds} hold:
\begin{align}
f(\xbar_k;\theta^*)-f^* & \geq -{\rho \over 2} {\cal
	V}^2(k)-\|\lambda^*\|{\cal V}(k),\\ %(x^*;\theta^*)
f(\xbar_k;\theta^*)- f^* & \leq \frac{U}{k},
\end{align}
for any $\lambda^*\in\argmax\, g_0(\lambda,\theta^*)$, where ${\cal V}(k)$ is defined in Theorem~\ref{Theorem:Aug_prim_infeas_cons_rho},
$U\triangleq \sum_{i=0}^\infty
\alpha_i+{\tfrac{1}{2}\norm{\lambda_0}^2}+\tfrac{\rho}{2}{L^2_{h,\theta}}\frac{\|\theta_0-\theta^*\|^2}{1-\tau^2}+\left(\bar{C}{L_{h,\theta}}+2{L_{f,\theta}}\right)\frac{\|\theta_0-\theta^*\|}{1-\tau}$,
	and $\bar{C}\triangleq C_\lambda+ \|\lambda^*\|.$
%\begin{equation*}
%\Big({\rho \over 2}L^2_{h,\theta}
%		{\|\theta_0-\theta^*\|^2\over 1-q^2_\ell}+(L_{h,\theta}C_\lambda+2L_{f,\theta}){\|\theta_0-\theta^*\|\over 1-\tau}+{1 \over 2\rho}\|\lambda_0\|^2+\sum_{i=0}^\infty \alpha_i\Big)
%\end{equation*}
\end{theorem}
\begin{IEEEproof}
%We first derive the lower bound and then the upper bound.\\
{We give the proof in two parts below.}\\
\noindent {\bf Proof of the lower bound:} Since $\sup_\lambda g_\rho(\lambda;\theta^*)=\min_{x\in X} {\cal L}_\rho(x,\lambda^*;\theta^*)=f^*$, we have that for all $k\geq 1$,
{\
\begin{align*}
f^*&\leq {\cal L}_\rho(\xbar_k,\lambda^*;\theta^*)\nonumber\\
&=f(\xbar_k;\theta^*)+{\rho\over
		2}d^2_{-\Kscr}\Big(h(\xbar_k;\theta^*)+{\lambda^*\over
				\rho}\Big) -{\|\lambda^*\|^2\over 2\rho}\\
&\leq f(\xbar_k;\theta^*)+{\rho\over
	2}\left(d_{-\Kscr}\big(h(\xbar_k;\theta^*)\big)+{\|\lambda^*\|\over \rho}\right)^2-{\|\lambda^*\|^2\over 2\rho},
\end{align*}}%
where the first equality is a consequence of \eqref{Aug_L_forrho} while
the second inequality follows from Lemma~\ref{Lemma:d_triangle}. By expanding the second term above inequality, we obtain
\begin{align*}
%f(x^*;\theta^*)
f^*  %\leq f(\xbar_k;\theta^*)-{\|\lambda^*\|^2\over 2\rho}\\
%& +{\rho\over 2}\Bigg[d_{-\Kscr}\Big(-h(\xbar_k;\theta^*)\Big)^2+\Big({\|\lambda^*\|\over \rho}\Big)^2+2d_{\mathbb{R}^m_+}\Big(-h(\xbar_k;\theta^*)\Big){\|\lambda^*\|\over \rho}\Bigg]\\
&\leq f(\xbar_k;\theta^*)+\tfrac{\rho}{2}d^2_{-\Kscr}\big(h(\xbar_k;\theta^*)\big)+d_{-\Kscr}\big(h(\xbar_k;\theta^*)\big)\|\lambda^*\|\\
&\leq f(\xbar_k;\theta^*)+{\rho \over 2} {\cal V}^2(k)+\|\lambda^*\|{\cal V}(k),
\end{align*}
where the last inequality follows from
Theorem~\ref{Theorem:Aug_prim_infeas_cons_rho}. %The result follows from a rearrangement of the terms.
%We conclude this section by deriving an upper bound on the primal suboptimality.
%\begin{theorem}[{\bf Upper bound on primal suboptimality}]
%\label{Theorem: Aug_prim_upp_bnd_cons_rho}
%Let Assumption~\ref{Assump: Aug_1}--\ref{Assump:Aug_2} hold and let $\{x_k\}$ and $\{\lambda_k\}$ be the
%sequences generated by Algorithm~\ref{Alg: Aug_const_rho}. In addition, let $\xbar_k={1\over k+1}\sum_{i=0}^k x_i$, and $\bar{C}=C_\lambda+\|\lambda^*\|$ for some optimal dual $\lambda^*$. Then, the following holds for all $k\geq 0$
%\begin{align*}
%f(\xbar_k;\theta^*)- f^*\leq \frac{U}{k},
%\end{align*}
%where $U:=\tfrac{\rho}{2}L^2_{h}\frac{\|\theta_0-\theta^*\|^2}{1-q^2}+\left(\bar{C}L_{h}+2L_{f}\right)\frac{\|\theta_0-\theta^*\|}{1-q}+\sum_{i=0}^\infty \alpha_i$.
%\begin{equation*}
%\Big({\rho \over 2}L^2_{h,\theta}
%		{\|\theta_0-\theta^*\|^2\over 1-q^2_\ell}+(L_{h,\theta}C_\lambda+2L_{f,\theta}){\|\theta_0-\theta^*\|\over 1-\tau}+{1 \over 2\rho}\|\lambda_0\|^2+\sum_{i=0}^\infty \alpha_i\Big)
%\end{equation*}
%\end{theorem}
%\begin{IEEEproof}

\noindent {\bf Proof of the upper bound:} Let $x^*$ be an optimal solution to $\mathcal{C}(\theta^*)$, {i.e., $x^*\in X^*(\theta^*)$}. Step~1 of %Algorithm~\ref{Alg: ALM}
{IPALM} implies
%\begin{align*}
 ${\cal L}_\rho(x_{i+1},\lambda_i;\theta_i)
 %&\leq  \inf_{x\in X} {\cal L}_\rho(x,\lambda_k;\theta_{k})+\alpha_k
 \leq  {\cal L}_\rho(x^*,\lambda_i;\theta_i)+\alpha_i$ for all $i\geq 0$.
%\end{align*}
Hence, by the definition of ${\cal L}_\rho$, it follows that
%{\
%\begin{align*}
%& \quad f(x_{i+1};\theta_i)+{\rho\over 2}d^2_{-\Kscr}\left(h(x_{i+1};\theta_i)+{\lambda_i\over
%			\rho}\right)-{\|\lambda_i\|^2\over 2\rho}\\
%& \leq  f(x^*;\theta_i)+{\rho\over 2}d^2_{-\Kscr}\left(h(x^*;\theta_i)+{\lambda_i\over \rho}\right)-{\|\lambda_i\|^2\over 2\rho}+\alpha_i,
%\end{align*}}%
%which leads to
{\
\begin{align}
f(x_{i+1};\theta_i)-f(x^*;\theta_i) \leq &
\tfrac{\rho}{2}d^2_{-\Kscr}\left(h(x^*;\theta_i)+{\lambda_i\over
		\rho}\right) \label{proof:primal_sub_opt_cons_rho1}\\
& -\tfrac{\rho}{2}d^2_{-\Kscr}\left(h(x_{i+1};\theta_i)+{\lambda_i\over
		\rho}\right)+\alpha_i. \notag
\end{align}}%
Step~2 of %Algorithm~\ref{Alg: ALM}
{IPALM} {and \eqref{eq:lambda-relation} imply that}
{\
\begin{align}
%\lambda_{i+1} = %\lambda_i + \rho \left( h(x_{i+1};\theta_i)- \Pi_{-\Kscr} \left(\tfrac{\lambda_i}{\rho}+h(x_{i+1};\theta_i)\right)\right) \\
%%\notag & =
%\rho \left( h(x_{i+1};\theta_i) + \tfrac{\lambda_i}{\rho} -
%		\Pi_{-\Kscr}
%		\left(\tfrac{\lambda_i}{\rho}+h(x_{i+1};\theta_i)\right)\right) \implies\
d_{-\Kscr}\left(h(x_{i+1};\theta_i)+{\lambda_i\over \rho}\right)  ={\|\lambda_{i+1}\|\over \rho}.
\label{proof:primal_sub_opt_cons_rho2}
\end{align}}%
In addition, by using Lemma~\ref{Lemma:d_triangle}, it follows that
{\
\begin{align}
d_{-\Kscr}\left(h(x^*;\theta_i)+{\lambda_i\over \rho}\right)\leq d_{-\Kscr}\big(h(x^*;\theta_i)\big)+{\|\lambda_i\|\over \rho}.
\label{proof:primal_sub_opt_cons_rho3}
\end{align}}%
Substituting \eqref{proof:primal_sub_opt_cons_rho2}
and \eqref{proof:primal_sub_opt_cons_rho3} in \eqref{proof:primal_sub_opt_cons_rho1}, we obtain for all $i\geq 0$
{\
\begin{eqnarray}
\lefteqn{f(x_{i+1};\theta_i)-f(x^*;\theta_i)} \notag\\
& \leq &
%& {\rho \over 2}\left(d_{-\Kscr}\big(h(x^*;\theta_i)\big)+{\|\lambda_i\|\over \rho}\right)^2-{1 \over 2\rho}\|\lambda_{i+1}\|^2+\alpha_i\notag\\
 \tfrac{\rho}{2}~d^2_{-\Kscr}\big(h(x^*;\theta_i)\big)+\|\lambda_i\|~d_{-\Kscr}\big(h(x^*;\theta_i)\big)\nonumber \\
& & +\tfrac{1}{2\rho}\big(\|\lambda_i\|^2-\|\lambda_{i+1}\|^2\big)+\alpha_i.
\label{proof:primal_sub_opt_cons_rho5}
\end{eqnarray}}%
{According to Remark~\ref{rem:Lipschitz-h}, we have $\norm{
	h(x^*;\theta_i) - h(x^*;\theta^*)} \leq {L_{h,\theta}} \|\theta_i -
		\theta^*\|$; hence, we have
		$h(x^*;\theta_i)=h(x^*;\theta^*)+v_i$ for some $v_i\in
		L_{h,\theta}\norm{\theta_i-\theta^*}~\cB(\mathbf{0},1)$ for all
		$i\geq 0$. Using Lemma~\ref{Lemma:d_triangle}, for each $i\geq
		0$, we get $d_{-\cK}\big(h(x^*;\theta_i)\big)\leq
		d_{-\cK}\big(h(x^*;\theta^*)\big)+\norm{v_i}$; thus, since
		$x^*\in X^*(\theta^*)$, we have {$h(x^*;\theta^*)\ineq\mathbf{0}$}, and
		this implies
\begin{equation}
\label{proof:primal_sub_opt_cons_rho4}
d_{-\cK}\big(h(x^*;\theta_i)\big)\leq L_{h,\theta}\norm{\theta_i-\theta^*},
\end{equation}
for $i\geq 0$.}
%From Lipschitz continuity of $h_j$ in $\theta$ for $j=1,\hdots,m,$
%\begin{align*}
%h_j(x^*;\theta_i)\leq h_j(x^*;\theta^*)+L_{h}\|\theta_i-\theta^*\|;
%\end{align*}
%hence,
%\begin{align}
%d_{-\Kscr}\left(h(x^*;\theta_i)\right)\leq
%d_{-\Kscr}\left(h(x^*;\theta^*)\right)+L_{h}\|\theta_i-\theta^*\|.
%\label{proof:primal_sub_opt_cons_rho4}
%\end{align}
%Since $h_j(x^*;\theta^*)\leq 0$ for $j=1,\hdots,m$, it follows that
%$d_{\mathbb{R}^m_-}(h(x^*;\theta^*))=0$,
%and inequality~\eqref{proof:primal_sub_opt_cons_rho4} becomes
%\begin{align}
%d_{-\Kscr}\left(h(x^*;\theta_i)\right)\leq L_{h}\|\theta_i-\theta^*\|.\notag
%\end{align}
By substituting \eqref{proof:primal_sub_opt_cons_rho4} into~\eqref{proof:primal_sub_opt_cons_rho5}, we get for all $i\geq 0$
\begin{eqnarray}
\lefteqn{f(x_{i+1};\theta_i)-f(x^*;\theta_i)}\notag \\
&& \leq\tfrac{\rho}{2}{L^2_{h,\theta}}\|\theta_i-\theta^*\|^2+{L_{h,\theta}}\|\theta_i-\theta^*\|\|\lambda_i\|\nonumber\\
&&\quad +\tfrac{1}{2\rho}\left(\|\lambda_i\|^2-\|\lambda_{i+1}\|^2\right)+\alpha_i\notag\\
&&\leq\tfrac{\rho}{2}{L^2_{h,\theta}}\|\theta_i-\theta^*\|^2+\bar{C}{L_{h,\theta}}\|\theta_i-\theta^*\|\notag \\
&&\quad +\tfrac{1}{2\rho}\left(\|\lambda_i\|^2-\|\lambda_{i+1}\|^2\right)+\alpha_i,
\end{eqnarray}
where the last inequality follows from $\|\lambda_i-\lambda^*\| \leq C_{\lambda}$
(%Proposition
Theorem~\ref{Theorem: Aug_lambda_bnd_cont_rho}), i.e.,
	$\|\lambda_i\|\leq \bar{C}\triangleq C_\lambda+\|\lambda^*\|$ for all $i\geq
	0$. Next, from the Lipschitz continuity of $f$ in $\theta$ {by Assumption~\ref{Assump: Aug_1}.ii}, %it follows that
$$f(x_{i+1};\theta_i)-f(x^*;\theta_i)\geq f(x_{i+1};\theta^*)-f(x^*;\theta^*)-2{L_{f,\theta}}\|\theta_i-\theta^*\|.$$
Combining two above inequalities results in the following:
{\
\begin{align*}
\notag
f(x_{i+1};\theta^*)-f^* \leq &
\tfrac{\rho}{2}{L^2_{h,\theta}}\|\theta_i-\theta^*\|^2+\left(\bar{C}{L_{h,\theta}}+2{L_{f,\theta}}\right)\|\theta_i-\theta^*\|\\
& +{1\over 2\rho}\Big(\|\lambda_i\|^2-\|\lambda_{i+1}\|^2\Big)+\alpha_i.
\end{align*}}%
Summing the above inequality for $i=0$ to $k-1$, we obtain %the following:
{\small
\begin{align*}
\sum_{i=0}^{k-1}\Big(f(x_{i+1};\theta^*)-f^* \Big)
%&\leq\tfrac{\rho}{2}L^2_{h}\sum_{i=0}^k\left[ \|\theta_{i}-\theta^*\|^2+\left(\bar{C}L_{h}+2L_{f}\right) \|\theta_{i}-\theta^*\|\right]+\tfrac{1}{2\rho}\sum_{i=0}^k\left(\|\lambda_i\|^2-\|\lambda_{i+1}\|^2\right)+\sum_{i=1}^k \alpha_i\\
%&\leq\tfrac{\rho}{2}L^2_{h}\sum_{i=0}^\infty \|\theta_{i}-\theta^*\|^2+\Big[L_{h}(C_\lambda+\|\lambda^*\|)+2L_{f}\Big]\sum_{i=0}^\infty \|\theta_{i}-\theta^*\|\\
%&\quad+\tfrac{1}{2\rho}\|\lambda_0\|^2+\sum_{i=0}^\infty \alpha_i\\
\leq & \sum_{i=0}^\infty
\alpha_i{+\frac{1}{2\rho}\norm{\lambda_0}^2}
+\tfrac{\rho}{2}{L^2_{h,\theta}}\frac{\|\theta_0-\theta^*\|^2}{1-\tau^2} \\
	& +\left(\bar{C}{L_{h,\theta}}+2{L_{f,\theta}}\right)\frac{\|\theta_0-\theta^*\|}{1-\tau}%+\tfrac{1}{2\rho}\|\lambda_0\|^2
.
\end{align*}}%
Since $f(x;\theta^*)$ is convex in $x$, dividing both sides of the above inequality by $k$ gives the desired result.
%\begin{align*}
%&\quad f(\xbar_k;\theta^*)-f(x^*;\theta^*) \\
%	& \leq{1 \over k+1}\Big({\rho \over 2}L^2_{h,\theta}
%			{\|\theta_0-\theta^*\|^2\over
%			1-q^2_\ell}+(L_{h,\theta}C_\lambda\\
%	&+2L_{f,\theta}){\|\theta_0-\theta^*\|\over 1-\tau}+{1 \over 2\rho}\|\lambda_0\|^2+\sum_{i=0}^\infty \alpha_i\Big).
%\end{align*}
\end{IEEEproof}
\vspace*{-2mm}
%------------------------------------------------------------
\subsection{Convergence analysis for increasing $\{\rho_k\}$}
\label{sec:rate-increasing}
%------------------------------------------------------------
In 1976, Rockafellar~\cite{rockafellar:76:ala} proposed several different
variants of inexact augmented Lagrangian schemes where the penalty parameter could
be updated between iterations, and under suitable summability conditions on the sequence $\{\alpha_k \rho_k\}$, it was established that the sequence of dual iterates $\{\lambda_k\}$ is bounded. In addition, \emph{upper} bounds on primal
suboptimality and infeasibility were derived. Recently, Aybat and Iyengar~\cite{aybat2013augmented} extended this result to conic convex
programs, provided \emph{both} upper and lower bounds on the suboptimality, and presented sequences $\{\rho_k\}$ and $\{\alpha_k\}$ under which the
primal function converges linearly to its optimal value. Necoara et
al.~\cite{1506.05320} also considered similar inexact schemes for conic convex programs  where the penalty parameter $\rho_k$ is
tuned adaptively. In this scheme, the {authors} apply a search procedure which effectively finds an
upper bound on $\norm{\lambda_0-\lambda^*}$ in logarithmic number of steps {and} in each search step, {a subproblem of the form $\min_{x\in X}\cL_{\rho_k}(x,\lambda;\theta)$ is inexactly solved to compute an $\epsilon$-optimal solution, where $\cL_{\rho_k}$ is defined by setting $\rho=\rho_k$ in \eqref{Aug_L_forrho}}. The search stops if an $\epsilon$-feasible solution is computed; otherwise, $\rho_k$ is doubled and a new subproblem is solved. Note that in contrast to our AL scheme, in~\cite{1506.05320}, each subproblem is solved to $\epsilon$-optimality for some small $\epsilon>0$, which will cause numerical issues in its practical implementations. In what follows, we analyze the convergence properties of %Algorithm~\ref{Alg: ALM}
{IPALM} {for an increasing
penalty sequence}. %The scheme is presented in Algorithm~\ref{Alg: Aug_increase_rho}.

We initiate the analysis on convergence rate by first deriving a bound on primal infeasibility, which is subsequently used to derive bounds on primal suboptimality. These statements are then utilized in deriving the overall {computational} complexity in Section~\ref{sec:complexity-increasing}.
%-----------------------------------------------
%\subsection{Rate of convergence analysis}\label{sec:4.1}
%-----------------------------------------------
%For notational simplicity, let
Let $\{\rho_k\}_{k\geq 0}\subset\reals_{++}$. %recall according to
{For each $k\geq 0$, $\cL_{\rho_k}(x,\lambda;\theta)$ is defined by setting $\rho=\rho_k$ in \eqref{Aug_L_forrho},}
%\begin{align}
%\label{eq:Lk}
% %{\cal L}_k(x,\lambda;\theta)\triangleq
% \cL_{\rho_k}(x,\lambda;\theta)=f(x;\theta)+\frac{\rho_k}{2}d^2_{-\Kscr}\left( h(x;\theta)+\frac{\lambda}{\rho_k}\right)-{\|\lambda\|^2\over 2\rho_k},
% \end{align}
%and, ${g_k(\lambda;\theta)\triangleq g_{\rho_k}(\lambda;\theta)}$,
i.e., $g_{\rho_k}(\lambda;\theta)=\inf_{x\in X} {\cal
	L}_{\rho_k}(x,\lambda;\theta)$. Since $g_\rho(\cdot;\theta)$ is the
	Moreau envelope of $g_0(\cdot;\theta)$ for any $\rho>0$, {both
		$g_{\rho_k}(\cdot;\theta)$ and $g_{0}(\cdot;\theta)$ have the
			same set of maximizers for any fixed $\theta\in\Theta$,
		i.e., $\argmax_{\lambda}
		g_{\rho_k}(\lambda;\theta)=\argmax_{\lambda}
		g_0(\lambda;\theta)$, for all $k\geq 0$. Moreover, it is also
			true that $\max_{\lambda}
		g_{\rho_k}(\lambda;\theta)=\max_{\lambda} g_0(\lambda;\theta)$;
		hence,} $f^*=\max_{\lambda} g_{\rho_k}(\lambda;\theta^*)$ for
		all $k\geq 0$.  %We begin this section with a lemma
%{\us{the proof of which follows from Proposition~\ref{Theorem: Aug_lambda_bnd_cont_rho}}, \us{and shows}
First, we show that %the dual iterate sequence
$\{\lambda_k\}_{k\in\integers_+}$ stays bounded under Condition~\ref{Assump:rho_k}.
{
\begin{condition}
\label{Assump:rho_k} %The inexactness
{\it $\{\alpha_k,\rho_k\}$ satisfy
{$\sum_{k=0}^{\infty} \sqrt{\alpha_k\rho_k} <\infty$}.}
\end{condition}}%
\begin{lemma}
\label{Theorem: Aug_bnd_lambda_increa_rho}
Let Assumptions~\ref{Assump: Aug_1} and~\ref{Assump: learn_lin_rate} hold. Given $\lambda_0\in\cK^*$, let $\{\lambda_k\}_{k\in\integers_+}$ be the
%dual iterate  sequence generated by %Algorithm~\ref{Alg: ALM}
{IPALM} dual sequence when $\rho_k=\rho_0\beta^k$ for $k\geq 1$, where $\rho_0>0$ and $\beta>1$ such that $\beta \tau < 1$. %are positive scalars such that $\beta >1$%when the following condition holds:
%\begin{enumerate}
%\item[i)] ,
%\item[ii)] $\beta \tau < 1$, where {$\tau\in(0,1)$} is the constant defined in Assumption~\ref{Assump: learn_lin_rate},
%\item[iii)]$\sum_{k=0}^{\infty} (\alpha_k\rho_k)^{1\over2} <\infty$.
%\end{enumerate}
If %$\sum_{k=0}^{\infty} (\alpha_k\rho_k)^{1\over2} <\infty$, then for all $k\geq 0$, it follows that
{Condition~\ref{Assump:rho_k} holds}, then $\sup_k\|\lambda_k-\lambda^*\|\leq C'_\lambda$, where
%for $C'_\lambda$ defined as
{\
$$C'_\lambda \triangleq\sum_{i=0}^\infty\sqrt{2\alpha_i\rho_i}+{ \rho_0 {\bar{\kappa}}\|\theta_0-\theta^*\|\over 1-\beta \tau}+\|\lambda_0-\lambda^*\|,$$
}%
and $\lambda^*$ is any point in $\Lambda^*\triangleq \argmax_{\lambda} g_0(\lambda;\theta^*).$
\end{lemma}
\begin{IEEEproof}
The proof immediately follows from \eqref{eq:aux-ineq-lambda-bound} in the proof of %Proposition
Theorem~\ref{Theorem: Aug_lambda_bnd_cont_rho}
\end{IEEEproof}
%---------------------------------------------------

%---------------------------------------------------
Next, we derive a bound on the primal infeasibility.
%---------------------------------------------------
\begin{lemma}
\label{Theorem:primal_infeas_increa_rho}
{Under Assumption~\ref{Assump: Aug_1}, let $\{x_k,\lambda_k\}_{k\in\integers_+}$ be the primal-dual iterate sequence generated by %Algorithm~\ref{Alg: ALM}
{IPALM} for a given penalty sequence $\{\rho_k\}_{k\in\integers}$.} Then, for all $k\geq 0$, %it follows that
\begin{align*}
 d_{-\Kscr}\big(h(x_{k+1};\theta^*)\big)
 \leq  {\|\lambda_{k+1}-\lambda_k\|/ \rho_k} + {L_{h,\theta}} \|\theta_k-\theta^*\|. %,\quad \forall\ .
\end{align*}
\end{lemma}
%---------------------------------------------------
\begin{IEEEproof}
%According to
From Step~2 of %Algorithm~\ref{Alg: ALM}
{IPALM} %, in particular from
and \eqref{eq:lambda-relation}, %we have that
%{
%\begin{align*}
%\lambda_{k+1}-\lambda_k=\rho_k \nabla_{\lambda}{\cal L}_k(x_k,\lambda_k;\theta_{k})=\rho_k\left(h(x_k;\theta_{k})-\Pi_{-\Kscr}\left (h(x_k;\theta_{k})+{\lambda_k \over \rho_k}\right)\right),
%\end{align*}
%}
%where $\nabla_{\lambda} {\cal L}_k$ is defined in \eqref{Grad_l}.  Hence,
\begin{align*}
h(x_{k+1};\theta_{k}) & ={\lambda_{k+1}-\lambda_k \over \rho_k} +\Pi_{-\Kscr}\Big(h(x_{k+1};\theta_{k})+{\lambda_k \over \rho_k}\Big). %\label{proof:bnd_prime_infeas}
\end{align*}
Therefore, Lemma~\ref{Lemma:d_triangle} implies that
{\
\begin{align}
& \quad d_{-\Kscr}\big(h(x_{k+1};\theta_{k})\big) \label{proof:bnd_prime_infeas1}\\
& \notag \leq
{\|\lambda_{k+1}-\lambda_k\|\over
	\rho_k}  +d_{-\Kscr}\Big(\Pi_{-\Kscr}\big(h(x_{k+1};\theta_{k})+{\lambda_k \over \rho_k}\big)\Big).
%& \leq {\|\lambda_{k+1}-\lambda_k\|\over \rho_k},
\end{align}}%
{Note that %\eqref{proof:bnd_prime_infeas1} follows from
$d_{-\Kscr}\left(\Pi_{-\Kscr}\left (h(x_k;\theta_{k})+{\lambda_k \over \rho_k}\right) \right)=0$} since $\Pi_{-\cK}(y)\in-\cK$ for all $y\in\reals^m$ and $d_{-\cK}(y)=0$ for all $y\in -\cK$.
Under Assumption~\ref{Assump: Aug_1}.i, in particular from Remark~\ref{rem:Lipschitz-h}, we also have
$\norm{ h(x_{i+1};\theta_i) - h(x_{i+1};\theta^*)} \leq {L_{h,\theta}} \|\theta_i - \theta^*\|$ for all $i\geq 0$.
%Finally, from Lipschitz continuity of $h(x;\theta)$ in $\theta$ and by the
Therefore, from %the triangle inequality of
Lemma~\ref{Lemma:d_triangle}, %it follows that
\vspace*{-2mm}
 \begin{align*}
 d_{-\Kscr}\big(h(x_{k+1};\theta^*)\big)&\leq
	 d_{-\Kscr}\big(h(x_{k+1};\theta_{k})\big)+{L_{h,\theta}}
 \|\theta_k-\theta^*\|\\
	 & \leq  {\|\lambda_{k+1}-\lambda_k\|/\rho_k} +{L_{h,\theta}} \|\theta_k-\theta^*\|.
 \end{align*}
\end{IEEEproof}
%----------------------------------------------------
Our next {result} provides bounds on the primal sub-optimality.
%----------------------------------------------------
{\begin{lemma}
\label{Theorem: Aug_low_bnd_increa_rho}
%Let Assumption~\ref{Assump: Aug_1} and ~\ref{Assump: learn_lin_rate} hold and let $\{x_k\}$ and $\{\lambda_k\}$ be the sequences generated by Algorithm~\ref{Alg: Aug_increase_rho}. Then, the following hold:
{Under Assumption~\ref{Assump: Aug_1}, let $\{x_k,\lambda_k\}_{k\in\integers_+}$ be the primal-dual iterate sequence generated by %Algorithm~\ref{Alg: ALM}
{IPALM} for a given penalty sequence $\{\rho_k\}_{k\in\integers}$. Then, for all $k\geq 1$, %it follows that
\begin{align}
f(x_{k+1};\theta^*)-f^*
\geq & -{1\over \rho_k}\left(\|\lambda_{k+1}\|
		+\|\lambda_k-\lambda^*\|\right)^2\notag \\
& -\rho_k{L^2_{h,\theta}}\|\theta_k-\theta^*\|^2, \label{sub-1} \\
%& \geq {-1\over \rho_k}\left(\|\lambda_{k+1}\|+\|\lambda_{k}-\lambda^*\|\right)^2-L^2_{h}\|\theta_0-\theta^*\|^2\rho_k\tau^{2k} \label{sub-1} \\
f(x_{k+1};\theta^*)-f^*
%&\leq{\|\lambda_k\|^2\over \rho_k}+\alpha_k+L^2_{h}\|\theta_0-\theta^*\|^2\rho_k\tau^{2k}+2L_{f}\|\theta_0-\theta^*\|\tau^{k}. \label{sub-2}\\
\leq & {\|\lambda_k\|^2\over \rho_k}+2{L_{f,\theta}}\|\theta_{k}-\theta^*\| \notag \\
& + \rho_k {L^2_{h,\theta}}\|\theta_{k}-\theta^*\|^2+\alpha_k. \label{sub-2}
\end{align}}%
%for all $k\geq 1$, where $\lambda^*\in
%is any point in $
%\Lambda^*\triangleq\argmax_\lambda\, g_0(\lambda;\theta^*)$.
\end{lemma}
%\begin{IEEEproof}
%\sa{The proof is given in the extended version of the manuscript~\cite{ahmadi16_preprint}.}
%\end{IEEEproof}
%\begin{comment}
\begin{IEEEproof}
{We start by proving the lower bound in~\eqref{sub-1}, and then prove
	the upper bound in~\eqref{sub-2}.\\
\textbf{Proof of \eqref{sub-1}}: Since $f^*=\max_{\lambda}
	g_{\rho_k}(\lambda;\theta^*)$, $\lambda\in\Lambda^*$ implies %that
		$\lambda^*\in\argmax g_{\rho_k}(\lambda;\theta^*)$.
		Consequently, we have $f^*=\min_{x\in X} {\cal L}_{\rho_k}(x,\lambda^*;\theta^*)$. Hence, for all $k\geq 0$,}
{\small
\begin{align}
f^*&\leq {\cal
	L}_{\rho_k}(x_{k+1},\lambda^*;\theta^*)\label{proof:Aug_low_bnd_pri_opt_1} \\
	& =f(x_{k+1};\theta^*)+{\rho_k\over 2}d^2_{-\Kscr}\Big(h(x_{k+1};\theta^*)+{\lambda^*\over \rho_k}\Big)-{\|\lambda^*\|^2\over 2\rho_k}. \nonumber
%&\leq f(x_{k+1};\theta^*)+{\rho_k\over 2}d_{-\Kscr}\Big(h(x_{k+1};\theta^*)+{\lambda^*\over \rho_k}\Big)^2 ,\label{proof:Aug_low_bnd_pri_opt_1}
\end{align}
}
%where in the last inequality, we discard the last term since it is negative. Note that
%Once again, by
Using the triangle inequality of Lemma~\ref{Lemma:d_triangle}, we get
{\small
\begin{align}
\notag & {d_{-\Kscr}\big( h(x_{k+1};\theta^*)+{\lambda^*\over \rho_k }\big)}\\
\leq & \notag  d_{-\Kscr}\big( h(x_{k+1};\theta_k)+{\lambda_k\over \rho_k}\big)\\
& + \left\|h(x_{k+1};\theta^*) -h(x_{k+1};\theta_k)+{\lambda^*-\lambda_k \over \rho_k}\right\|\notag\\
%\leq & \left\|{\lambda_{k+1} \over \rho_k}\right\|+ \left\|h(x_{k+1};\theta_k)-h(x_{k+1};\theta^*)+{\lambda_k-\lambda^* \over \rho_k}\right\|\notag\\
 \leq & {1\over \rho_k}\big(\|\lambda_{k+1}\| +\|\lambda_k-\lambda^*\|\big)+L_{h,\theta}\|\theta_k-\theta^*\|\label{bound_on_d},
\end{align}}%
where for the second inequality, {we use the identity
$\norm{\lambda_{k+1}}=\rho_k~
	d_{-\cK}\big(h(x_{k+1};\theta_k)+\frac{\lambda_k}{\rho_k}\big)$,
	which follows from~\eqref{proof:bnd_prime_infeas1}, and
%in the third inequality, we
invoke the %triangular inequality along with
			Lipschitz continuity of function $h(x;\theta)$ in $\theta$ (see
		Remark~\ref{rem:Lipschitz-h}).}
Hence, using the identity $\|a+b\|^2\leq 2\|a\|^2+2\|b\|^2$ on~\eqref{bound_on_d} to bound $d^2_{-\Kscr}\Big( h(x_{k+1};\theta^*)+{\lambda^*\over \rho_k }\Big)$, and using the resulting bound within \eqref{proof:Aug_low_bnd_pri_opt_1}, we obtain%the desired lower bound in
~\eqref{sub-1}.

%we obtain
%\begin{align}
%d_{-\Kscr}\left( h(x_k;\theta^*)+{\lambda^*\over \rho_k }\right)^2\leq
%{2\over \rho^2_k}\left(\|\lambda_{k+1}\| +\|\lambda_k-\lambda^*\|\right)^2+2L^2_h\|\theta_k-\theta^*\|^2.\label{bound_on_d2}
%\end{align}
%By replacing~\eqref{bound_on_d2} into~\eqref{proof:Aug_low_bnd_pri_opt_1}, we obtain
%\begin{align*}
%f^*\leq f(x_{k+1};\theta^*)+{1\over \rho_k}\left(\|\lambda_{k+1}\| +\|\lambda_k-\lambda^*\|\right)^2+\rho_kL^2_h\|\theta_k-\theta^*\|^2.
%\end{align*}
%which completes the proof after rearranging the terms. \\
\noindent {\bf Proof of \eqref{sub-2}}:
%\begin{proposition}
%\label{Theorem: Aug_upp_bnd_increa_rho}
%Let Assumption~\ref{Assump: Aug_1} and ~\ref{Assump: learn_lin_rate} hold and let $\{x_k\}$ and $\{\lambda_k\}$ be the sequences generated by Algorithm~\ref{Alg: Aug_increase_rho}. Then, the following holds:
%\begin{align*}
%f(x_k;\theta_{k})-f^* &\leq{\|\lambda_k\|^2\over \rho_k}+\alpha_k+L^2_{h,\theta}\|\theta_0-\theta^*\|^2\rho_k\tau^{2k}+L_{f,\theta}\|\theta_0-\theta^*\|q_l^{k}.
%\end{align*}
%\end{proposition}
%\begin{IEEEproof}
Step 1 of %Algorithm~\ref{Alg: ALM}
{IPALM} implies that
\begin{align*}
 {\cal L}_{\rho_k}(x_{k+1},\lambda_k;\theta_{k}) & \leq  \inf_{x\in X}
{\cal L}_{\rho_k}(x,\lambda_k;\theta_{k})+\alpha_k\\
		& \leq  {\cal L}_{\rho_k}(x^*,\lambda_k;\theta_{k})+\alpha_k.
 \end{align*}
%From the definition of ${\cal L}_{\rho_k}$ {in \eqref{eq:Lk}}, and %using
\sa{Setting $\rho=\rho_k$ within \eqref{Aug_L_forrho} and using}
the fact that $d_{-\cK}(\cdot)\geq 0$, %we have %for any $y\in\reals^m$, we have that
%\begin{align*}
%f(x_k;\theta_{k})+{\rho_k\over 2}d_{-\Kscr}\Big(h(x_k;\theta_{k})+{\lambda_k\over \rho_k}\Big)^2-{\|\lambda_k\|^2\over 2\rho_k}\leq f(x^*;\theta_{k})+{\rho_k\over 2}d_{-\Kscr}\Big(h(x^*;\theta_{k})+{\lambda_k\over \rho_k}\Big)^2-{\|\lambda_k\|^2\over 2\rho_k}+\alpha_k,
%\end{align*}
%which, after rearrangement of terms, leads to inequality
{
\begin{align}
f(x_{k+1};\theta_{k})-f(x^*;\theta_{k}) %&\leq {\rho_k\over 2}d_{-\Kscr}\left(h(x^*;\theta_{k})+{\lambda_k\over \rho_k}\right)^2-{\rho_k\over 2}d_{-\Kscr}\left(h(x_k;\theta_{k})+{\lambda_k\over \rho_k}\right)^2+\alpha_k\notag\\
&\leq {\rho_k\over 2}d^2_{-\Kscr}\left(h(x^*;\theta_{k})+{\lambda_k\over
		\rho_k}\right)+\alpha_k.\label{proof:Aug_upp_bnd_pri_opt_1}
\end{align}
}
Using the triangle inequality of Lemma~\ref{Lemma:d_triangle} twice, we get
\begin{align}
d_{-\Kscr}\left( h(x^*;\theta_k)+{\lambda_k\over \rho_k }\right)
%&\leq d_{-\Kscr}\left( h(x^*;\theta^*)+{\lambda_k\over \rho_k }\right)+ \left\|h(x^*;\theta_k) -h(x^*;\theta^*)\right\|\notag\\
\leq & d_{-\Kscr}\left( h(x^*;\theta^*)\right)+{ \left\|\lambda_k
	\right\|\over \rho_k }\notag \\
& + \left\|h(x^*;\theta_k) -h(x^*;\theta^*)\right\|\notag\\
\leq & { \left\|\lambda_k \right\|\over \rho_k }+ L_{h,\theta}\|\theta_k -\theta^*\|,\label{proof:Aug_upp_bnd_pri_opt_2}
\end{align}
where %in the second inequality, we used triangular inequality and
in the third inequality we used the fact that $d_{-\Kscr}\left( h(x^*;\theta^*)\right)=0$ along with Lipschitz continuity of function $h(x^*;\theta)$ in $\theta$. By substituting~\eqref{proof:Aug_upp_bnd_pri_opt_2} into~\eqref{proof:Aug_upp_bnd_pri_opt_1} and again using the identity $\|a+b\|^2\leq 2\|a\|^2+2\|b\|^2$, we obtain %the following:
\begin{align}
f(x_{k+1};\theta_{k})-f(x^*;\theta_{k})
 \leq & \rho_k {L^2_{h,\theta}}\|\theta_{k}-\theta^*\|^2\notag \\
		 &+{\|\lambda_k\|^2\over \rho_k}+\alpha_k. \label{subopt-1}
\end{align}
%Since $f(x;\theta)$ is Lipschitz continuous in $\theta$ for all $x\in X$ by
Using Assumption~\ref{Assump: Aug_1}.ii and \eqref{subopt-1}, we bound each {difference}
		term in the right hand side of the following equality,
{\
\begin{align*}
f(x_{k+1};\theta^*)-f^*= & f(x_{k+1};\theta^*)-f(x_{k+1};\theta_k)\notag \\
	& +f(x_{k+1};\theta_{k})-f(x^*;\theta_{k})\notag \\
    & + f(x^*;\theta_{k})-f(x^*;\theta^*).
%&\leq 2L_{f}\|\theta_{k}-\theta^*\| + \rho_k L^2_{h}\|\theta_{k}-\theta^*\|^2+{\|\lambda_k\|^2\over \rho_k}+\alpha_k.
\end{align*}}%
to obtain the %desired
upper bound given in \eqref{sub-2}.\
%Finally, by invoking the bound on $\|\theta_k-\theta^*\|$, we obtain
%\begin{align*}
%f(x_k;\theta^*)-f^*\leq{\|\lambda_k\|^2\over \rho_k}+\alpha_k+L^2_{h}\|\theta_0-\theta^*\|^2\rho_k\tau^{2k}+2L_{f}\|\theta_0-\theta^*\|\tau^{k}.
%\end{align*}
\end{IEEEproof}
%\end{comment}
}
%-------------------------------------------
We conclude this subsection with a formal rate statement for the sub-optimality and the infeasibility of %the sequence
$\{x_k\}$. The proposed sequences of $\rho_k$ and $\alpha_k$ in
Theorem~\ref{Theorem:Aug_increas_rho_linear_rate} are extensions of those presented
in~\cite{aybat2013augmented} and take learning into consideration. %\us{\bf should be delete this algorithm since we presented first one in gereneral form}
\begin{theorem}
\label{Theorem:Aug_increas_rho_linear_rate}
{Under Assumptions~\ref{Assump: Aug_1} and \ref{Assump: learn_lin_rate},
let $\{x_k,\lambda_k\}_{k\in\integers_+}$ be the %primal-dual
iterate sequence generated by %Algorithm~\ref{Alg: ALM}
{IPALM} for the increasing penalty sequence $\{\rho_k\}_{k\in\integers}$ and inexact optimality parameter sequence $\{\alpha_k\}_{k\in\integers}$ defined as follows: given some $c,\alpha_0,\rho_0>0$, $\beta>1$ and $\tau\in(0,1)$ such that $\delta\triangleq\beta\tau<1$, let $\alpha_k \triangleq \alpha_0(k+1)^{-2(1+c)}\beta^{-k}, \rho_k \triangleq \rho_0 \beta^k, k\geq 0.$}
Then, %for all $k\geq 0$, %the following bounds hold:
{$\lambda^*\triangleq\lim_{k\rightarrow\infty}\bar{\lambda}_k$ exists; moreover, $\lambda^*\in\argmax_\lambda g_0(\lambda;\theta^*)$ and}
{\small
\begin{subequations} \label{rate-bound}
\begin{align}
& \mid f(x_{k+1};\theta^*)-f^* \mid \leq \ {B_k\over \beta^k}, \label{rate-bound-a} \\
&  {d_{-\Kscr} }\big( h(x_{k+1};\theta^*) \big) \leq  {1\over
	\beta^k}  \left( {2C'_\lambda \over \rho_0 }+ L_{h,\theta} \|\theta_0-\theta^*\|\delta^k \right), \label{rate-bound-b}
\end{align}
\end{subequations}}%
for all $k\geq 0$, where $C'_\lambda$ is %the constant
defined in Lemma~\ref{Theorem: Aug_bnd_lambda_increa_rho} and
%$B_k$ is defined as follows
{\small
\begin{align}
\label{eq:Bk}
B_k \triangleq  &\frac{\alpha_0}{(k+1)^{2(1+c)}}+{1\over\rho_0}(2C'_\lambda+\|\lambda^*\|)^2\nonumber\\
&+\rho_0\left(L_{h,\theta}\norm{\theta_0-\theta^*}\delta^k+\frac{{L_{f,\theta}}}{\rho_0L_{h,\theta}}\right)^2.
\end{align}
}%
\end{theorem}
\begin{IEEEproof}
{Let $\lambda^*\in\Lambda^*$ {denote} an arbitrary dual solution.}
From Assumption~\ref{Assump: learn_lin_rate}, for some
	$\tau\in(0,1)$, we have
		$\norm{\theta_k-\theta^*}\leq\norm{\theta_0-\theta^*}\tau^k$ for
		all $k\geq 0$. By hypothesis, $\rho_k = \rho_0 \beta^k$ for
		$k\geq 0$, $\delta\triangleq\beta\tau < 1$, and
		$\sum_{k=0}^\infty \sqrt{\alpha_k\rho_k}=\sum_{k=0}^\infty
		{\alpha_0\over k^{1+c}}< \infty$; therefore, the conditions of
	Lemma~\ref{Theorem: Aug_bnd_lambda_increa_rho} are satisfied and as
		a consequence, we have $\|\lambda_k-\lambda^*\|\leq
		C'_\lambda$ for all $k\geq 0$. Next, by combining the lower and
		upper bounds on primal suboptimality obtained in
		%Proposition
Lemma~\ref{Theorem: Aug_low_bnd_increa_rho}, we obtain
{\
\begin{align}
& \Big|f(x_{k+1};\theta^*)-f^*\Big| \notag \\
\leq & \max\Big\{{\|\lambda_k\|^2\over \rho_k}+2{L_{f,\theta}}\|\theta_k-\theta^*\|+\alpha_k,\notag \\
& {1\over \rho_k}\Big(\|\lambda_{k+1}\|+\|\lambda_{k}-\lambda^*\|\Big)^2\Big\}+ \rho_k{L^2_{h,\theta}}\|\theta_k-\theta^*\|^2\notag\\
\leq &
\frac{1}{\beta^k}\max\left\{\tfrac{1}{\rho_0}(C'_\lambda+\norm{\lambda^*})^2+2{L_{f,\theta}}\norm{\theta_0-\theta^*}\delta^k
	\right. \notag \\
& \left. +\tfrac{\alpha_0}{(k+1)^{2(1+c)}}\tfrac{1}{\rho_0}(2C'_\lambda+\norm{\lambda^*})^2\right\}+\frac{1}{\beta^k}\rho_0L_{h,\theta}^2\norm{\theta_0-\theta^*}^2\delta^{2k}\notag\\
\leq &
\frac{1}{\beta^k}\Big[\tfrac{1}{\rho_0}(2C'_\lambda+\norm{\lambda^*})^2+\rho_0L_{h,\theta}^2\norm{\theta_0-\theta^*}^2\delta^{2k}
\notag \\
	& +2{L_{f,\theta}}\norm{\theta_0-\theta^*}\delta^k+\tfrac{\alpha_0}{(k+1)^{2(1+c)}}\Big] \notag%\label{proof:primal_sub_opt}
%= &\max\Bigg\{{\|\lambda_k\|^2\over \rho_k}+\alpha_k,{1\over \rho_k}\Big(\|\lambda_{k+1}\|+\|\lambda_{k}-\lambda^*\|\Big)^2\Bigg\}\notag\\&+L^2_{h}\|\theta_0-\theta^*\|^2\rho_k\tau^{2k}+2L_{f}\|\theta_0-\theta^*\|\tau^{k}\notag\\
%\leq &{1\over \rho_k}\max\Big\{\|\lambda_k\|^2+\alpha_k\rho_k,\Big(\|\lambda_{k+1}\|+\|\lambda_{k}-\lambda^*\|\Big)^2\Big\}\notag\\&+{1\over \rho_k}\Big(L^2_{h}\|\theta_0-\theta^*\|^2\rho_k^2\tau^{2k}\Big)+{1\over \rho_k}\Big(2L_{f}\|\theta_0-\theta^*\|\rho_k\tau^{k}\Big).\label{proof:primal_sub_opt}
\end{align}}%
for all $k\geq 0$. Hence, completing the square in the last inequality, %\eqref{proof:primal_sub_opt},
we obtain the desired result. To prove the rate statement for infeasibility, we use %Proposition
Lemma~\ref{Theorem:primal_infeas_increa_rho} and Lemma~\ref{Theorem: Aug_bnd_lambda_increa_rho}:
{\
\begin{align*} %\notag
{d_{-\Kscr} }\big( h(x_k;\theta^*) \big)  & \leq
{\|\lambda_{k+1}-\lambda_k\| \over \rho_k} + L_{h,\theta}
\|\theta_k-\theta^*\|, \notag
%&\leq {\|\lambda_{k+1}-\lambda^*\| +\|\lambda_k-\lambda^*\| \over \rho_k} +L_{h,\theta} \|\theta_0-\theta^*\|\tau^k\notag\\
%& \leq \frac{1}{\beta^k}\left({2C'_\lambda \over \rho_0} +L_{h,\theta} \|\theta_0-\theta^*\|\delta^k\right).%\label{bound_h}
%\vspace*{-4mm}
\end{align*}}%
which implies the desired inequality in \eqref{rate-bound-b}.  {The proof of dual convergence is provided in the appendix. %electronic
% online supplement.
%(also see {extended version of the manuscript}~\cite{ahmadi16_preprint}).
}
%where in the last inequality we use the result of Lemma~\ref{Theorem: Aug_bnd_lambda_increa_rho} and definition of $\rho_k$. In addition, by the assumption, we have that $\beta \tau<1$. Therefore, we can further simplify ~\eqref{bound_h} as follows:
%\begin{align*}
%{d_{-\Kscr} }\left( h(x_k;\theta^*) \right) &\leq {2C'_\lambda \over \rho_0\beta^k} +L_{h,\theta} \|\theta_0-\theta^*\|{1\over \beta^k}\\
%&\leq  {1\over \beta_k}  \left( {2C'_\lambda \over \rho_0 }+ L_{h,\theta} \|\theta_0-\theta^*\| \right)
%\end{align*}
%{$\rho_k \tau^{k}= \rho_0 (\tau\beta)^{k}\leq \rho_0 (\tau\beta)\leq \rho_0$}
%and $\alpha_k \rho_k ={\rho_0\beta^k\over k^{2(1+c)}\beta^k}={\rho_0 \over k^{2(1+c)}} \leq \rho_0$. In addition,
%Hence, inequality~\eqref{proof:primal_sub_opt} can be simplified to
%\begin{align*}
%\Big|f(x_k,\theta_{k})-f^*\Big|
%&\leq {1\over \beta^k}\max\Big\{{1\over \rho_0}\|\lambda_k\|^2+1,{1\over \rho_0}\Big(\|\lambda_{k+1}\|+\|\lambda_{k}-\lambda^*\|\Big)^2\Big\}\\&+{1\over \beta^k}\Big(L^2_{h}\|\theta_0-\theta^*\|^2\rho_0\Big)+{1\over \beta^k}\Big(2L_{f}\|\theta_0-\theta^*\|\Big)\\
%&\leq {1\over \beta^k}\max\Big\{{1\over \rho_0}(C'_\lambda+\|\lambda^*\|)^2+1,{1\over \rho_0}(2C'_\lambda+\|\lambda^*\|)^2\Big\}\\&+{1\over \beta^k}\Big(L^2_{h}\|\theta_0-\theta^*\|^2\rho_0+2L_{f}\|\theta_0-\theta^*\|\Big)
%= {B_f \over \beta^k}.
%\end{align*}
%\vspace*{-6mm}
\end{IEEEproof}
\vspace*{-3mm}
%--------------------------------------------------------------
\section{Overall iteration complexity analysis}\label{sec:complexity}
%--------------------------------------------------------------
Implementing {the inexact augmented Lagragian algorithm} involves performing
inner and outer loops, where each outer loop corresponds to {one update of
the Lagrange multiplier} according to Step~2 in %Algorithm~\ref{Alg: ALM}
{IPALM}, while the inner loops correspond
to iterations of the scheme employed to compute {$x_{k+1}$ as in {Step~1} of {IPALM}}. Hence, to {assess} the overall
 {computational} complexity of {our inexact} augmented Lagrangian
 approach, it is essential to
specify the algorithm used for inner optimization. We assume that \vspace*{-2mm}
	$$f(x;\theta)\triangleq q(x;\theta)+p(x;\theta),$$
where {the functions $p$ and $q$ represent the \emph{smooth} and \emph{nonsmooth} parts of $f$}, respectively. Then, {according to}~\eqref{Aug_L_forrho},
\begin{align}
{\cal L}_\rho(x,\lambda;\theta)=q(x;\theta)+\nu_\rho(x,\lambda;\theta),\label{Aug_L_complex_rep}
\end{align}
where
$
{\nu_\rho(x,\lambda;\theta)}\triangleq p(x;\theta)+{\rho\over 2} {d^2_{-\Kscr}}\Big(\frac{\lambda}{ \rho}+h(x;\theta)\Big)-\frac{\|\lambda\|^2}{2\rho}.
$
{In the representation~\eqref{Aug_L_complex_rep}, the function $q$
	captures the nonsmooth part of augmented Lagrangian function while
		the function {$\nu_\rho$} represents the smooth part. This is a
		particular case of the \emph{composite} convex minimization
		problem studied in~\cite{beck2009fast,Nesterov13,Tseng08_1J}.
		In~\cite{beck2009fast,Nesterov13}, the authors developed
		Accelerated Proximal Gradient~(APG) methods, inspired by
		Nesterov's optimal scheme~\cite{nesterov2013introductory}}, that can compute an $\epsilon$-optimal solution to composite convex %optimization
problems within $\cO(1/\sqrt{\epsilon})$ iterations. %Algorithm~\ref{Alg: APG} displays an APG algorithm studied in~\cite{beck2009fast}.

{In what follows, we assume that the inner loop is solved by a
	particular implementation of the APG algorithm called Fast
		Iterative Shrinkage-Thresholding
		Algorithm~(FISTA)~\cite{beck2009fast}, displayed below in Algorithm~\ref{Alg: APG}. Implementing this scheme}
		requires that $\nabla_x \nu_\rho(x;\theta)$ be Lipschitz
			continuous in $x$ uniformly for all $\theta\in\Theta$. In this respect, {the next lemma states that $\nu_\rho(\cdot;\theta)$ is indeed smooth for any $\theta\in\Theta$ under Assumption~\ref{Assump: Lipsch_grad_nu}.}
%the conditions under which we may have such a property. Before presenting this result, we make the following assumptions.
\begin{assumption}
\label{Assump: Lipsch_grad_nu}
{Let $q,p:X\times\Theta\rightarrow\reals{\cup\{+\infty\}}$ be proper, closed, convex functions of $x$ for all $\theta\in\Theta$ such that $p(x;\theta)$ is differentiable in $x$ on an open set containing $X$ for any fixed $\theta\in\Theta$. Moreover,
 $\nabla_x p(\cdot;\theta)$ is Lipschitz continuous %in $x$
 with a uniform constant $L_{p,x}$ for all $\theta \in \Theta$ .}
\end{assumption}
Recall that under Assumption~1.i, function $h(\cdot;\theta)$ is Lipschitz continuous, %in $x$,
{uniformly for all $\theta \in \Theta$}, with constant $L_{h,x}\triangleq\max_{\theta\in\Theta}\|A(\theta)\|$.
{
\begin{lemma}[{\bf Lipschitz continuity of $\nabla_x \nu_\rho(x;\theta)$}]
\label{Lemma: Lips_cons_grad_nu}
{Under Assumption~\ref{Assump: Lipsch_grad_nu}, for any given $\theta \in \Theta$,} the gradient function $\nabla_x \nu_\rho(\cdot;\theta)$ is Lipschitz continuous %in $x$
with constant %$L_{\nu,x}(\rho,\theta)$ defined as
$$L_{\nu,x}(\rho,\theta) \triangleq L_{p,x}+\rho \|A(\theta)\|^2\leq L_{p,x}+\rho L_{h,x}^2.$$
%where $L_{h,x}:=\max_{\theta \in \Theta} \|A(\theta)\|.$
\end{lemma}
\begin{IEEEproof}
{Recall $h(x;\theta)=A(\theta)x+b(\theta)$ is an affine function in $x$ for all $\theta\in\Theta$. Since for any closed convex cone $\cK\in\reals^m$, $\Pi_{\Kscr^*}(x) = x-\Pi_{-\Kscr}(x)$ for all $x\in\reals^m$ and $\grad d_{-\cK}^2(y)=2(y-\Pi_{-\cK}(y))$ for all $y\in\reals^m$, the following holds:}
\begin{align*}
\nabla_x \nu(x,\lambda;\theta) %&=\nabla_x p(x;\theta)+ {\rho} {A(\theta)}^\top\left(h(x;\theta)+{\lambda\over \rho}-\Pi_{-\Kscr}\left(h(x;\theta)+{\lambda\over \rho}\right)\right)\\
=\nabla_x p(x;\theta)+ {\rho} {A(\theta)}^\top\Pi_{\Kscr^*}\big(h(x;\theta)+\tfrac{\lambda}{\rho}\big).
\end{align*}
Then, {by adding and subtracting terms and by invoking the triangle inequality and Lipschitz continuity, it} follows that for all $x,x' \in X$ and $\theta \in \Theta$,
{\
\begin{align}
& \|\nabla_x \nu(x,\lambda;\theta)-\nabla_x \nu(x',\lambda;\theta)\| \nonumber\\
=&\Big\|\nabla_x p(x;\theta)+ {\rho} {A(\theta)}^\top\Pi_{\Kscr^*}\big(h(x;\theta)+\tfrac{\lambda}{\rho}\big)\nonumber \\
& -\nabla_x p(x';\theta)- {\rho} {A(\theta)}^\top\Pi_{\Kscr^*}\big(h(x';\theta)+\tfrac{\lambda}{\rho}\big)\Big\|\notag\\
\leq & \left\|\nabla_x p(x;\theta)-\nabla_x p(x';\theta)\right\|\nonumber\\
& +\rho\left\|A(\theta)\right\| \Big\|\Pi_{\Kscr^*}\big(h(x;\theta)+\tfrac{\lambda}{\rho}\big) -\Pi_{\Kscr^*}\big(h(x';\theta)+\tfrac{\lambda}{\rho}\big)\Big\|\notag\\
\leq & L_{p,x} \|x-x'\|+\rho \left\|A(\theta)\right\|^2\left\|x-x'\right\|,\label{proof: Iter_complex_1}
\end{align}}%
{where the last inequality follows from the nonexpansivity of the projection.} %operator.}
\
\end{IEEEproof}
Given $x_{k}$ computed in the previous outer iteration of %Algorithm~\ref{Alg: ALM}
{IPALM}, we {employ} \emph{inner} iterations, displayed in Algorithm~\ref{Alg: APG}, to compute $x_{k+1}$, an $\alpha_k$-optimal solution to %an inexact solution, $x_{k+1}$, with accuracy $\alpha_k$ to the following optimization problem
\begin{equation}
\label{eq:inner-problem}
g_{\rho_k}(\lambda_k;\theta_k)=\min_{x\in X}\ q(x;\theta_k)+\nu_{\rho_k}(x,\lambda_k;\theta_k),
\end{equation}
{i.e., $x_{k+1}\in X$ and $\cL_{\rho_k}(x_{k+1},\lambda_k;\theta_k)\leq g_{\rho_k}(\lambda_k;\theta_k)+ \alpha_k$.} %is any point in $X$ that satisfies
%such that
%$q(x_{k+1};\theta_k)+\nu_{\rho_k}(x_{k+1},\lambda_k;\theta_k) \leq
%	g_{\rho_k}(\lambda_k,\theta_k)+ \alpha_k$.
\begin{algorithm}
\
\caption{APG$(x_k,\theta_k,\alpha_k,\rho_k)$ -- {\small implemented on the subproblem in~\eqref{eq:inner-problem}}}
{\textbf{Initialization:}\ Set $z_0\gets x_{k}$, $y_1\gets z_0$, $m_1\gets1$ and $t\gets 1$.}\\[1.5mm]
For all $t\geq 1$, update:
\label{Alg: APG}
{\small
\begin{enumerate}
\item[1.] {$z_{t}  \gets\argmin_{z\in X} \,\Big\{q(z;\theta_k)+\nabla_x \nu_{\rho_k}(y_t,\lambda_k;\theta_k)^\top (z-y_t)$
\text{ } \hspace*{2.5cm} $+ {L_{\nu,x}(\rho_k,\theta_k)\over 2} \|z-y_t\|^2\Big\}$} %\label{APG: update}\\
\item[2.] $m_{t+1} \gets\Big(1+\sqrt{1+4m_t^2}\Big)/2$ %\notag\\
\item[3.] $y_{t+1} \gets z_t+\Big({m_t-1 \over m_{t+1}}\Big)\Big(z_t-z_{t-1}\Big)$%\notag
\item[4.] If $t {\geq} T_k \ \triangleq \
	\sqrt{\tfrac{{8}L_{\nu,x}(\rho_k,\theta_k)}{\alpha_k}} D_x$, then STOP;\\ else $t \gets t+1$ and go to
	Step 1.
\end{enumerate}}%
\end{algorithm}

To obtain an accuracy of $\alpha_k$, at {most} $T_k\triangleq \sqrt{\tfrac{{8}L_{\nu,x}(\rho_k,\theta_k)}{\alpha_k}} D_x$ %steps of the
{APG (inner) iterations %scheme
are required %at epoch $k$
within the $k$-th (outer) iteration of IPALM;} %, where $T_k$ is bounded from below by $\sqrt{\tfrac{2L_{\nu,x}(\rho_k,\theta_k)}{\alpha_k}} D_x$;
this follows from the iteration complexity result of APG -- see~\cite[Theorem~4.4]{beck2009fast}. %for a proof.
\begin{lemma}
\label{Lemma: comp_compl_APG}
{Let %functions
$q$ and $p$ satisfy Assumption~\ref{Assump: Lipsch_grad_nu}.}
%a proper, closed, convex functions such that $\mbox{Dom}\, q$ is closed and $\nabla_x \nu(x,\lambda_k;\theta_k)$ is Lipschitz continuous in $x$ with constant $L_{\nu,x}(\theta_k)$.
Fix $\alpha_k >0$ and let $\{z_t,y_t\}$ denote the %sequence of
{APG} iterate sequence. %generated by the APG algorithm.
If {$x_{k+1}^*\in \argmin_{x\in X}\, q(x;\theta_k)+\nu_{\rho_k}(x,\lambda_k;\theta_k)$},
		 then %$q(z_t,\theta_k)+\nu_{\rho_k}(z_t,\lambda_k;\theta_k)
${{\cal L}_\rho(z_t,\lambda_k;\theta_k)} \leq g_{\rho_k}(\lambda_k;\theta_k)+\alpha_k$ whenever
{\
\begin{equation}
{t\geq \sqrt{\tfrac{2}{\alpha_k}L_{\nu,x}(\rho_k,\theta_k)}\|x_{k}-x_{k+1}^*\| -1}. \label{eq:stop_Lipschitz}
\end{equation}}%
\end{lemma}
\subsection{Overall iteration complexity for constant penalty $\rho$}
\label{sec:complexity-constant}
Next, we derive the overall iteration complexity of
%Algorithm~\ref{Alg: ALM}
{IPALM} in which APG is used for solving {the subproblem in
		\eqref{eq:inner-problem} to compute $x_k$ satisfying Step~1 in %Algorithm~\ref{Alg: ALM}
{IPALM}.
}
{\begin{theorem}
\label{Theo: overall complex const rho}
{Let Assumptions~\ref{Assump: Aug_1},~\ref{Assump:
	learn_lin_rate} and~\ref{Assump: Lipsch_grad_nu}
hold}, and {let $\{x_k,\lambda_k\}$ denote the primal-dual iterate sequence} generated by %Algorithm~\ref{Alg: ALM}
{IPALM} {when $\{\alpha_k\}$ is chosen as $\alpha_k = {\alpha_0 \over (k+1)^{2(1+c)}}$ for $k\geq 0$ for some $\alpha_0>0$ and $c>0$}. Then,
for all $\epsilon\in(0,1)$, there exists a {$k(\epsilon)\in\integers_+$} such that {$|f(\bar{x}_k;\theta^*)-f^*|\leq
\epsilon$ and $d_{-\cK}\big(h(\bar{x}_k;\theta^*)\big)\leq\epsilon$} for
	all {$k\geq k(\epsilon)=\cO(\epsilon^{-2})$} {and requires at most $O\big(\epsilon^{-4}\big)$ proximal-gradient computations as shown in Step~1 of %Algorithm~\ref{Alg: APG}
{APG}}, where $\bar{x}_k\triangleq\tfrac{1}{k}{\sum_{i=1}^k x_i}$ for $k\geq 1$.
\end{theorem}
}
\begin{IEEEproof}
{To simplify the notation, %throughout the proof,
let
	$\gamma\triangleq\sum_{i=0}^\infty\sqrt{\alpha_k}<+\infty$ and $\eta
		\triangleq {\bar{\kappa}}~{\|\theta_0-\theta^*\|\over 1-\tau}$}.
{According to %Proposition
Theorems~\ref{Theorem: Aug_lambda_bnd_cont_rho} and~\ref{Theorem:Aug_bnd_dual_sub},
	$\norm{\lambda_k-\lambda^*}\leq C_\lambda$ for all $k\geq 1$, where}
%{\
%\begin{align} \notag
$C_\lambda \ \triangleq \ \sqrt{2\rho}~\gamma+\rho~\eta+ \| \lambda_0-\lambda^*\|$ and {$\lambda^*\in\Lambda^*$ such that $\lambda^*=\lim_k \bar{\lambda}_k$.}
%\end{align}}%
%In addition,
{%According to
From Theorem~\ref{Theorem:Aug_bnd_dual_sub}, $0\leq f^*-g_\rho(\bar{\lambda}_k;\theta^*)\leq B_g/k$ for all $k\geq 1$, where}
{\
\begin{align}
B_{g} \ \triangleq \ {\tfrac{1}{2} \rho}~\|\lambda_0-\lambda^*\|^2+C_{\lambda}\big(\sqrt{\tfrac{2}{\rho}}~\gamma+\eta\big). \label{Rate_ineq_1}
\end{align}}%
%{and $\bar{\lambda}_k\triangleq {1\over k}{\sum_{i=1}^k \lambda_i}$ for $k\geq 1$.}
By choosing $\alpha_0>0$ {appropriately}, we can generate an $\{\alpha_k\}$ sequence such that $\gamma={1\over \sqrt{2\rho}}$. Hence,
\begin{subequations}\label{B_g}
\begin{align}
C_\lambda & =1+\rho~\eta+\|\lambda_0-\lambda^*\|\\
B_{g} & = \tfrac{1}{2\rho}\|\lambda_0-\lambda^*\|^2+C_{\lambda}\big(\tfrac{1}{\rho}+\eta\big).
\end{align}
\end{subequations}
{Moreover, according to Theorem~\ref{Theorem:Aug_prim_infeas_cons_rho}, $d_{-\cK}\big(h(\bar{x}_k;\theta^*)\big)\leq\cV(k)$ for all $k\geq 1$, where}
{\small
\begin{align*}
{\cal V}(k) = {C_1\over \sqrt{k}} +{C_2 \over k},\quad C_1 \
	\triangleq \  \sqrt{{2B_g\over \rho}+\left(C_\lambda\over
			\rho\right)^2 },\quad C_2  \ \triangleq \  {1\over \rho} +\mu,
\end{align*}}%
and $\mu\triangleq{ {(L_{h,\theta}+\bar{\kappa})} \|\theta_0-\theta^*\| \over 1-\tau}$.
Using~\eqref{B_g}, we expand $C_1^2$ as %follows:
{\
\begin{align*}
C_1^2&={1\over \rho^2}\|\lambda_0-\lambda^*\|^2+{2\over
	\rho}\left({1\over \rho}+\eta\right)C_\lambda
	+\left(\frac{C_\lambda}{\rho}\right)^2 \\
%&={1\over \rho^2}\|\lambda_0-\lambda^*\|^2+{2\over \rho}\left({1\over \rho}+\eta\right)\left(1+\rho~\eta+\norm{\lambda_0-\lambda^*}\right) +\frac{1}{\rho^2}{\left(1+\rho~\eta+\norm{\lambda_0-\lambda^*}\right)^2}\\
& ={1\over \rho^2}\left[ 2{\left(1+\rho~\eta+\norm{\lambda_0-\lambda^*}\right)^2}+(1+\rho~\eta)^2\right].
\end{align*}}%
Since $\sqrt{a + b } \leq \sqrt{a} + \sqrt{b}$ for $a, b \geq 0$, it follows that
%{\begin{align*}
$C_1\leq  (\sqrt{2}+1) \left(\frac{1}{\rho}+\eta\right)+\frac{\sqrt{2}}{\rho}\norm{\lambda_0-\lambda^*}$.
%\end{align*}}%
Hence, we can {bound} ${\cal V}(k)$ from above by %in terms of problem parameters
{\
\begin{align}
{\cal V}(k) \leq
%& \frac{1}{\sqrt{k}} \left(
%		\frac{1}{\rho}\left(\sqrt{2}+1+\sqrt{2}\|\lambda_0-\lambda^*\|\right)
%		+ \eta (\sqrt{2}+1)\right) \notag \\
%	& + \frac{1}{k}
%\left(\frac{1}{\rho}+\mu\right) \nonumber \\
& {1\over \rho} \Big(  \frac{1}{k} +
		\frac{\sqrt{2}(\|\lambda_0-\lambda^*\|+1)+1}{\sqrt{k}}
		\Big)\notag \\
& +\frac{\mu}{k}+\frac{(\sqrt{2}+1)~\eta}{\sqrt{k}}.\label{v-ineq1}
%&\leq {1\over \rho}\left( \sqrt{{A_{\lambda^*}\over (k+1)} }+{1\over (k+1) }\right)+{C_\theta\over \sqrt{k+1}} + \frac{\mu_\theta}{k+1},
\end{align}}%
{To further simplify the notation, let $\Gamma_\theta\triangleq(\sqrt{2}+1)\eta$ and $\Gamma_\lambda\triangleq \sqrt{2}(\|\lambda_0-\lambda^*\|+1)+1$. Next, given $\epsilon\in(0,1)$, let
$$\rho =\rho_o\epsilon^{-r},\qquad k(\epsilon,s)\triangleq \lceil k_o\epsilon^{-s}\rceil$$
for some fixed $r,s\geq 0$ and $\rho_o, k_o>0$, where $\rho_o$ and
	$k_o$ depend only on problem parameters and are independent of %solution accuracy
$\epsilon$. From~\eqref{v-ineq1}, it immediately follows that for all $k\geq k(\epsilon,s)$}
{\
\begin{align}
{\cal V}(k)\leq & {1\over \rho_o\epsilon^{-r}}\left({1\over
		k_o\epsilon^{-s}}+{\Gamma_\lambda\over \sqrt{k_o\epsilon^{-s}}
		\sqrt{k_o\epsilon^{-s}} }\right)\notag \\
	& +{\mu \over k_o\epsilon^{-s}}+{\Gamma_\theta\over {\sqrt{k_o}\epsilon^{-s\over 2}}}\notag\\
= & {1\over \rho_ok_o}~\epsilon^{r+s}+{\Gamma_\lambda\over
	\rho_o\sqrt{k_o}}~\epsilon^{r+{s\over 2}}+{\mu \over
		k_o}~\epsilon^{s}+{\Gamma_\theta\over
			\sqrt{k_o}}~\epsilon^{s\over 2}\notag \\
	\triangleq & \bar{\cV}(\epsilon,r,s,\rho_o,k_o).\label{Rate_ineq3}
%&\leq \left(\sqrt{\Gamma_\lambda\over \rho_o^2k_o}+{1\over k_o\rho_o}+{\mu\over k_o}+{\Gamma_\theta\over \sqrt{k_o}} \epsilon^{s\over 2},\label{Rate_ineq3}
\end{align}}%
%the second inequality being a consequence of
{Both $\epsilon^{r} < 1$ and $\epsilon^{s/2} < 1$ for any $r,s\geq 0$ since $\epsilon\in(0,1)$.}

Since $\cV(k)$ decreases with $\cO(1/\sqrt{k})$ rate, the suboptimality
bounds obtained in Theorem~\ref{Theorem: Aug_prim_low_bnd_cons_rho}
satisfy ${\rho\over 2} {\cal V}^2(k)+\|\lambda^*\|{\cal V}(k)\geq U/k$
for all $k\geq \lceil k_o\epsilon^{-s}\rceil$ when $\epsilon \ \in
(0,1)$ is {sufficiently small. Although the proof can be written for all
$\epsilon\in(0,1)$,  we assume that  $\epsilon>0$ is sufficiently small to
simplify the notation;} %in the rest of this section;
therefore, {$|f(\bar{x}_k;\theta^*)-f^*|\leq
\epsilon$ and $d_{-\cK}\big(h(\bar{x}_k;\theta^*)\big)\leq\epsilon$ for all $k\geq k(\epsilon,s)$ whenever $k$ also satisfies}
%On the other hand, to obtain $\epsilon-$primal {sub-optimality}, according to the bound obtained in Theorem~\ref{Theorem: Aug_prim_low_bnd_cons_rho}, it suffices to have
\begin{equation}
\label{eq:condition-Vk}
{\rho\over 2} {\cal V}^2(k)+\|\lambda^*\|{\cal V}(k)\leq \epsilon,\quad {\cal V}(k)\leq \epsilon.
\end{equation}
{Clearly, \eqref{eq:condition-Vk} holds whenever} $(\|\lambda^*\|+1){\cal V}(k) \leq
{\epsilon \over 2}$ and ${\rho \over 2} {\cal
V}^2(k)\leq {\epsilon \over 2}$; hence, using \eqref{Rate_ineq3}, we can conclude
that %\uss{for $\epsilon \in (0,1)$},
{given sufficiently small $\epsilon>0$},
 $\bar{x}_k$ is $\epsilon$-optimal and $\epsilon$-feasible for all $k\geq k(\epsilon,s)$ if $r,s\geq 0$ and $\rho_o,k_o>0$ satisfy
 %the following sufficient condition
{%\
\begin{equation}
\label{eq:sufficient-cond}
\bar{\cV}(\epsilon,r,s,\rho_o,k_o)\leq\min\left\{\frac{\epsilon^{\frac{r+1}{2}}}{\sqrt{\rho_o}},~\frac{\epsilon}{2\|\lambda^*\|+2}\right\}.
\end{equation}}%
%\begin{align}
% {\cal V}(k)\leq {\epsilon^{{r+1\over 2} }\over \sqrt{\rho_o}}\quad\mbox{and}\quad
% {\cal V}(k) \leq {\epsilon \over 2\|\lambda^*\|}.\label{Rate_ineq4}
% \end{align}
{{Next, among all $r,s\geq0$ and $\rho_o,k_o>0$ satisfying this
sufficient condition, we investigate the optimal choice that minimizes
	the overall computational complexity corresponding to
	$k(\epsilon,s)$ outer iterations to achieve $\epsilon$-accuracy.
	According to Step~1 in %Algorithm~\ref{Alg: ALM}
{IPALM}, we need to ensure
	$\alpha_k$ level accuracy {in function values} at iteration $k$
	of the outer loop}; {hence, according to Lemma~\ref{Lemma:
		comp_compl_APG}, starting from $x_k$,} we
			need to perform {at most}
	$\sqrt{{\tfrac{2}{\alpha_k}{L_{\nu,x}(\rho,\theta_k)}}}~\|{x_k-x_{k+1}^*}\|$
		{inner iterations, where $x_{k+1}^*$ is an arbitrary optimal
			solution to \eqref{eq:inner-problem}.} {%Recalling that
Since $\|x-x'\| \leq 2D_x$ for any $x,x' \in X$ (see
					Remark~\ref{rem:Lipschitz-h}), the number of inner
					iterations within the $k$-th outer loop can bounded
					as follows for $k\geq 0$,}
 {\
 \begin{align*}
& \quad \sqrt{\tfrac{2}{\alpha_k}L_{\nu,x}(\rho,\theta_k)}{\|x_{k}-x_{k+1}^*\|}
	 \leq\sqrt{{8(L_{p,x}+\rho~L_{h,x}^2)\over {\alpha_0 \over
													 {(k+1)}^{2(1+c)}
												 }}} {D_x} \\
 &
  =  \sqrt {\tfrac{8}{\alpha_0}(L_{p,x}+\rho~L_{h,x}^2)}{D_x}{(k+1)}^{1+c}.
 \end{align*}}%
%Recall that
Since $\alpha_0$ is chosen such that $\gamma=\sum_{k=0}^{\infty}\sqrt{\alpha_k}={1\over \sqrt{2\rho}}$, %. Hence,
%we have
{\
\begin{align*}
\frac{1}{\sqrt{2\rho}}
& =\sum_{k=0}^{\infty}\sqrt{\alpha_k}=\sqrt{\alpha_0}\sum_{k=1}^{\infty}\left(\frac{1}{k}\right)^{1+c}
 \\
& \leq \sqrt{\alpha_0}\left[1+\int_{t=1}^{\infty}\left(\frac{1}{t}\right)^{1+c}~dt\right]=\sqrt{\alpha_0}\left(1+\frac{1}{c}\right).
\end{align*}}%
Hence, for $c\in(0,1)$, we get $\frac{1}{\alpha_0}\leq\frac{8\rho}{c^2}$. %Moreover,
We also have %$\sum\limits_{k=0}^{\lceil k_o\epsilon^{-s}\rceil}~(k+1)^{1+c}\leq\int\limits_0^{\lceil k_o\epsilon^{-s}\rceil}(t+1)^{1+c}~dt$; therefore,
$\sum\limits_{k=0}^{k(\epsilon,s)}~(k+1)^{1+c}\leq\int\limits_0^{k(\epsilon,s)}(t+1)^{1+c}~dt.$ Therefore, the total number of inner iterations to obtain $\epsilon$ accuracy is bounded by
{\small
\begin{align}
& \frac{8D_x}{c}\sqrt{\rho~(L_{p,x}+\rho
		L_{h,x}^2)}~\sum_{k=0}^{k(\epsilon,s)}~(k+1)^{1+c} \notag \\
\leq & \frac{8D_x}{c(2+c)}\sqrt{\rho_oL_{p,x}+(\rho_oL_{h,x})^2\epsilon^{-r}}
(\lceil k_o\epsilon^{-s}\rceil+1)^{2+c}~\epsilon^{-r/2},
%\leq{2\sqrt{L_{p,x}\rho_o}\over c(2+c)} {D}\epsilon^{-r\over 2} (k_o\epsilon^{-s})^{2+c} +{\sqrt{2}L_{h,x}\rho_o\over c(2+c)}{D}\epsilon^{-r} (k_o\epsilon^{-s})^{2+c}\notag\\
%&\leq {\sqrt{2}L_{h,x}\rho_o\over c(2+c)}{D}k_o^{2+c}\epsilon^{-r-2s-cs},
\label{complex analysis}
\end{align}}%
where $\rho=\rho_o\epsilon^{-r}$ and $k(\epsilon,s) = \lceil k_o
\epsilon^{-s}\rceil$. %and $\epsilon^{-r/2} \leq \epsilon^{-r}$.
{{According to \eqref{Rate_ineq3}, for \eqref{eq:sufficient-cond} to hold for all sufficiently small $\epsilon$, we require $s \geq \max\{r+1,~2\}$.}
	%\uss{we have that
%	\begin{align*}
%	& \quad r+s \geq r+1, r+\frac{s}{2} \geq r+1, r+s \geq 1, r+\frac{s}{2}
%	\geq 1, s \geq \frac{r+1}{2}, \frac{s}{2} \geq \frac{r+1}{2}, s \geq 1,
%	\frac{s}{2} \geq 1\\
%	& \implies s \geq \max(r+1,2).\end{align*}}	
		Hence, {the best achievable rate} is obtained when $r=0$ and $s=2$, which results in $\tilde{\cal O}(\epsilon^{-4})$ rate. Indeed, choosing $r=0$, $s=2$, $\rho_o\geq 4(\norm{\lambda^*}+1)^2$ and $k_o$ such that $\sqrt{\rho_o}~k_o-(\Gamma_\lambda+\Gamma_\theta\rho_o)~\sqrt{k_o}-(1+\mu\rho_o)\geq 0$ satisfies the sufficient condition in \eqref{eq:sufficient-cond}.}}
%Hence, for any given sufficiently small $\epsilon>0$, if we choose $r,s\geq 0$, and $\rho_o,k_o>0$ such that
% \begin{align}
%\left(\sqrt{\Gamma_\lambda\over \rho_o^2k_o}+{1\over k_o\rho_o}+{\mu\over k_o}+{\Gamma_\theta\over \sqrt{k_o}} \right) \epsilon^{s\over 2}\leq \min\left\{{\epsilon^{{r+1\over 2} }\over \sqrt{\rho_o}},~{\epsilon \over 2\|\lambda^*\|}\right\}, \label{Rate_cond1}
% \end{align}
% %and
%%  \begin{align}
%%\left(\sqrt{A_\lambda^*\over \rho_o^2k_o}+{1\over k_o\rho_o}+{\mu_\theta\over k_o}+{C_\theta\over \sqrt{k_o}} \right) \epsilon^{s\over 2}\leq ,\label{Rate_cond2}
%% \end{align}
%\sa{then according to~\eqref{Rate_ineq3}, $|f(\bar{x}_k;\theta^*)-f^*|\leq
%\epsilon$ will be satisfied for all $k\geq \lceil k_o\epsilon^{-s}\rceil$}. In order to satisfy~\eqref{Rate_cond1}, it {suffices} to choose $s,r\geq 0$ and $\rho_o,k_o>0$ such that $s \geq \max\{2,1+r\}$, and
%   \begin{align}
%\left(\sqrt{\Gamma_\lambda\over \rho_o^2k_o}+{1\over k_o\rho_o}+{\mu\over k_o}+{\Gamma_\theta\over \sqrt{k_o}} \right) \min\left\{ {1 \over 2\|\lambda^*\|},{1\over \sqrt{\rho_o}}\right\}.
% \end{align}
\
\end{IEEEproof}
We conclude with a corollary that specifies the iteration complexity
associated %with the perfectly specified problem in which
when $\theta_0 =
\theta^*${; hence, Assumption~\ref{Assump: learn_lin_rate} implies
	that $\theta_k=\theta^*$ for all $k\geq 0$.}
{
\begin{corollary}
\label{Corol:Constant rho no learn complex}
Suppose $\theta_0=\theta^*$ in %Algorithm~\ref{Alg: ALM}
{IPALM}. Under the same assumptions as stated in Theorem~\ref{Theo: overall complex const rho},
for all sufficiently small $\epsilon>0$, there exist $\rho=\cO(1/\epsilon)$ and {$k_o\in\integers_+$}, independent of $\epsilon>0$, such that {$|f(\bar{x}_k;\theta^*)-f^*|\leq
\epsilon$ and $d_{-\cK}\big(h(\bar{x}_k;\theta^*)\big)\leq\epsilon$} for all $k\geq k_o$ {which requires at most $O\big(\epsilon^{-1}\big)$ proximal-gradient computations as shown in Step~1 of %Algorithm~\ref{Alg: APG}
{APG}}. %where $\bar{x}_k\triangleq\tfrac{1}{k}{\sum_{i=1}^k x_i}$ for $k\geq 0$.
\end{corollary}
\begin{IEEEproof}
{When $\theta_0=\theta^*$, we have $\Gamma_\theta=\mu=0$.}
%and as a result, the upper bound in~\eqref{Rate_ineq3} can be made tighter as below}
%\begin{align*}
%\bar{\cV}(\epsilon,r,s,\rho_o,k_o)={1\over \rho_ok_o}~\epsilon^{r+s}+{\Gamma_\lambda\over \rho_o\sqrt{k_o}}~\epsilon^{r+{s\over 2}}.
%\end{align*}
%Consequently, the condition in~\eqref{eq:sufficient-cond} can be simplified to
%\begin{align}
%\label{eq:sufficient-cond2}
%{1\over \rho_ok_o}~\epsilon^{r+s}+{\Gamma_\lambda\over \rho_o\sqrt{k_o}}~\epsilon^{r+{s\over 2}}\leq \min\left\{\frac{\epsilon^{\frac{r+1}{2}}}{\sqrt{\rho_o}},~\frac{\epsilon}{2\|\lambda^*\|+2}\right\}.
% \end{align}
{Moreover, according to \eqref{Rate_ineq3} and
	\eqref{eq:sufficient-cond}, $s \geq 2(1-r)$; hence, we obtain {the
		best achievable rate} for the overall iteration complexity shown
		in \eqref{complex analysis} when $r=1$ and $s=0$, which results
		in ${\cal O}(\epsilon^{-1})$ rate. Indeed, choosing $r=1$,
		   $s=0$, $\rho_o\geq 4(\norm{\lambda^*}+1)^2$ and $k_o$ such
			   that $\sqrt{\rho_o}~k_o-\Gamma_\lambda~\sqrt{k_o}-1\geq
			   0$ satisfies the new sufficient condition in
			   \eqref{eq:sufficient-cond}.}
%Therefore, choosing $r,s\geq 0$ and $\rho_o,k_o>0$ such that
%\begin{equation}
%s\geq \max\{1-r,2-2r\},\qquad
%\left(\sqrt{\Gamma_\lambda\over \rho_o^2k_o}+{1\over k_o\rho_o}\right) \leq \min\left\{ {1 \over 2\|\lambda^*\|},{1\over \sqrt{2\rho_o}}\right\},
%\end{equation}
%and following~, the best rate is achieved if we choose $r=1$ and $s=0$, yielding the rate ${\cal O}(\epsilon^{-1})$
\
\end{IEEEproof}
}
%Next, we will derive an upper bound on $\alpha_0$, which will be used next when we discuss the implications of Corollary~\ref{Corol:Constant rho no learn complex}.
%\begin{equation}
%\label{eq:aux-bound}
%\frac{1}{\sqrt{2\rho}}=\sum_{k=0}^{\infty}\sqrt{\alpha_k}=\sqrt{\alpha_0}\sum_{k=1}^{\infty}\left(\frac{1}{k}\right)^{1+c}\geq \sqrt{\alpha_0}\int_{t=1}^{\infty}\left(\frac{1}{t}\right)^{1+c}~dt=\sqrt{\alpha_0}~\frac{1}{c}.
%\end{equation}
\begin{remark}
{Suppose we set $\lambda_0=\mathbf{0}$. The proof of Corollary~\ref{Corol:Constant rho no learn complex} shows that %$\rho_o\geq\max\{4(\norm{\lambda^*}+1)^2,~2(\norm{\Delta\lambda}+\sqrt{2}+1)^2\}$,
for $\rho_o\geq4(\norm{\lambda^*}+2)^2$, $k_o=1$ satisfies the sufficient condition, $\sqrt{\rho_o}~k_o-\Gamma_\lambda~\sqrt{k_o}-1\geq 0$. Therefore, setting $\rho=\rho_o/\epsilon$ and computing one outer iteration is suffcient. Indeed, according to \eqref{complex analysis}, implementing APG on $\min_{x\in X}\cL_\rho(x,\lambda_0;\theta^*)=f(x;\theta^*)+\frac{\rho}{2}d^2_{-\cK}(h(x;\theta^*))$ will generate an $\epsilon$-optimal and $\epsilon$-feasible solution to the original problem $\cC(\theta^*)$ within $\cO(D_x \rho_oL_{h,x}~\frac{1}{\epsilon})$ iterations.
%and \eqref{eq:aux-bound},
%setting $\alpha_0=\frac{1}{3\rho}$ and choosing $c>0$ appropriately implies that
}
\end{remark}
{
\subsection{Overall iteration complexity for increasing $\{\rho_k\}$}\label{sec:complexity-increasing} As shown in
Theorem~\ref{Theorem:Aug_increas_rho_linear_rate}, {when
	$\rho_k=\rho_0\beta^{k}$ for some $\beta>1$ and $\rho_0>0$,} we
	obtain a {geometric rate of convergence of sub-optimality
	error in terms of outer iterations of
%Algorithm~\ref{Alg: ALM}
{IPALM} in the form of $B_k/\beta^k$
{such that $B_{k+1}<B_k$ for all $k$; hence, $\sup_k B_k<+\infty$}
%$B_k \leq \bar B$ for all $k$
}}. %by using the parameters specified in the theorem.
While this may seem to be promising at first glance, it should be noted that as $k$
increases, {$\rho_k$ also increases geometrically; hence, $L_{\nu,x}(\rho_k,\theta_k)$, the Lipschitz constant of $\grad_x\nu_{\rho_k}(\cdot;\theta_k)$ increases at a geometric rate as well (see Lemma~\ref{Lemma: Lips_cons_grad_nu}), which adversely affects the convergence rate of APG (see Lemma~\ref{Lemma: comp_compl_APG}).} Therefore, increasing $\{\rho_k\}$ has two distinct effects: on one side, compared to constant $\rho$, increasing $\{\rho_k\}$
increases the rate of convergence of outer iteration from ${\cal O}({1\over
\sqrt{k}})$ to ${\cal O}({1\over \beta^k})$; while on the other hand, it also increases the
complexity of the inner computation. The following theorem derives the overall
iteration complexity of %Algorithm~\ref{Alg: ALM}
{IPALM} for the increasing penalty sequence.
{\begin{theorem}
\label{Theo: Overall Iter Comp Increas rho}
Let Assumptions~\ref{Assump: Aug_1},~\ref{Assump: learn_lin_rate} and~\ref{Assump: Lipsch_grad_nu} hold. {Let $\{x_k,\lambda_k\}$ be the primal-dual iterate sequence} generated by %Algorithm~\ref{Alg: ALM}
{IPALM} for the parameter sequences $\{\alpha_k\}_{k\in\integers}$ and $\{\rho_k\}_{k\in\integers}$ as defined in Theorem~\ref{Theorem:Aug_increas_rho_linear_rate}. Then,
for all $\epsilon>0$, there exists a {$k(\epsilon)\in\integers_+$} such that {$|f(x_{k};\theta^*)-f^*|\leq
\epsilon$ and $d_{-\cK}\big(h(x_k;\theta^*)\big)\leq\epsilon$} for all {$k\geq k(\epsilon)=\cO(\log(\epsilon^{-1}))$} which requires at most $O\big(\epsilon^{-1}\log(\epsilon^{-1})\big)$ proximal-gradient computations as shown in Step~1 of %Algorithm~\ref{Alg: APG}
{APG}.
\end{theorem}
\begin{IEEEproof}
{Suppose $\{\alpha_k\}_{k\in\integers}$ and $\{\rho_k\}_{k\in\integers}$ chosen as in Theorem~\ref{Theorem:Aug_increas_rho_linear_rate}. The inequality $(2C'_\lambda+\max\{\|\lambda^*\|,~1\}\big)^2\geq 2C'_\lambda$ together with \eqref{rate-bound} and \eqref{eq:Bk}  %clearly
implies for all $k\geq 0$ \vspace*{-5mm}

{\
\begin{align*}
& {\max\left\{|f(x_{k+1};\theta^*)-f^*|,\ d_{-\Kscr}\big(
		h(x_{k+1};\theta^*)\big)\right\}}{\beta^k} \\
\leq & {1\over \rho_0}\big(2C'_\lambda+\max\{\|\lambda^*\|,~1\}\big)^2+\frac{\alpha_0}{(k+1)^{2(1+c)}}\\
& + \rho_0\left(L_{h,\theta}\norm{\theta_0-\theta^*}\delta^k +{\tfrac{1}{\rho_0}\max\big\{\tfrac{L_{f,\theta}}{L_{h,\theta}},~1\big\}}\right)^2\triangleq\bar{B}_k.
\end{align*}}%

\noindent Moreover, $0<\bar{B}_k\leq \bar{B}_0$ for all $k\geq 0$. Therefore, for
any given $\epsilon>0$,
	Theorem~\ref{Theorem:Aug_increas_rho_linear_rate} implies that it
	takes at most $k(\epsilon)=\log_{\beta} \big({\bar{B}_0 \over
			\epsilon}\big)$ outer iterations to achieve
	$\epsilon$-optimal and $\epsilon$-feasible primal solution.}
	According to Step~1 in %Algorithm~\ref{Alg: ALM}
{IPALM}, we need to ensure
	$\alpha_k$ level accuracy {in function values} at iteration $k$
	of the outer loop; {hence, according to Lemma~\ref{Lemma:
		comp_compl_APG}, starting from $x_k$,} we need %to perform
		%\uss{at least}
		$\sqrt{{\tfrac{2}{\alpha_k}{L_{\nu,x}(\rho_k,\theta_k)}}}~\|{x_k-x_{k+1}^*}\|$
{inner iterations, where $x_{k+1}^*$ is an %arbitrary
optimal solution to
	\eqref{eq:inner-problem}.} {%Recalling that
Since $\|x-x'\| \leq 2D_x$ for any $x,x' \in X$ (see Remark~\ref{rem:Lipschitz-h}), the number of inner iterations within the $k$-th outer loop can bounded as} \vspace*{-2mm}

{\
 \begin{align*}
& \quad \sqrt{{\tfrac{2}{\alpha_k}{L_{\nu,x}(\rho_k,\theta_k)}}}~\|{x_k-x_{k+1}^*}\|\\
 %&\leq \sqrt{{2(L_{p,x}+\rho_0\beta^kL_{h,x}^2)\over {\alpha_0 \over \beta^k (k+1)^{2(1+c)} }}}D_x \\
&\leq \beta^k (k+1)^{1+c}  \underbrace{\sqrt{
	\tfrac{2}{\alpha_0}\big(\tfrac{L_{p,x}}{\beta^k } + {\rho_0}L_{h,x}^2\big)}~D_x}_{\triangleq \ M_k},\quad \forall k\geq 0.
%&\leq \beta^k k^{1+c}  \sqrt{ \left( {2L_{p,x}  } + 2{\rho_0 } L_{h,x}^2\right)}D \\
%& =  M_r\beta^k k^{1+c},
 \end{align*}}%
%Let $M_k\triangleq \sqrt{ \frac{2}{\alpha_0}\left({L_{p,x} \over \beta^k } + {\rho_0 }L_{h,x}^2\right)}~D_x$ for $k\geq 0$.
Clearly,
	$0<M_k\leq M_0$ for all $k\geq 0$; hence, the %overall
number of inner iterations to obtain $\epsilon$ accuracy can be bounded as
%above as follows:
{\small
 \begin{align*}
\sum_{k=0}^{\log_{\beta} \big({\bar{B}_0 \over \epsilon}\big)}
M_0\beta^{k} (k+1)^{1+c}
%& \leq M_0\left(\log_{\beta} \left({\bar{B}_0 \over \epsilon}\right)+1\right)^{(1+c)}\ \sum_{k=0}^{\log_{\beta} \big({\bar{B}_0 \over \epsilon}\big)} \beta^{k}  \\
%& =
%M_0 \left(\log_{\beta} \left({\bar{B}_0\over \epsilon}\right)+1\right)^{1+c} \frac{ \beta^{\log_{\beta} \left({\bar{B}_0\over \epsilon}\right) + 1}-1}{\beta-1}\\
\leq
M_0 \big(\log_{\beta} \big(\tfrac{\bar{B}_0}{\epsilon}\big)+1\big)^{1+c} \tfrac{ \beta }{\beta-1} \big(\tfrac{\bar{B}_0}{\epsilon}\big).
 \end{align*}}%
 \
 \end{IEEEproof}
 {In following remark, we note that when $\theta^*$ is known, i.e., when $\theta_0=\theta^*$, then the order of overall iteration complexity of %Algorithm~\ref{Alg: ALM}
 {IPALM} remains unchanged, i.e., $\tilde{O}\Big(\epsilon^{-1} \log(\epsilon^{-1})\Big)$ as in the case when learning is involved. However, there is a reduction in the $\cO(1)$ constant.
 \begin{remark}
 Under the %same assumptions as stated in
 premise of Theorem~\ref{Theo: Overall Iter
	 Comp Increas rho}, when $\theta_0=\theta^*$ in %Algorithm~\ref{Alg: ALM}
{IPALM}, the bounds in \eqref{rate-bound} can be modified as follows:\\
$|f(x_{k+1},\theta^*)-f^*| \leq \frac{1}{\beta^k}\left[\tfrac{1}{\rho_0}(2C'_\lambda+\norm{\lambda^*})^2+\alpha_0\right]$
and $d_{-\Kscr}\left( h(x_{k+1};\theta^*) \right) \leq  {1\over \beta^k}~{2C'_\lambda \over \rho_0}$;
%\begin{align*}
% \mid f(x_{k+1},\theta^*)-f^* \mid \ & \leq \
%	 \frac{1}{\beta^k}\left[\tfrac{1}{\rho_0}(2C'_\lambda+\norm{\lambda^*})^2+\alpha_0\right],\\\quad
%\mbox{ and }\quad
%{d_{-\Kscr} }\left( h(x_{k+1};\theta^*) \right) & \leq  {1\over \beta^k}~{2C'_\lambda \over \rho_0 };
%\end{align*}
%therefore
as in Theorem~\ref{Theo: Overall Iter Comp Increas rho}, the %overall
number of inner iterations to obtain $\epsilon$ accuracy can be bounded %above
by \vspace*{-3mm}
{\small
\begin{align*}
& \quad \sum_{k=0}^{\log_{\beta} \big({B' \over \epsilon}\big)}
M'\beta^{k} (k+1)^{1+c}\\
& \leq M' \big(\log_{\beta} \big(\tfrac{B'}{\epsilon}\big)+1\big)^{1+c} \tfrac{ \beta }{\beta-1}
\big(\tfrac{B'}{\epsilon}\big)\
%{\beta \over 1-\beta}~{M'B'\over \epsilon}\left(\log_{\beta} \left({B'\over \epsilon}\right)+1\right)^{1+c}
 = \tilde{\cal O }\big(\epsilon^{-1} \log(\epsilon^{-1})\big),
\end{align*}}%
where $B'=\tfrac{1}{\rho_0}\big(2C'_\lambda+\max\{\|\lambda^*\|,~1\}\big)^2+\alpha_0$,
	in which %$C'_{\lambda}$
{the constant {$C'_\lambda$}} (see Lemma~\ref{Theorem: Aug_bnd_lambda_increa_rho}) reduces to $C'_\lambda =\sum_{i=0}^\infty\sqrt{2\rho_i\alpha_i}+\|\lambda_0-\lambda^*\|$, and $M'=\sqrt{\frac{2}{\alpha_0}(L_{p,x}+\rho_0\norm{A(\theta^*)}^2)}$.
 \end{remark}}
 \vspace*{-2mm}
\section{Numerical Results}\label{sec:numerical}
 \vspace*{-1mm}
%In this section,
We now present %a description of the
the {\it misspecified} portfolio optimization {problem} %in Section~\ref{sec:51},
and {examine} the empirical behavior % and define the problem parameters in Section~\ref{sec:52}.
%The empirical performance of
of {IPALM.} %the proposed algorithms.
\vspace*{-4mm} %in Section~\ref{sec:53}.
\subsection{Problem description}\label{sec:51}
%\vspace*{-1mm} %In this subsection,
Here {we describe
the %{(misspecified)}
computational and learning problems}.

\noindent {\bf Computational Problem:} We consider the Markowitz portfolio
optimization problem~\cite{JOFI:JOFI1525}, where
$\{\mathcal{R}_i\}_{i=1}^n$ denote the {\it random} returns for $n$
financial assets. {We assume} that the joint distribution of aggregated
	return (given by
		  $\mathcal{R}=[\mathcal{R}_i]_{i=1}^n$) is a multivariate
		  normal distribution, ${\cal N}(\mu^0,\Sigma^0)$, with mean vector
		  $\R^n\ni\mu^o\triangleq \mathbb{E}[\mathcal{R}]$ and covariance matrix
		  $\R^{n\times
			  n}\ni\Sigma^o\triangleq
			  \mathbb{E}[(\mathcal{R}-\mu^o)^\top(\mathcal{R}-\mu^o)]$.
			  We assume that $\Sigma^o=[\sigma_{ij}]_{1\leq i,j\leq n}$
			  is \emph{positive definite}, implying that there are no
			  redundant assets in our collection. Suppose
			  $x_i\in\R$ denotes the proportion of asset $i$ in the
			  portfolio held throughout the given period. Hence,
			  $x=[x_i]_{i=1}^n\in \mathbb{R}^n$ such that $\sum_{i=1}^n
			  x_i=1$ and $x_i\geq 0$ for all $i = 1,\hdots,n$
			  corresponds to a portfolio with {\it no short selling}.
			  Practitioners often use additional constraints to reduce
{\it sector risk} by grouping together investments in securities of a
sector and setting a limit on the exposure to this
sector~\cite{9780511753886}. Suppose there are $s$ sectors and $m_j$ is
the maximum proportion of the portfolio that {may} be invested in sector
$j$ for $j=1,\ldots,s$. Let $I_j\subset\{1,\ldots,n\}$ be the set of
indices corresponding to assets belonging to sector $j$ for
$j=1,\ldots,s$. Note that the same asset {may} belong to
{multiple} sectors; hence, %we do {\it not} assume that
$\{I_j\}_{j=1}^s$ is {\it not necessarily} a
partition. These sector constraints {given by} $\sum_{i\in I_j}
x_i\leq m_j, \hbox{ for } j=1,\ldots,s$ {can be compactly
	written as}
%Clearly, the above set of constraints can be represented by the matrix notation
$Ax\leq b$, where
$b\triangleq [m_1,\hdots,m_s]^\top$ and %$A\in \R^{s\times n}$ such that
$A_{ji}=1$ if asset $i$ belongs to sector $j$, $0$ otherwise. %, and is $0$ otherwise.
As per Markowitz theory, the optimal portfolio is derived by solving
%maximizing
%the expected portfolio return less the risk (captured by the variance) as specified by
%To decide the optimal portfolio, we face two competing objectives: minimize
%the risk, i.e., the variance of the portfolio return, and maximize the
%expected portfolio return. This portfolio selection model is called
%If the expected return of asset $A_i$ is $\mu_i$ and the covariance between returns of asset $A_i$ and $A_j$ is $\sigma_{ij}$ , then we denote the $n\times n$ covariance matrix associated with the assets $A_1,\hdots,A_n$ by
%$\Sigma:= (\sigma_{ij})_{1\leq i,j\leq n}$ and the expected return by $\mu := (\mu_i)_{i=1}^n$.
%the Markowitz portfolio optimization problem;
{\small
\begin{align}
\big({\cal C}(\Sigma)\big):\
\min_{x\in\R^n} \ \left\{\tfrac{1}{2} x^\top\Sigma x-\k \mu^\top x:\ Ax\leq
b,\ x\in X\right\},
\end{align}}%
{where $\k>0$ is a trade-off parameter {between the expected
	portfolio return and the risk (captured by the variance),} $\mu$ and $\Sigma$ {are
	estimates of $\mu^o$ and $\Sigma^o$, respectively}; and  %it can be written as follows:
$X \triangleq \{x\in\R^n:\ \sum_{i=1}^n x_i=1,\ x\geq 0\}$}.
%where $\k$ is a given positive parameter that represents the trade-off between return and risk.
{A %\us{lower}
lower value of $\k$ %corresponds
leads to a ``risk averse" portfolio while
larger values  correspond to ``risk seeking" ones.}
{One may solve} %compute an optimal solution of
${\cal C}(\Sigma)$ {by} %one may use
constrained optimization techniques when Euclidean projections onto the
polyhedral set, defined by $Ax\leq b$ and $x\in X$, cannot be computed
efficiently. Note that when $\{I_j\}_{j=1}^s$ is not a partition, i.e.,
	$I_j\cap I_k\neq \emptyset$ for some $1\leq j\neq k\leq s$,
%it may usually be not efficient to
{it would not be efficient to} compute projections at each iteration. Hence, one may overcome {this challenge} %the projection requirement
	by relaxing %the constraint
$Ax\leq b$, and adopting an augmented Lagrangian scheme {-- to verify Assumption~1.iii, it is enough to note that $\grad^2 f(x;\Sigma^*)=\Sigma^*\succ\mathbf{0}$ and $\grad f(x;\cdot)$ is Lipschitz continuous for all $x\in X$ with $L_{F,\theta}=1$; hence, Theorem~\ref{thm:pseudoLipschitz} implies Assumption~1.iii.}

If the {\it true} values {of the parameters} $\mu^o$ and $\Sigma^o$ are known, then
the Markowitz
problem is just a convex quadratic optimization problem over a polyhedral set. However, knowing the
true values of $\mu^o$ and $\Sigma^o$ often cannot be taken for granted.
In fact, even the estimation of these parameters is generally {not an easy task}.
%under some prior distributions is not easy in practice, since the historical data
%sets are generally quite large.
In this section, we consider a setting where true $\mu^o$ vector is
specified, i.e., $\mu=\mu^o$, and the true covariance matrix $\Sigma^o$
is {\it unknown}; but, it can be computed as the optimal solution to a
suitably defined {\it learning problem}.

\noindent {\bf Learning Problem:} Given sample returns for $n$ assets
{with} sample size $p$ for each asset, let $S =
(s_{ij})_{1\leq i,j \leq n}\in\reals^{n\times n}$ denote the sample covariance matrix. In
practice, %we usually have
$p\ll n$, {implying} that the number of
assets is far greater than the sample size. Since $p < n$, $S$ cannot
be positive definite; on the other hand, $\Sigma^o\succ
\mathbf{0}\in\R^{n\times n}$. Hence, instead of using $S$ as our true
covariance estimator, we consider the sparse covariance selection (SCS)
	problem, proposed in~\cite{SCS_Xue}, as our learning problem and
		is defined below:
\begin{align*}
\Sigma^*\triangleq\argmin_{\Sigma\in
	\mathbb{S}^n}\left\{\tfrac{1}{2}\|\Sigma-S\|_F^2+\upsilon|\Sigma|_1:\ \Sigma
	\succeq \epsilon I\right\},
\end{align*}
where $\upsilon$ and $\epsilon$ are positive regularization parameters,
	  $\mathbb{S}^n$ denotes the set of $n\times n$ symmetric matrices,
	  $\|\cdot\|_F$ is the Frobenius norm, $|\Sigma|_1$ is the $\ell_1$
	  norm of the vector formed by all off-diagonal elements of $\Sigma$, %treated as a column vector
and $\Sigma \succeq \epsilon I$ implies that all eigenvalues of $\Sigma$ are
	  greater than equal to $\epsilon>0$. Notice that the constraint
%in this problem
guarantees that the estimate $\Sigma^*$ is positive
	  definite and the $\ell_1$ regularization term in the objective
	  promotes sparsity in $\Sigma^*$. Therefore, the optimal solution
	  $\Sigma^*$ will satisfy our full-rank assumption on the covariance
	  matrix. Lack of such properties may lead to  undesirable
	  under-estimation of risk in the high dimensional Markowitz problem
	  and also may cause the corresponding optimization problem to be
	  ill defined~\cite{SCS_Xue}. Hence, we assume $\Sigma^*$ can
	  be safely used to approximate $\Sigma^o$.  We choose to solve the SCS learning problem using an ADMM algorithm.
%In order
To apply this scheme, we adopt a variant of the formulation %presented
in~\cite{SCS_Xue}: %we first introduce a new variable $\Phi$ and an equality constraint as follows:
%\vspace*{-1mm}
\begin{align}
\tag{SCS}\
\min_{\Sigma,\Phi\in \mathbb{S}^n}\{1_Q(\Sigma)+\tfrac{1}{2}\|\Sigma-S\|_F^2+\upsilon|\Phi|_1:\ \Sigma=\Phi\}
\vspace*{-3mm}
\end{align}
where $1_Q$ denotes the indicator function of $Q \triangleq \{\Sigma \in
\mathbb{S}^n:\quad \Sigma \succeq \epsilon I \}$. Now, (SCS) {allows for developing
	an efficient} ADMM scheme. Let
$\gamma:\mathbb{S}^n\rightarrow \R$ such that $\dom(\gamma)=Q$ and
$\gamma(\Sigma)=\tfrac{1}{2}\|\Sigma-S\|_F^2$. Since $\gamma$ is
strongly convex with a Lipschitz continuous gradient over its
domain, then the ADMM algorithm generates a sequence
$\{\Sigma_k\}$,
guaranteed to converge at a linear rate to the optimal
	solution $\Sigma^*$~\cite{Dang_Yin_ADMM_Str_Cnv}, i.e.,
	$\|\Sigma_k-\Sigma^*\|\leq \tau^k \|\Sigma_0-\Sigma^*\|$, for some
	$\tau\in(0,1)$. Hence, Assumption~\ref{Assump: learn_lin_rate} is
	satisfied.%\\[2mm]
%\subsection{Problem parameters}\label{sec:52}

\noindent {\bf Problem parameters:} Suppose there are $n=1500$ available assets from $s=10$ sectors. %Size of each sector is set to $50$ and assets are randomly assigned to each sector.
{The learning problem is constructed by %by construct the learning problem, we first
	generating} the true
%parameters, namely the
mean return $\mu^o$ and covariance matrix
$\Sigma^o=[\sigma_{ij}]_{1\leq i,j\leq n}$ based on the following rules:
$\mu^o$ is generated from a uniform distribution over the hyperbox
$[-1,1]^n\in\R^n$ while
	$\sigma_{ij}=\max\{1-(|i-j|/10),0\}$. Next, we generate $p$
	i.i.d. samples of random returns $\{\mathcal{R}_t\}_{t=1}^p$ from
	{${\cal N}(\mu^0,\Sigma^0)$},
%a multivariate normal distribution with mean $\mu^o$ and covariance matrix $\Sigma^o$
	where the sample size is set to $p={n\over 2}$.
	Sample returns are then used to calculate the sample covariance matrix
	$S$. %Given that
Since $\mu^0$ is known, $\Sigma^*$ is the solution to
	the learning problem (SCS)
	with $\upsilon=0.4$. Consequently, the optimal portfolio, $x^*\in
	X$, is the solution to ${\cal C}(\Sigma^*)$, where $\k= 0.1$.\vspace*{-2mm}
\subsection{Empirical analysis of performance}\label{sec:53}
{In this
subsection, we {study} constant
and increasing penalty parameter schemes, {and} conclude with a
discussion on how the simultaneous schemes compare with
their sequential counterparts. In all tables, CPU times are reported in
\emph{seconds}. Recall that at iteration $k$, given $\alpha_k, \rho_k, x_k,$
and $\theta_k$, starting from $x_{k}\in X$, %Algorithm~\ref{Alg: APG}
{APG} is employed to find an $\alpha_k$-optimal solution $x_{k+1}\in X$ via inner loop iterations. Lemma~\ref{Lemma: comp_compl_APG} provides a bound on the number of APG iterations to obtain such an $x_{k+1} \in X$; % that guarantees %to obtain
%such $x_{k+1}\in X$, %As stated in the lemma,
%{This} number is bounded above by
%\us{such a solution}, is
specifically, no more than $\sqrt{2L^k_{\nu,x} / \alpha_k}~\|x_{k}-x^*_{k+1}\|$ APG iterations are required where $x_{k+1}^* \in \argmin_{x\in X} {\cal
	L}_{\rho_k}(x,\lambda_k;\Sigma_k)$ and $L^k_{\nu,x}\triangleq \sigma_{\max}(\Sigma_k)+\rho_k \sigma^2_{\max}(A)$
denotes the Lipschitz constant of %the gradient of
{$\grad_x {\cal L}_{\rho_k}(x,\lambda_k;\Sigma_k)$}.
%with respect to $x$.
%We recall that at iteration $k$, given $\alpha_k, \rho_k, x_k,$ and $\theta_k$, Algorithm~\ref{Alg: APG} produces an iterate $x_{k+1}$
Hence, this can be achieved in practice by proceeding through {$T_k$ APG iterations,} where $T_k\triangleq
\sqrt{8L^k_{\nu,x}/\alpha_0}~{D_x}{(k+1)}^{1+c}$, implying
$\alpha_k$-optimality. Suppose given a tolerance $\epsilon>0$ for suboptimality
and infeasibility, %Algorithm~1
{IPALM} requires $K$ outer iterations to terminate. In
all tables below, we use $\texttt{iter}_\mathrm{L}\triangleq\sum_{k=1}^KT_k$ to
denote the \emph{cumulative} number of APG iterations until termination over
$K$ outer iterations. That said, to evaluate the tightness of
$\texttt{iter}_\mathrm{L}$, we also directly check the $\alpha_k$-optimality
condition in Step~1 of %Algorithm~\ref{Alg: ALM}
{IPALM} to stop APG iterations earlier
(at \emph{most} $T_k$ iterations needed). This can be done by computing the
optimal value, $g_{\rho_k}(\lambda_k;\theta_k)$, for the $k$-th subproblem to
evaluate the suboptimality for each inner (APG) iterate computed within $k$-th
outer loop; we use $\texttt{iter}_\mathrm{\alpha}$ to denote the
\emph{cumulative} number of APG iterations over all outer iterations as a
result of checking $\alpha_k$ condition \emph{directly} -- note that the number
of outer iterations may be different. Clearly,
$\texttt{iter}_\mathrm{\alpha}\leq \texttt{iter}_\mathrm{L}$. Although, this
\emph{cannot} be implemented in practice, still the gap
$\texttt{gap}\triangleq\texttt{iter}_\mathrm{L}-\texttt{iter}_\mathrm{\alpha}$
provides with valuable information: if it is large, it hints that there is a
room for improvement on the cumulative number of APG iterations using other
sufficient conditions implying $\alpha_k$-optimality, e.g., based on
subgradient norm~\cite{aybat2013augmented}.} {Next we define the error metrics that are reported in the tables below: for suboptimality $\texttt{s}(\cdot) \triangleq \frac{|f(\cdot;\Sigma^*)-f^*|}{|f^*|}$, for infeasibility $\texttt{infs}(\cdot) \triangleq d_{\R^m_-}(A\cdot-b)$ and for learning error $\texttt{le}(\cdot) \triangleq {\|\cdot -\Sigma^*\|}/{\|\Sigma^*\|}$.}
\\[1mm]
{\bf a. Constant penalty parameter $\rho$.} {In the first set of
	experiments, we assume that $\Sigma^*$ is known, implying {a
{\it perfectly} specified problem}, and investigate the performance
	%of the inexact AL scheme (
			%Algorithm~\ref{Alg: ALM}
{IPALM} using %the sequence %of inexactness
		{$\{\alpha_k\}$} as stated in Theorem~\ref{Theo: overall complex const
			rho}.  Recall that, given $\epsilon$, Theorem~\ref{Theo:
				overall complex const rho} and
				Corollary~\ref{Corol:Constant rho no learn complex}
	%suggest
{prescribe the penalty parameter} $\rho={\rho_0\over \epsilon}$ and sequence
		of inexactness $\{\alpha_k\}$, where $\alpha_k
		={\alpha_0k^{-2(1+c)}}$ %in order
		to obtain the best
		overall iteration complexity of $O( {1 \over \epsilon})$
		{and}
	%	$\rho$ and $\alpha_k,$
  $\alpha_0$ {satisfies} $\sum_{k=0}^\infty \sqrt{\alpha_k}
	={1\over \sqrt{2\rho}}$. We choose $\rho_0=1$ and $c=1$e-$3$. Table~\ref{Table1} details the sub-optimality, infeasibility and
	the %computational
	effort to obtain an $\epsilon$-optimal and $\epsilon-$feasible solution for various values of $\epsilon$ when $\Sigma^*$ is available. Next, we compare the results in Table~\ref{Table1} with those
	obtained by {implementing %Algorithm~\ref{Alg: ALM}
{IPALM} using the learning sequence $\{\Sigma_k\}$
for misspecified parameter $\Sigma^*$}.
%According to Theorem~\ref{Theo: overall complex const rho}, to obtain \us{the} best overall iteration complexity
%of ${\cal O}(\epsilon^{-4})$, we choose \us{$\rho=\rho_0 = 1$} and $\alpha_k ={\alpha_0 / k^{2(1+c)}}$, where
%$\alpha_0$ satisfies $\sum_{k=0}^\infty \sqrt{\alpha_k} ={1/
%	\sqrt{2\rho}}$ \us{and} $c=1$e-$3$.

%taken in Algorithm 1 are shown together with other quality measures to compute an $\epsilon$-optimal and $\epsilon$-feasible solution. In addition, the actual number of required APG steps to obtain $\alpha_k$ accuracy is compared to the theoretical upper bound.
\begin{table}[h]
\tiny
\centering
\caption{Solution quality and performance statistics: Constant $\rho$ and
	known $\Sigma^*$.}
%\comt{%is relative error reported for suboptimality, see Table~2. Also CPU time?}
\begin{tabular}{|c|c|c|c|c|c|c|c|}
\hline
\multicolumn{1}{|c|}{$\epsilon$} &{$\texttt{s}(\bar{x}_K)$}&
{$\texttt{infs}(\bar{x}_K)$}&\multicolumn{1}{c|}{\begin{tabular}[c]{@{}c@{}}
	$K$ \end{tabular}}  & $\texttt{iter}_\mathrm{\alpha}$ &
\multicolumn{1}{c|}{\begin{tabular}[c]{@{}c@{}}
	$\texttt{iter}_\mathrm{L}$ \end{tabular}} & {\begin{tabular}[c]{@{}c@{}}{CPU
		time}\\
	\end{tabular}}\\ \hline
$1$e-$1$&$8.2$e-$2$& $1.1$e-$3$ & $4$ & $15$ &$35$& $17$\\ \hline
 $1$e-$2$&$7.3$e-$3$&$4.5$e-$4$&$5$&$28$ & $88$&$43$\\ \hline
 $1$e-$3$& $1.6$e-$4$&$3.5$e-$4$ & $5$&$64$ &$157$& $77$\\ \hline
$1$e-$4$&$4.3$e-$5$&$7.5$e-$5$&$5$& $117$ &$372$ &$188$\\ \hline
\end{tabular}
\label{Table1}
\vspace*{2mm}
\caption{Solution quality and performance statistics: Constant $\rho$ and
	misspecified $\Sigma^*$}
{\tiny
\begin{tabular}{|c|c|c|c|p{2.5mm}|p{4.5mm}|p{4.5mm}|c|c|c|}
\hline
\multirow{2}{*}{$\epsilon$} &
\multirow{2}{*}{$\texttt{s}(\bar{x}_K)$} &
%\multirow{2}{*}{$\frac{\|\Sigma_K-\Sigma^*\|}{\|\Sigma^*\|}$} &
\multirow{2}{*}{{$\texttt{le}(\Sigma_K)$}} &
\multirow{2}{*}{{$\texttt{infs}(\bar{x}_K)$}} &
\multirow{2}{*}{ $K$
%\begin{tabular}[c]{@{}c@{}}\# outer \\
%	$K$
%\end{tabular}
} &
		\multirow{2}{*}{$\texttt{iter}_\mathrm{\alpha}$} & \multirow{2}{*}{$\texttt{iter}_\mathrm{L}$}
&\multicolumn{2}{c|}{CPU time} \\\cline{8-9}
                  &                   &                   &                   &      &             &                   &    learn &     opt.    \\ \hline
$1$e-$1$  &$8.6$e-$2$& $4.6$e-$1$ &$1.2$e-$3$  & $4$    & $12$&$24$&$97$&$12$\\ \hline
$1$e-$2$&$8.8$e-$3$&$3.3$e-$1$ &$6.5$e-$5$   & $5$  &$19$&$65$&$129$&$35$\\ \hline
$1$e-$3$ &$9.9$e-$4$&$5.1$e-$2$ &$3.8$e-$5$ & $16$&$327$&$1247$&$388$&$630$ \\ \hline
$1$e-$4$ &$9.7$e-$5$&$9.4$e-$3$&$2.7$e-$6$& $47$& $4083$&$15052$&$1175$&$7533$ \\ \hline
\end{tabular}}
\label{Table2}
\end{table}
Table~\ref{Table2}
%(bottom) lists the results for various values of $\epsilon$. In addition, we
compares the CPU time spent {in}
		computation versus learning. %Note that
While our theoretical bound requires %at least
${\cal O}(\epsilon^{-4})$ %overall number of
APG iterations, the
				empirical behavior is far better, suggesting that the
				bound obtained in Theorem~\ref{Theo: overall complex
					const rho} is {loose and requires} further
					study.} {Furthermore},
%When comparing Tables~\ref{Table1}  and \ref{Table2}, we note that
	the overall effort
in the misspecified regime is significantly larger, {which} is not
surprising, since Table ~\ref{Table1} does not include the effort to
provide an exact $\Sigma^*$. The comparison of running times provided in
Table~\ref{Table2} suggests that {the effort for learning} is by no
means modest. {We also note that $\texttt{iter}_\mathrm{L}$ may be significantly larger
	than $\texttt{iter}_\mathrm{\alpha}$, suggesting potential improvement using other termination conditions implying $\alpha_k$-optimality as in~\cite{aybat2013augmented}.
%that the theoretically prescribed termination requirement is relatively loose.
} In Figure~\ref{Fig:Primal_subopt_scaled}({left}), we provide a graphical
	representation of how the empirically observed primal suboptimality
		error changes with $K$, the number of outer iterations when
		$\epsilon=1\texttt{e}$-$2$. This graph is overlaid by the
		theoretical bound based on
	Theorem~\ref{Theorem: Aug_prim_low_bnd_cons_rho}.
%, empirical performance is compared to the corresponding theoretical lower and upper bounds. %As it is shown, the primal infeasiblity error is within the bounds.
%In addition,
Figure~\ref{Fig:Primal_subopt_scaled}({right}) displays the
corresponding primal infeasibility and the associated theoretical bound obtained in
Theorem~\ref{Theorem:Aug_prim_infeas_cons_rho}. %\vspace*{-1mm} %Note that $y$-axis is
%depicted in log scale. %This figure also confirms that the actual error is within analytically specified bound.
\begin{figure}[htbp]
  \centering
%\scalebox{.8}
    \includegraphics[width=1.72in]{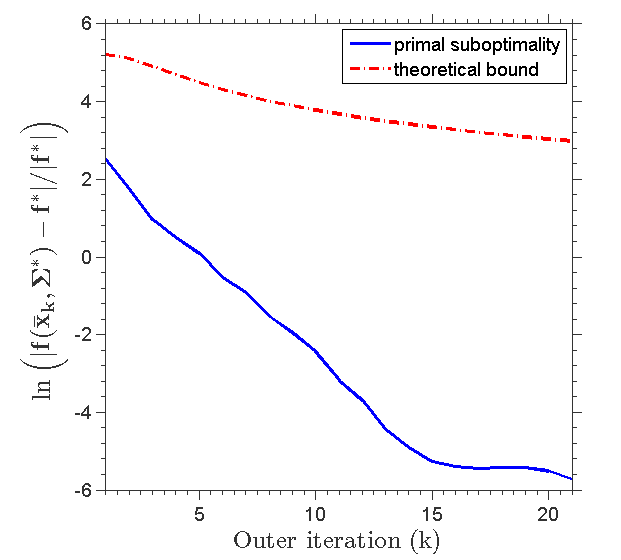}
    \includegraphics[width=1.72in]{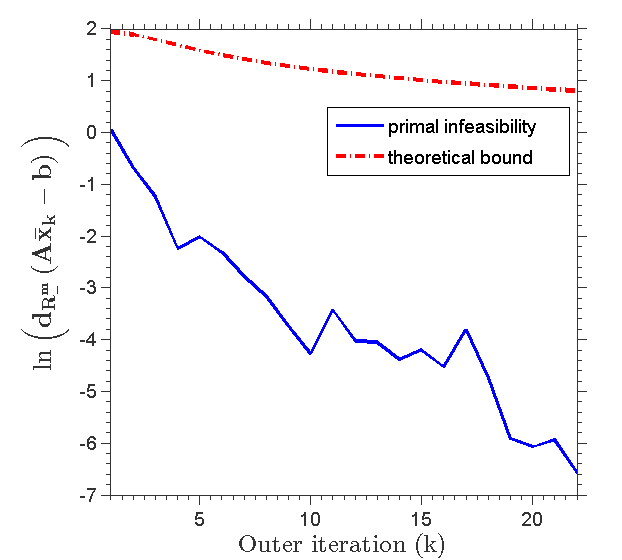}
%\scalebox{.8}
 \caption{Empirical error vs theoretical bound using constant $\rho$:
	 (left) Primal suboptimality $\ln\left(|f(\bar{x}_k;\Sigma^*)-f^*|/|f^*|\right)$; and (right) Primal
		 infeasiblity
		 $\ln\left(d_{\mathbb{R}^m_-}\left(A\bar{x}_k-b\right)\right)$.
%\comt{fonts in figure axis are different; we should report relative suboptimality error to be consistent with tables.}
}
       \label{Fig:Primal_subopt_scaled}
\end{figure}

\noindent {\bf b. Increasing penalty parameter sequence $\{\rho_k\}$.}
Next, we examine the %computational
performance of %Algorithm~\ref{Alg: ALM} %when using
{IPALM} for an increasing sequence of penalty parameters,
		$\{\rho_k\}$. To achieve the overall iteration complexity of ${\cal
			O}(\epsilon^{-1} \log(\epsilon^{-1}))$,
		Theorem~\ref{Theorem:Aug_increas_rho_linear_rate} suggests
			using sequences $\rho_k$ and $\alpha_k$ where $\rho_k=\rho_0 \beta^k$ and
$\alpha_k=k^{-2(1+c)}\beta^{-k}$,
where $\beta \tau<1$ and $\tau$ is such that $\|\Sigma_k-\Sigma^*\|\leq
\tau^k \|\Sigma_0-\Sigma^*\|$. We set $c=1$e-$3$ and $\beta=1.05$, based on the
calculated value of $\tau=0.91$.
For various values of $\epsilon$,
Tables~\ref{Table_3} and~\ref{Table_4} display the numerical
	results for known and misspecified $\Sigma^*$,
respectively. We begin by noting that the
overall complexity in terms of inner iterations
%(\us{specifically iter$_{\alpha}$ and iter$_L$})
is not significantly larger %in terms of the number of computational steps as that
than the iteration number observed with known $\Sigma^*$, providing empirical support for the theoretical findings of
	Theorem~\ref{Theo: Overall Iter Comp Increas rho} and Remark 4.2.
	Figure~\ref{Fig:Primal_infeas_increas_rho}({left}) depicts how the
		empirically observed  primal
	suboptimality error changes with $K$ when
	$\epsilon=1$e-$2$. By overlaying the trajectory derived from the
	non-asymptotic upper bound which diminishes at a linear rate as derived in
	Theorem~\ref{Theorem:Aug_increas_rho_linear_rate}, it is seen that
	the numerics support the theoretical findings.
	%The empirical result confirms the theoretical findings, since the actual error is within the specified bound. Note that $x$ axis is shown in log scale for a better visualization.
In addition, Figure~\ref{Fig:Primal_infeas_increas_rho}({right}) displays the
corresponding primal infeasibility and the associated theoretical bound obtained in
Theorem~\ref{Theorem:Aug_increas_rho_linear_rate}.  Note that when $\Sigma^*$ is known, as shown in
	Corollary~\ref{Corol:Constant rho no learn complex}, choosing a
		constant $\rho$ results in ${ \cal O}(1/\epsilon)$ overall
		complexity. While in theory, this is preferable to employing an
			increasing sequence  $\{\rho_k\}$ which has a larger
			complexity of ${\cal O}\left(\epsilon^{-1}
					\log(\epsilon^{-1})\right)$, the constant $\rho$
				version requires careful estimation of problem
				parameters and $\rho$ based on the choice
			$\epsilon$. In contrast, when $\rho_k$ is an increasing
			sequence, the choice of $\beta$ and $\rho_0$ is independent
			of problem parameters, a distinct advantage of the
				increasing penalty parameter scheme.
\begin{table}[htbp]
\renewcommand{\arraystretch}{1.1}
\centering
{\tiny
\caption{Solution quality and performance statistics: Increasing $\rho_k$
	and known $\Sigma^*$
%($\mbox{opt}(\bar x_K) \triangleq \tfrac{|f(\bar{x}_K;\Sigma^*)-f^*|}{f^*}$ \mbox{ and } $\mbox{infeas}(\bar x_K) \triangleq d_{\R^m_-}(A\bar{x}_K-b)$).
}
\label{Table_3}
\begin{tabular}{|c|c|c|c|c|c|c|}
\hline
%\multicolumn{1}{|c|}{$\epsilon$} &$\frac{|f(x_K;\Sigma^*)-f^*|}{|f^*|}$ &
\multicolumn{1}{|c|}{$\epsilon$} &{$\texttt{s}(x_K)$} &
{$\texttt{infs}(x_K)$} &\multicolumn{1}{c|}{ $K$
%\begin{tabular}[c]{@{}c@{}}
%	\# outer~($K$)
%\end{tabular}
}  &$\texttt{iter}_\mathrm{\alpha}$ &
\multicolumn{1}{c|}{\begin{tabular}[c]{@{}c@{}}{$\texttt{iter}_\mathrm{L}$}\end{tabular}} & CPU time \\
\hline
$1$e-$1$ &$7.9$e-$2$&$1.4$e-$2$  &$5$  & $7$ &$7$ &$4$\\ \hline
$1$e-$2$&$8.0$e-$3$&$4.3$e-$3$ & $9$ & $24$ &$28$&$15$\\ \hline
 $1$e-$3$& $5.3$e-$4$&$5.4$e-$4$ & $15$&$60$ &$97$ & $49$\\ \hline
$1$e-$4$&$9.8$e-$5$ &$1.3$e-$4$& $26$& $224$& $415$ & $201$\\ \hline
\end{tabular}}%
\vspace*{2mm}
{\tiny
\caption{Solution quality and performance statistics: Increasing $\rho_k$ and misspecified $\Sigma^*$}
\label{Table_4}
\begin{tabular}{|c|c|c|c|p{2mm}|p{4.5mm}|p{4.5mm}|c|c|c|}
\hline
\multirow{2}{*}{$\epsilon$} &
\multirow{2}{*}{{$s(x_K)$}} &
\multirow{2}{*}{{$\texttt{le}(\Sigma_K)$}} &
\multirow{2}{*}{{$\texttt{infs}(x_K)$}} &
\multirow{2}{*}{ $K$
%\begin{tabular}[c]{@{}c@{}}
%\# outer \\
%$K$
%\end{tabular}
} &
		\multirow{2}{*}{$\texttt{iter}_\mathrm{\alpha}$} &
		\multirow{2}{*}{$\texttt{iter}_\mathrm{L}$} &\multicolumn{2}{c|}{CPU time} \\\cline{8-9}
                  &                   &                   &                   &      &             &                   &    learn &     opt.     \\ \hline
$1$e-$1$ &$9.3$e-$2$& $3.3$e-$1$ &$1.0$e-$2$  & $5$  &  $7$&$7$&$121$&$3$ \\ \hline
$1$e-$2$ &$9.5$e-$3$&$9.1$e-$2$ &$2.4$e-$3$   & $11$ & $30$&$40$&$278$&$23$\\ \hline
$1$e-$3$ &$8.3$e-$4$&$4.1$e-$2$ &$4.9$e-$4$ & $19$&  $82$&$153$&$477$&$79$ \\ \hline
$1$e-$4$ &$9.7$e-$5$&$8.8$e-$3$ &$5.1$e-$6$& $49$& $632$&$3488$&$1221$&$1735$\\ \hline
\end{tabular}}
\end{table}
\begin{figure}[htbp]
  \centering
%\scalebox{.8}
    \includegraphics[width=1.72in,height=1.52in]{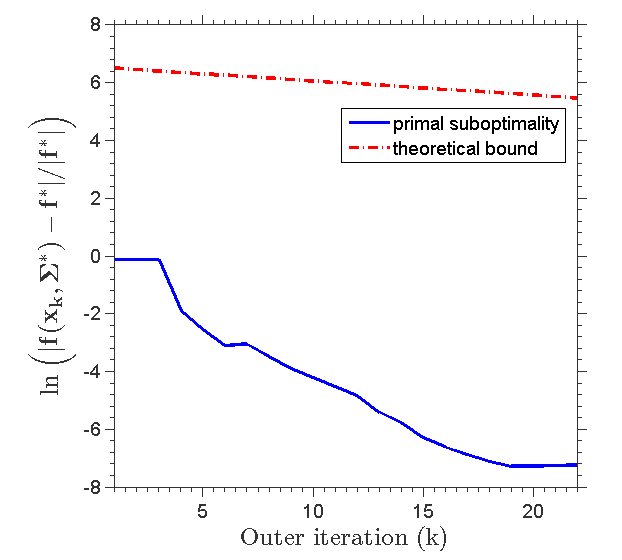}
    \includegraphics[width=1.72in,height=1.52in]{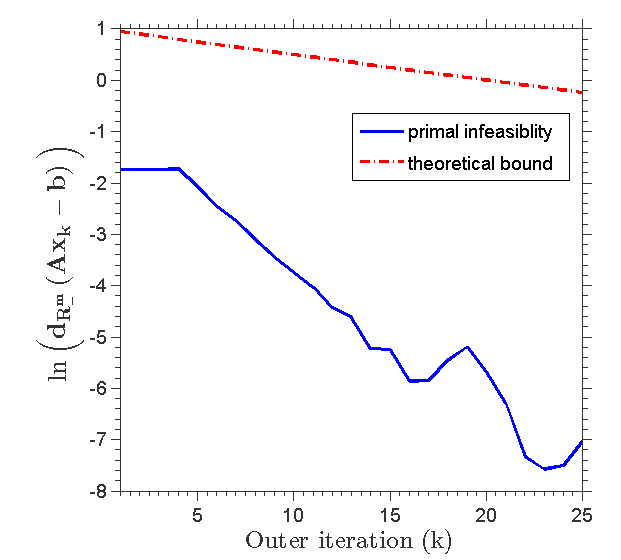}
%\scalebox{.8}
 \caption{Empirical error vs theoretical bound using increasing $\rho_k$:
	 (left) Primal relative suboptimality $\ln\left(|f(x_k;\Sigma^*)-f^*|/|f^*|\right)$; and (right) Primal
		 infeasiblity
		 $\ln\left(d_{\mathbb{R}^m_-}\left(Ax_k-b\right)\right)$.
%\comt{fonts in figure axis are different; we should report relative suboptimality error to be consistent with tables.}
}
       \label{Fig:Primal_infeas_increas_rho}
       \vspace*{-10mm}
\end{figure}
\begin{figure}[hbtp]
    \vspace*{-5mm}
  \centering
    \includegraphics[width=2.5in,height=1.5in]{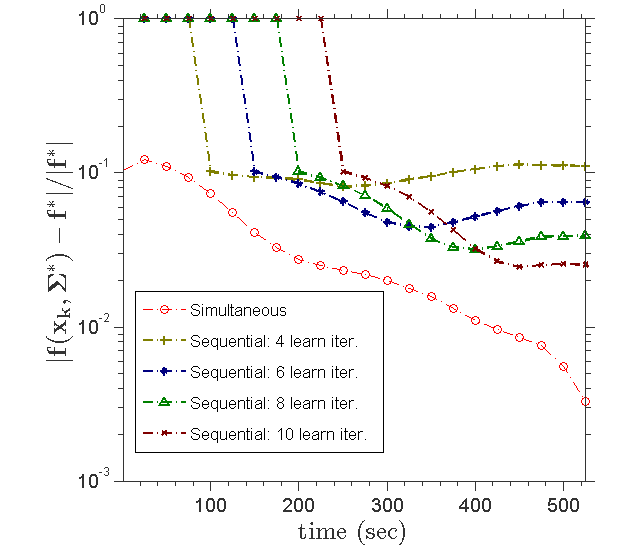}
      \caption{Simultaneous vs sequential approach -- Primal suboptimality for increasing $\rho_k$: $|f(x_k;\Sigma^*)-f^*|$ in $\log$-scale}
          \label{Fig:Simul_vs_seq}
\end{figure}
%\us{(c).  When $\Sigma^*$ is misspecified, the overall complexity associated with the increasing penalty scheme is not of increasing $\rho$ does not change, while complexity of constant $\rho$ dramatically increases to ${\cal O}(\epsilon^{-4})$. However, based on the numerical result, we strongly suspect that this overall complexity is underestimated in Theorem~\ref{} and more careful analysis of the problem may result in a tighter bound.  }
%\newpage
%\vspace*{-3mm}
{\bf c. Sequential vs simultaneous schemes.}
%Our last set of numerics
Here we provide a graphical representation of the
benefits of simultaneous schemes and captures the overall effort/time in
%a single
one figure (Figure \ref{Fig:Simul_vs_seq}). To compare our proposed
scheme with standard sequential schemes, we incorporate the effort to
solve the learning problem {in sequential schemes %{an} a priori fashion, and then use
 -- this possibly inexact solution is then used to resolve the misppecified problem.} For
instance, in Figure~\ref{Fig:Simul_vs_seq}, we consider 4 different
implementations of the sequential scheme, where the implementations
differ by the amount of effort (number of learning steps) employed for obtaining an approximation
to $\Sigma^*$. On the $y-$axis, we capture the sub-optimality error. Note while the sequential schemes are making an effort to get an
approximation of $\Sigma^*$, no improvement is being made in $x$.
Consequently, all of the graphs corresponding to the  sequential schemes
stay constant. Once an approximation {to $\Sigma^*$} is obtained, the sequential scheme
will obtain an approximate solution to {the misspecified problem}; but, the sub-optimality error never
diminishes to zero, since the sequential scheme never updates its $\Sigma^*$ approximation. In comparison, the simultaneous approach %{on the other hand}
has several benefits: (i) it is
characterized by asymptotic convergence, a property that does not hold
for sequential schemes; (ii) one can provide
non-asymptotic rate bounds for the entire trajectory $\{x_k\}$; and
	(iii) when it is unclear as to the extent of accuracy required in
	solving the learning problem, sequential methods can prove to be
	quite poor while simultaneous schemes perform well.
%\vspace*{-2mm}
\section{Conclusions}
%\vspace*{-2mm}
This paper has been motivated by the question of resolving convex
	optimization {problems plagued by} parametric misspecification.
 %both in the objective and the constraints.
 We consider settings
		where this misspecification may be resolved by solving a
		suitably defined learning problem, {resulting in}
%Accordingly, we consider the setting where we have two
two	coupled optimization problems; of these, the first
is a misspecified optimization problem where the unknown parameters
appear both in the objective and  the constraints,
	   while the second is a learning problem that arises from having access to a learning
data set, collected a priori. One avenue for contending with
such a problem is through an inherently sequential approach
that solves the learning problem and subsequently utilizes this solution
in solving the %computational
{misspecified} problem. Unfortunately,
unless \emph{accurate} solutions of the learning problem
are available in a {practically reasonable} number of iterations,
sequential approaches may not be advisable {due to propagation of error}. Instead, we focus on a
\emph{simultaneous} approach that combines learning and %computation
optimization by adopting
inexact augmented Lagrangian (AL) scheme. Two classes of inexact AL
schemes have been investigated; {the} first uses {a} constant penalty
parameter %in its implementation
while {the} second employs {an} increasing
sequence of penalty parameters. In this regard, we make the following contributions:
(i) Derivation of the convergence rate for dual optimality, primal infeasibility and primal suboptimality;
(ii) Quantification of the learning effect on the rate degradation;
(iii) Analysis of overall iteration complexity. Preliminary numerics
	suggest %that
the proposed schemes perform well on a misspecified
		portfolio optimization problem while {sequential} approaches
		%for addressing misspecification
		may perform poorly in practice.
\bibliography{ref2,ref2_revision}   % name your BibTeX data base
\bibliographystyle{unsrt}
\vspace*{-2cm}
\begin{IEEEbiographynophoto}{Necdet Serhat Aybat} received his B.S. in 2003
and M.S. in 2005 from Bogazici University, and Ph.D. in Operations Research
from Columbia University in 2011. Dr. Aybat joined the Department of Industrial
Engineering, Penn State University in August 2011. His current research focuses
on developing first-order algorithms for large-scale convex optimization
problems from diverse application areas, such as compressed sensing, matrix
completion, convex regression, and distributed optimization. Dr. Aybat is a
runner-up of the INFORMS Computing Society Student Paper Award in 2010, and the
recipient of the SIAM Student Paper Prize in 2011.
\end{IEEEbiographynophoto} \vspace*{-1in}
\begin{IEEEbiographynophoto}{Hesam Ahmadi} received his Ph.D. in Operations
Research from Penn State University in 2016 and currently works at Optym in
Gainseville, Florida. His research interests lie in the solution of
optimization problems plagued by misspecification with application interests in
power systems operation.
\end{IEEEbiographynophoto} \vspace*{-1in}
\begin{IEEEbiographynophoto}{Uday V. Shanbhag} has held the Gary and Sheila
Bello Chaired professorship in Industrial and Manufacturing Engineering at
Pennsylvania State University since Nov. 2017 and has been at Penn. State
University since Fall 2012, prior to which he was an assistant and subsequently
a tenured associate professor at University of Illinois at Urbana-Champaign
(between 2006--2012). His interests lie in the analysis and solution of
optimization problems, variational inequality problems, and Nash games
complicated by nonsmoothness and uncertainty. His research awards include the
A.W. Tucker prize from the Mathematical Optimization Society (2006), the best
paper award from Computational Optimization and Applications (2007, jointly
with W. Murray), and the Best theory paper award from the Winter Simulation
Conference in 2013 (with A. Nedi\'{c} and F. Yousefian). He holds undergraduate
and Masters degrees from IIT, Mumbai and MIT, Cambridge respectively and  a
Ph.D. in management science and engineering (Operations Research) from Stanford
University in 2006.
\end{IEEEbiographynophoto}
\appendices
\section{Proof of Lemma \ref{lip-pi}}
\label{sec:app-LipschitzI}
%{\bf Proof of Lemma \ref{lip-pi}.}
{Given $\lambda\in\cK^*$, for any $\theta\in\Theta$, the definition of $\pi_{\rho}(\lambda;\theta)$ in~\eqref{def-pi} implies that $\tfrac{1}{\rho} (\lambda - \pi_{\rho}(\lambda;{\theta})) \in \partial_w g_0(w;{\theta})|_{w=\pi_{\rho}(\lambda;{\theta})}$.}
%$$ 0 \in \partial_w g_0(\pi_{\rho}(\lambda;{\theta_1});{\theta_1}) + \tfrac{1}{\rho} (\pi_{\rho}(\lambda;{\theta_1})-\lambda) \implies \tfrac{1}{\rho} (\lambda - \pi_{\rho}(\lambda;{\theta_1})) \in \partial_w g_0(\pi_{\rho}(\lambda;{\theta_1});{\theta_1}). $$
%Similarly, we have that
%$$ \tfrac{1}{\rho} (\lambda - \pi_{\rho}(\lambda;\theta_2)) \in \partial_{w} g_0(\pi_{\rho}(\lambda;\theta_2);\theta_2). $$
%By assumption, we have that
%$$ \partial_{w} g_0(w;{\theta_1}) \subseteq \partial_{w} g_0(w;\theta_2) + \kappa \|{\theta_1}-\theta_2\|~ \cB(\mathbf{0},1).$$
Consequently, from the assumption in \eqref{lip-sub}, there exists %a vector
$\xi \in {\kappa_d(\lambda)} \|\theta -\theta^*\|~\cB(\mathbf{0},1)$ such that
{
\begin{align*}
\tfrac{1}{\rho} (\lambda - \pi_{\rho}(\lambda;{\theta}))
&\in \partial_{w} g_0(w;{\theta})|_{w=\pi_{\rho}(\lambda;{\theta})}\\
&\subseteq \xi+ \partial_{w} g_0(w;\theta^*)|_{w=\pi_{\rho}(\lambda;\theta)}.
\end{align*}}%
Therefore, %we have that
$\tfrac{1}{\rho} (\lambda - \pi_{\rho}(\lambda;{\theta})) -\xi \in \partial_w g_0(w;\theta^*)|_{w=\pi_{\rho}(\lambda;\theta)}$.
By the monotonicity of the map $\lambda\mapsto\partial_\lambda g_0(\lambda;\theta)$ for every $\theta \in \Theta$ and for nonnegative $\rho$, we have
%{\small
%\begin{align*}
{\small $0 \leq \fprod{\tfrac{1}{\rho} \big(\lambda -
		\pi_{\rho}(\lambda;{\theta})) -\xi -\tfrac{1}{\rho} (\lambda -
		\pi_{\rho}(\lambda;\theta^*)\big),~\pi_{\rho}(\lambda;{\theta})-\pi_{\rho}(\lambda;\theta^*)}$}.
    %\right.,\\
	%& \left.~\pi_{\rho}(\lambda;{\theta}_1)-\pi_{\rho}(\lambda;\theta_2)}. %\\
	%& =\tfrac{1}{\rho} (\pi_{\rho}(\lambda;\theta_2)- \pi_{\rho}(\lambda;{\theta_1}) -\rho\xi)^\top(\pi_{\rho}(\lambda;{\theta_1})-\pi_{\rho}(\lambda;\theta_2)) \\
	%& = -\tfrac{1}{\rho} \|\pi_{\rho}(\lambda;\theta_1)-\pi_{\rho}(\lambda;{\theta_2})\|^2 - \xi^\top (\pi_{\rho}(\lambda;{\theta_1})-\pi_{\rho}(\lambda;\theta_2)).
%\end{align*}}%
Hence, we obtain $\frac{1}{\rho} \|\pi_{\rho}(\lambda;\theta)-\pi_{\rho}(\lambda;{\theta^*})\| \leq \|\xi\|$ by rearranging the terms.
%the following inequality:
%\begin{align*}
% \frac{1}{\rho}
% & \|\pi_{\rho}(\lambda;\theta_1)-\pi_{\rho}(\lambda;{\theta}_2)\|^2
% \leq \|\xi\|\|
% \pi_{\rho}(\lambda;{\theta_1})-\pi_{\rho}(\lambda;\theta_2)\|\\
%\implies & \frac{1}{\rho} \|\pi_{\rho}(\lambda;\theta_1)-\pi_{\rho}(\lambda;{\theta_2})\| \leq \|\xi\|.
%\end{align*}
Moreover, $\xi \in {\kappa_d(\lambda)} \|{\theta} -\theta^*\| \cB(\mathbf{0},1)$ implies that
$\|\xi\| \leq {\kappa_d(\lambda)} \|{\theta} -\theta^*\|$, which leads to \eqref{eq:lipschitz-MoreauMap}.\
%\begin{align*}
% \|\pi_{\rho}(\lambda;\theta_1)-\pi_{\rho}(\lambda;{\theta_2})\|  \leq \kappa \rho \|{\theta_1}-\theta_2\|.
%\end{align*}
\section{Proof of Lemma~\ref{lip-g0}}
\label{sec:app-LipschitzII}
%{\bf Proof of Lemma~\ref{lip-g0}.}
Recall that
%$g_0(\lambda;\theta)$ is defined as
$g_0(\lambda;\theta) \triangleq \min_{x \in X} \ {\cal L}_0(x,\lambda;\theta)$.
%where $\cL_0$ %$\cL_0(.,.;\theta)$
%is defined in~\eqref{eq:lagrangian}.
Then {for any $\lambda\in\cK^*$ and $\theta\in\Theta$, Danskin's theorem implies}
$\partial_{\lambda} g_0(\lambda;\theta) = \mbox{conv} \left\{
{h(x;\theta)}: x \in X^*(\lambda;\theta)\right\}$. Hence,
%As a consequence, we have %one may note that $\partial_{\lambda}g_0(\lambda;\theta)$ is given by the following:
\begin{align}
 \partial_{\lambda} g_0(\lambda;\theta) = \mbox{conv} \left\{
{h\left(X^*(\lambda;\theta);\theta\right)}\right\}=h(X^*(\lambda;\theta);\theta),\label{g_0_eq_h}
\end{align}
due to the {image
	$h(X^*(\lambda;\theta);\theta)$ being a convex set} -- since {$h(x;\theta)$ is an affine map in $x$ for every $\theta \in
	\Theta$} and $X^*(\lambda;\theta)$ is a convex set for any
	$\theta$ and $\lambda$. %implying that
From {Assumption~\ref{Assump: Aug_1}.iii}, there exists {some $\kappa(\lambda)>0$ such that for all $\theta\in\Theta$,}
\begin{align}
\label{pseudo-L-X}
{X^* (\lambda;\theta) \subseteq X^*
(\lambda;\theta^*) + \kappa(\lambda) \| \theta - \theta^* \|~\cB(\mathbf{0},1).}
\end{align}
Since {$h(\cdot;\theta)$ is an affine map, from \eqref{pseudo-L-X}, for every $\theta \in
\Theta$},  %it follows that
{\small
\begin{align}
{h(X^*(\lambda;\theta){;\theta})  \subseteq h\big(X^*
		(\lambda;\theta^*) %\nonumber\\
		+ \kappa(\lambda) \| \theta - \theta^* \|~\cB(\mathbf{0},1){;\theta}\big).}
\label{lhs-final}
\end{align}}%
Define {$\bar{h}(x;\theta)\triangleq h(x;\theta)-b(\theta)=A(\theta)x$}, as the linear part of $h(x;\theta)$. Then, the image of the Minkowski sum of sets %, {$\bar{h}\big(X^* (\lambda;\theta_2) + \kappa^* \| \theta_1 - \theta_2 \|~\cB(\mathbf{0},1);\theta_1\big)$},
can be written as follows:
{\
\begin{align}
& \quad \bar{h}\left(X^* (\lambda;\theta^*) + \kappa(\lambda) \| \theta -
		\theta^* \|~\cB(\mathbf{0},1);{\theta}\right)	\nonumber  \\
%& =  \left\{ \bar{h}(x+y;{\theta_1}):\ x \in X^*(\lambda;\theta_2),\ y \in \kappa^* \| \theta_1-\theta_2\|~\cB(\mathbf{0},1)\right\} \cr
%& = {\left\{ {\bar{h}(x;\theta_1)} + \bar{h}(y;\theta_1):\ x \in X^*(\lambda;\theta_2),\ y \in \kappa^* \| \theta_1-\theta_2\|~\cB(\mathbf{0},1)\right\}} \cr
& = \left\{ {\bar{h}(x;\theta)} + z:\
\begin{array}{l}
  x \in X^*(\lambda;\theta^*),\\
   z \in \kappa(\lambda) \| \theta-\theta^*\|~{\bar{h}\big(\cB(\mathbf{0},1);\theta\big)}
\end{array}
\right\} \cr
%& = \bar{h}(X^*(\lambda;\theta_2) ) + \kappa \| \theta_1-\theta_2\| (-A(B(0,1))) \cr
%& \subseteq \bar{h}(X^*(\lambda;\theta_2) ) + \kappa \| A\|  \| \theta_1-\theta_2\| B(0,1) \cr
& \subseteq \bar{h}\big(X^*(\lambda;\theta^*){;\theta}\big) + \kappa(\lambda) {\norm{A(\theta)}} \| \theta-\theta^*\|~\cB(\mathbf{0},1) \cr
& \subseteq \bar{h}\big(X^*(\lambda;\theta^*){;\theta^*}\big) + {L_A D_x
\|\theta-\theta^*\|~\cB(\mathbf{0},1)}\nonumber \\
& \mbox{ } +  \kappa(\lambda) L_{h,x}  \| \theta-\theta^*\|~\cB(\mathbf{0},1),
\label{rhs}
\end{align}}%
{where \eqref{rhs} follows from Assumption~\ref{Assump: Aug_1}.i and the definitions of $\bar{h}$ and $L_{h,x}$.}
%By adding $b(\theta_1)$ to the both sides of above inclusion, and using
From $h(x;\theta)\triangleq\bar{h}(x;\theta) + b(\theta)$, it follows that
{\begin{align*}
& h\left(X^* (\lambda;\theta^*) + \kappa(\lambda) \| \theta -
		\theta^* \|~\cB(\mathbf{0},1);{\theta}\right)	\notag  \\
\subseteq & h(X^*(\lambda;\theta^*){;\theta^*} ) + {(L_A D_x+\kappa(\lambda) L_{h,x})
\|\theta-\theta^*\|~\cB(\mathbf{0},1)}  \nonumber \\
& +b(\theta)-b(\theta^*). %\label{tilda_rel}
\end{align*}}%
Since $b(\theta)$ is Lipschitz continuous in $\theta$ with constant $L_b$ {and $L_{h,\theta}=L_A D_x+L_b$}, %we can rewrite ~\eqref{tilda_rel} as follows:
\begin{align*}
&h \left(X^* (\lambda;\theta^*) + \kappa(\lambda) \| \theta -
		\theta^* \|~\cB(\mathbf{0},1);{\theta}\right)	\notag  \\
\subseteq & h(X^*(\lambda;\theta^*){;\theta^*} ) + {(L_{h,\theta}+\kappa(\lambda) L_{h,x})
\|\theta-\theta^*\|~\cB(\mathbf{0},1)}; %\label{rhs2}
\end{align*}
%According to \eqref{g_0_eq_h}, $ h(X^*(\lambda;\theta_i){;\theta_i}) = \partial_{\lambda} g_0(\lambda;\theta_i)$ for $i=1,2$;
{hence, the result follows from \eqref{g_0_eq_h} and \eqref{lhs-final}.} %and \eqref{rhs2}. %it follows that
%\begin{align*}
%\partial_{\lambda} g_0(\lambda;\theta_1)
%%& = h(X^*(\lambda;\theta_1);\theta_1) \cr
%%& {\ \subseteq \ } h\left(X^* (\lambda;\theta_2) + \kappa^* \| \theta_1 - \theta_2 \| B(0,1){;\theta_1}\right) \cr
%%& \subseteq  h(X^*(\lambda;\theta_2){;\theta_2} ) + {(L_{h,\theta}+L_b +  \kappa^* L_{h,x})}  \| \theta_1-\theta_2\| B(0,1)\cr
%\subseteq &\partial_{\lambda} g_0(\lambda;\theta_2)  \nonumber \\
%& +{(L_{h,\theta} + \kappa^*
%		L_{h,x})}  \| \theta_1-\theta_2\|~\cB(\mathbf{0},1). %\label{bd-sigma}
%\end{align*}
%\begin{align}
%\label{pseudo-L-pg} \partial_{\lambda} g_0(\lambda;\theta_1) \subseteq \partial_{\lambda} g_0(\lambda;\theta_2) + \kappa \sigma_{\rm max} \| \theta_1 - \theta_2 \| B(0,1). \end{align}

{%Next, we show that %the sequence

%We prove this result for the general case where penalty parameter sequence is allowed to change.
}
\section{Proof of Theorem~\ref{thm:pseudoLipschitz}}
\label{sec:app-LipschitzThm}
%{\bf Proof of Theorem~\ref{thm:pseudoLipschitz}.}
First, note that since $X$ is compact, due to Weierstrass' theorem,
$X^*(\lambda;\theta)\neq\emptyset$ {for all $\lambda\in\cK^*$ and $\theta\in\Theta$.}
Moreover, for $\theta=\theta^*$, strict convexity of $f$ implies that
\begin{equation}
\label{eq:singleton}
{X^*(\lambda;\theta^*)=\{x^*(\lambda;\theta^*)\}}
\end{equation}
is a singleton. Throughout the sequel, let $F(x,\theta)\triangleq \grad_x f(x;\theta)$ and %and \us{for $\lambda\in\Lambda$,
we define $F:\tilde{X} \times \Lambda \times \Theta\rightarrow\reals^n$ such that
\begin{align}
F(x,\lambda,\theta)\triangleq F(x,\theta)+A(\theta)^\top\lambda=\nabla_x \mathcal{L}_0(x,\lambda;\theta).
\end{align}
%Under suitable convexity and differentiability assumptions,

{Given any $\lambda\in\cK^*$ and $\theta\in\Theta$, since $f(\cdot;\theta)$ is convex and differentiable, computing $X^*(\lambda;\theta)$ in \eqref{lag-prob} is equivalent to solving} a variational inequality problem ${\rm VI}(X,F(\cdot,\lambda,\theta))$, of which solution set is denoted by ${\rm SOL}(X,F(\cdot,\lambda,\theta))$, {i.e., $\bar{x}\in {\rm SOL}(X,F(\cdot,\lambda,\theta))$ if and only if $\bar{x}\in X$ such that $\fprod{F(\sa{\bar{x}},\lambda,\theta),~x-\bar{x}}\geq 0$ for all $x\in X$}.
%where $F(x,\lambda,\theta)  \triangleq\us{\nabla_x} \mathcal{L}_0(x,\lambda,\theta)$.
Since $f(\cdot;\theta)$ is assumed to be convex for all $\theta \in \Theta$, $F(\cdot,\lambda,\theta)$ is a monotone map %in $x$
on $X$ for all $\lambda\in\cK^*$ and $\theta \in \Theta$.

Next, \sa{given an arbitrary compact set $\Lambda\subset\cK^*$,} we prove a key lemma that will help us prove that \eqref{eq:main-assumption} holds for $\lambda\in\Lambda$.
\begin{lemma}
\label{lem:key}
Under the premise of Theorem~\ref{thm:pseudoLipschitz}, \sa{given a compact set $\Lambda\subset\cK^*$}, there exists $c>0$ such that for all $\lambda\in\Lambda$ and $q\in\reals^n$,
\begin{align}
\norm{x^*(\lambda;\theta^*)-\argmin_{x\in X}\{{\cL_0}(x,\lambda;\theta^*)+{q^\top x}\}}\leq c\norm{q}. %\quad \forall q\in\reals^n.
\end{align}
\end{lemma}
\begin{proof}
For $q=\mathbf{0}$, the claim trivially holds for any $\lambda\in\Lambda$. Hence, in the rest of the proof, suppose that $q\neq\mathbf{0}$. For the sake of contradiction, assume that there exist $\{\lambda_k\}\subset\Lambda$, $\{q_k\}\subset\reals^n\setminus\{\mathbf{0}\}$ such that
\begin{align}
\label{eq:contradiction}
\norm{x_k-x^*_k}>k \norm{q_k},\quad \forall k\geq 1,
\end{align}
where $x_k\triangleq \argmin_{x\in X}\{{\cL_0}(x,\lambda_k;\theta^*)+{q_k^\top x}\}$ and $x_k^*\triangleq x^*(\lambda_k;\theta^*)$ -- see \eqref{eq:singleton}.

To simplify the notation used in the analysis below, let $F(x)\triangleq
	F(x,\theta^*)$ and, %for each $k\geq 1$,
	we also define
$F_k:\tilde{X}\rightarrow\reals^n$ such that $F_k(x)\triangleq
F(x,{\lambda_k},\theta^*)=F(x)+A(\theta^*)^\top\lambda_k$ for $k\geq 1$.

Since $q_k\neq \mathbf{0}$, from strict convexity of $f(\cdot;\theta^*)$, it follows that $x_k\neq x_k^*$ for $k\geq 1$. Thus, \eqref{eq:contradiction} implies that
\begin{align}
\lim_{k\rightarrow\infty}\frac{\norm{q_k}}{\norm{x_k-x_k^*}}=0.
\end{align}
Moreover, since $x_k\in X$ and $x_k^*\in X$, it follows from \eqref{eq:contradiction} and compactness of $X$ that $\lim_{k\rightarrow\infty} q_k=\mathbf{0}$. Define $v_k { \ \triangleq \ } (x_k-x_k^*)/\norm{x_k-x_k^*}$ for all $k\geq 1$. Clearly $\norm{v_k}=1$ and {since $\{v_k\}$ is a bounded sequence}, there exists a  subsequence $\cI_1\subset\integers_+$ such that $\lim_{k\in\cI_1}v_k=\bar{v}\in\reals^n$ -- note that $\norm{\bar{v}}=1$; hence, $\bar{v}\neq \mathbf{0}$. Since both $\Lambda$ and $X$ are compact sets, there exists $\cI_2\subset\cI_1$ such that for some $\bar{\lambda}\in\Lambda$ and $\bar{x},x_{\infty}\in X$, we have $\bar{\lambda}=\lim_{k\in\cI_2}\lambda_k$, $\bar{x}=\lim_{k\in\cI_2}x_k$, and $x_{\infty}=\lim_{k\in\cI_2}x^*_k$. Thus, $x^*_k\in\mathrm{SOL}(X, F_k)$ and $x_k\in\mathrm{SOL}(X, F_k+q_k)$ imply
{\small
\begin{align}
& %x^*_k\in\mathrm{SOL}(X, F_k)\quad \Leftrightarrow\quad
\fprod{x-x_k^*,~F_k(x_k^*)}\geq 0,\quad\forall~x\in X, \label{eq:sol_1}\\
&%x_k\in\mathrm{SOL}(X, F_k+q_k)\quad \Leftrightarrow\quad
\fprod{x-x_k,~F_k(x_k)+q_k}\geq 0,\quad\forall~x\in X, \label{eq:sol_2}
\end{align}}%
respectively, for $k\geq 1$. {Let $\bar{F}:\tilde{X}\rightarrow\reals^n$ be such that
$\bar{F}(x)\triangleq F(x,{\bar{\lambda}},\theta^*)$.} Therefore, {since $\lim_{k \to\infty} q_k = {\bf 0}$}, taking the limit along the subsequence $\cI_2$, we conclude that $\fprod{x-x_\infty,~\bar{F}(x_\infty)}\geq 0$ and $\fprod{x-\bar{x},~\bar{F}(\bar{x})}\geq 0$ hold for all $x\in X$; thus, $\bar{x}=x_\infty$ due to strict convexity of $f(\cdot;\theta^*)$.  Moreover, from \eqref{eq:sol_1} and \eqref{eq:sol_2},
%it follows that
{\small
\begin{align*}
0
&\leq \fprod{x_k^*-x_k,~F_k(x_k)+q_k},\nonumber\\
&= \fprod{x_k^*-x_k,~F_k(x_k^*)+\grad^2 f(x_k^*;\theta^*) (x_k-x_k^*)}\nonumber\\
&\quad+\fprod{x_k^*-x_k,~q_k+o(\norm{x_k-x_k^*})},\nonumber\\
&\leq \fprod{x_k^*-x_k,~\grad^2 f(x_k^*;\theta^*) (x_k-x_k^*)+q_k+o(\norm{x_k-x_k^*})}. %\label{eq:key}
\end{align*}}%
Now, dividing both sides of the above inequality by ${\norm{x_k-x_k^*}^2}$ and using the definition of $v_k$, we get
{\small
\begin{align*}
0
&\leq \fprod{-v_k,~\grad^2 f(x_k^*;\theta^*)v_k+\frac{q_k}{{\|x_k-x_k^*\|}}+\frac{o(\norm{x_k-x_k^*})}{{\|x_k-x_k^*\|}}}. %\label{eq:key-2}
\end{align*}}%
Taking the limit along the subsequence $\cI_2$, we get $\bar{v}^\top\grad^2 f(\bar{x};\theta^*)\bar{v}\leq 0$, which clearly contradicts our assumption {that $f(\cdot;\theta^*)$ is strictly convex on $X$.}
\end{proof}
 Now we are ready to prove the main claim in Theorem~\ref{thm:pseudoLipschitz}.
\sa{\bf Proof of Theorem~\ref{thm:pseudoLipschitz}.}
%Suppose the assumption of Theorem~\ref{thm:pseudoLipschitz} holds. {Since $f(\cdot;\theta^*)$ is strictly convex on $X$, $X^*(\lambda;\theta^*)=\{x^*(\lambda;\theta^*)\}$ is a singleton for all $\lambda\in\Lambda$.}
Given $\lambda\in\Lambda$ and $\theta\in\Theta$, let $\bar{x}\in X^*(\lambda;\theta)=\argmin_{x\in X} \lambda_0(x,\lambda;\theta)$ -- note that $X^*(\lambda;\theta)$ is \emph{not necessarily} a singleton.
%Using a similar argument as in \eqref{eq:perturbed-sol},
\sa{It is easy to check using the first-order optimality conditions that}
\begin{align*}
\bar{x}\in\argmin_{x\in X}\{{\L_0}(x,\lambda;\theta^*)+\fprod{F(\bar{x},\lambda,\theta)-F(\bar{x},\lambda,\theta^*),~x}\}.
\end{align*}
\sa{Thus, given a compact set $\Lambda\subset\cK^*$, invoking Lemma~\ref{lem:key} %(under the premise of Theorem~\ref{thm:pseudoLipschitz}) with
for $q = F(\bar{x},\lambda,\theta)-F(\bar{x},\lambda,\theta^*)$ implies that there exists $c>0$ such that for all $\lambda\in\Lambda$,}
{\small
\begin{eqnarray*}
\lefteqn{\norm{\bar{x}-x^*(\lambda;\theta^*)}}\\
&&\leq c~\norm{F(\bar{x},\lambda,\theta)-F(\bar{x},\lambda,\theta^*)},\\
&&\leq c~\norm{\grad_x f(\bar{x};\theta)-\grad_x f(\bar{x};\theta^*)+(A(\theta)-A(\theta^*)){^\top}\lambda},\\
&&\leq c~(L_{F,\theta}+L_A\norm{\lambda})\norm{\theta-\theta^*},\quad \sa{\forall~\theta\in\Theta.}
\end{eqnarray*}}%
Since $\bar{x}$ is an arbitrary element of $X^*(\lambda;\theta)$, the claim in \eqref{eq:main-assumption} holds for $\kappa(\lambda)\triangleq c~(L_{F,\theta}+L_A\norm{\lambda})$ which is clearly continuous on $\Lambda$.\qed

\section{Proof of Theorem~\ref{Theorem:Aug_increas_rho_linear_rate} (dual convergence)}
\label{sec:app-dual}
%{\bf Proof of Theorem~\ref{Theorem:Aug_increas_rho_linear_rate} (dual convergence).}
For now, let $\lambda^*\in\Lambda^*$ be some dual optimal solution to the dual of $\cC(\theta^*)$. From \eqref{def-pi}, for all $k\geq 1$, we have
{\small
\begin{align}
g_0(\pi_{\rho_k}(\lambda_k;\theta^*);\theta^*)\geq g_0(\lambda^*;\theta^*)-\frac{1}{2\rho_k}\norm{\lambda^*-\lambda_k}^2;
\end{align}}%
hence, $\liminf_k g_0(\pi_{\rho_k}(\lambda_k;\theta^*);\theta^*)\geq f^*$ because $g_0(\lambda^*;\theta^*)=f^*$, $\rho_k\nearrow+\infty$ and $\{\lambda_k\}$ is a bounded sequence -- see Lemma~\ref{Theorem: Aug_bnd_lambda_increa_rho}. Moreover, from~\eqref{eq:lagrangian}, clearly $\pi_{\rho_k}(\lambda_k;\theta^*)\in\cK^*$; hence, from weak duality $\limsup_k g_0(\pi_{\rho_k}(\lambda_k;\theta^*);\theta^*)\leq f^*$. Thus,
{\small
\begin{align}
\lim_k g_0(\pi_{\rho_k}(\lambda_k;\theta^*);\theta^*)=f^*.
\end{align}}%
Note that writing \eqref{proof:Aug_bnd_lambda_1} for $\rho=\rho_k$, we get $\norm{\lambda_{k+1}-\pi_{\rho_k}(\lambda_k;\theta_k)}\leq\sqrt{2\rho_k\alpha_k}$ for $k\geq 1$; hence, $\lim_k \norm{\lambda_{k+1}-\pi_{\rho_k}(\lambda_k;\theta^*)}=0$ because $\sqrt{\alpha_k\rho_k}\to 0$ -- see Condition~\ref{Assump:rho_k}. From~\eqref{eq:lipschitz-MoreauMap}, we also have $\norm{\pi_{\rho_k}(\lambda_k;\theta_k)-\pi_{\rho_k}(\lambda_k;\theta^*)}\leq \rho_k \kappa_d(\lambda_k)\norm{\theta_k-\theta^*}$ for $k\geq 1$. Since $\kappa_d(\cdot)$ is continuous, $\{\lambda_k\}$ is bounded, and $\rho_k\norm{\theta_k-\theta^*}\leq \rho_0\norm{\theta_0-\theta^*}\delta^k\to 0$, we have $\lim_k \norm{\pi_{\rho_k}(\lambda_k;\theta_k)-\pi_{\rho_k}(\lambda_k;\theta^*)}=0$. Thus, from triangular inequality,
{\small
\begin{eqnarray}
\label{eq:lambda-relation-rho_k}
\lim_k\norm{\lambda_{k+1}-\pi_{\rho_k}(\lambda_k;\theta^*)}=0. %\leq\\
%&&\lim_k(\norm{\lambda_{k+1}-\pi_{\rho_k}(\lambda_k;\theta_k)}+ \norm{\pi_{\rho_k}(\lambda_k;\theta_k)-\pi_{\rho_k}(\lambda_k;\theta^*)})=0.\nonumber
\end{eqnarray}}%
Since $\{\lambda_k\}\subset\cK^*$ is a bounded sequence, there exist a subsequence $\cS\subset\integers_+$ %and some $\bar{\lambda}\in\cK^*$
such that $\lim_{k\in\cS}\lambda_{k}$ exists and
\begin{align}
\lim_{k\in\cS}\lambda_{k}=\lim_{k\in\cS}\pi_{\rho_{k-1}}(\lambda_{k-1};\theta^*)\in\cK^*,
\end{align}
where the second equation follow from \eqref{eq:lambda-relation-rho_k}. Moreover,
{\small
\begin{align*}
f^*&=\lim_k g_0(\pi_{\rho_{k-1}}(\lambda_{k-1};\theta^*);\theta^*)\\
&=\limsup_{k\in\cS} g_0(\pi_{\rho_{k-1}}(\lambda_{k-1};\theta^*);\theta^*)\leq g_0(\lim_{k\in\cS}\lambda_k;\theta^*)\leq f^*,
\end{align*}}%
where the first inequality follows from upper-semicontinuity of $g_0(\cdot;\theta^*)$ and the second one follows from weak duality as $\lim_{k\in\cS}\lambda_k\in\cK^*$. Therefore, $g_0(\lim_{k\in\cS}\lambda_k;\theta^*)=f^*$, i.e., $\lim_{k\in\cS}\lambda_k\in\Lambda^*$. Finally, we use the same argument we did in the proof of Theorem~\ref{Theorem:Aug_bnd_dual_sub} to conclude with dual convergence. Indeed, since \eqref{eq:SM-convergence} is true for any dual optimal point, setting $\lambda^*\triangleq\lim_{k\in\cS}\lambda_k$ within \eqref{eq:SM-convergence} implies that $\lambda^*=\lim_{k\rightarrow\infty}\bar{\lambda}_k$. To be more precise, \eqref{eq:SM-convergence} implies that for all $i\geq 1$ and $\ell\geq i+1$,
{\small
\begin{eqnarray*}
\lefteqn{\|\lambda_{\ell}-\lambda^*\|}\\ %&\leq\|\lambda_{k+1}-\pi_\rho(\lambda_k,\theta_{k})\|+\|\pi_\rho(\lambda_{k},\theta_k)-\lambda^*\|\\
&&\leq \|\lambda_i-\lambda^*\| +\sum_{k\geq i}\Big(\sqrt{2\rho_k\alpha_k}+\rho_k {\kappa_d(\lambda^*)}\|\theta_{k}-\theta^*\|\Big). %\quad \forall\ i\geq 0.
\end{eqnarray*}}%
According to Condition~\ref{Assump:rho_k}, we have $\sum_k\sqrt{\alpha_k\rho_k}<+\infty$; moreover, since we assume $\delta\triangleq\tau\beta<1$, we also have $\sum_k\rho_k \|\theta_{k}-\theta^*\|<+\infty$. Therefore, given $\epsilon>0$, there exist $N_1\in\integers_+$ such that $\sum_{k\geq N_1}\Big(\sqrt{2\rho_k\alpha_k}+\rho_k {\kappa_d(\lambda^*)}\|\theta_{k}-\theta^*\|\Big)\leq \tfrac{\epsilon}{2}$. Since $\lim_{k\in\cS}\lambda_k=\lambda^*$, given $\epsilon>0$, there exists $N_2\in\cS$ such that $N_2>N_1$ and $\|\lambda_{N_2}-\lambda^*\|\leq\tfrac{\epsilon}{2}$ for all $k\in\cS$. Thus, we conclude that $\norm{\lambda_{\ell}-\lambda^*}\leq\epsilon$ for all $\ell> N_2$, which implies $\lim_k\lambda_k=\lambda^*$.
\section{Continuity of $X^*(\cdot)$ and $f^*(\cdot)$}
\label{sec:app-set_continuity}
\sa{For each $\theta\in\Theta$, define %$\Lambda^*(\theta)\triangleq\argmax_{\lambda\in\cK^*}g_0(\lambda;\theta)$ and
\begin{align}
\label{eq:parametric_primal_problem}
X^*(\theta)=\argmin_{x\in X} \{f(x;\theta):\ A(\theta)x+b(\theta)\in-\cK\},
\end{align}
and $f^*(\theta)\triangleq f(x^*(\theta);\theta)$ for some $x^*(\theta)\in X^*(\theta)$.}

\sa{Since %$A(\cdot)$ and $b(\cdot)$ are continuous on $\Theta$,
$\cK$ is closed and $X$ is a compact set, $\cS(\theta)\triangleq\{x\in X:\ A(\theta)x+b(\theta)\in-\cK\}$ is a compact set for all $\theta\in\Theta$. Next, we argue under some weak conditions that the multifunction $\cS:\Theta\rightarrow X$ is continuous on $\Theta$, i.e., both upper and lower-hemicontinuous.}

\sa{Let $\{\theta_k\}\subset \Theta$ and $x_k\in \cS(\theta_k)$ for $k\geq 0$ such that $\theta_k\to\bar{\theta}$ and $x_k\to\bar{x}$. Since $\Theta$ and $X$ are closed sets, $\bar{\theta}\in\Theta$ and $\bar{X}\in X$. Taking the limit of $A(\theta_k)x_k+b(\theta_k)\in-\cK$ implies that $A(\bar{\theta})\bar{x}+b(\bar{\theta})\in-\cK$; hence, $\bar{x}\in\cS(\bar{\theta})$. Therefore, $\cS:\Theta\rightarrow X$ is closed on $\Theta$; and since $X$ is compact, we conclude that $\cS:\Theta\rightarrow X$ is upper-hemicontinuous on $\Theta$. Moreover, according to~\cite[Theorem~4]{mackowiak2006some}, since $\Theta$ is compact and convex, if $\rm{int}~\Theta\neq\emptyset$ and {the condition in~\eqref{eq:local_condition}} holds for $\bar{\theta}=\theta^*$, then $\cS:\Theta\rightarrow X$ is lower-hemicontinuous at $\theta^*$.
\begin{eqnarray}
\label{eq:local_condition}
\lefteqn{\exists\epsilon>0\ \forall \theta\in\Theta\setminus\{\bar{\theta}\}\ \norm{\theta-\bar{\theta}}<\epsilon}\\
&& \hspace{-0.3in} \implies\ \exists t\in[0,1]\ \exists d\in \Theta\ \norm{d-\bar{\theta}}=\epsilon,\ \theta=t\bar{\theta}+(1-t) d.\nonumber
\end{eqnarray}
When $f(x;\theta)$ is continuous on $X\times \Theta$ and \eqref{eq:local_condition} holds for $\bar{\theta}=\theta^*$, we may claim from the previous sequence of arguments that $S: \Theta \to X$ is both upper and lower-hemicontinuous at $\theta^*$. Thus, invoking Berge's maximum theorem~\cite{ok2007real}, we can conclude that the function $f^*:\Theta\rightarrow \reals$ is continuous at $\theta^*$ and the multifunction $X^*:\Theta\rightarrow X$ is \emph{upper-hemicontinuous} at $\theta^*$.}
%\sa{The continuity properties of the multifunction $X^*:\Theta\rightarrow X$ follows from Berge's Maximum Theorem; indeed, when $f(x;\theta)$ is continuous on $X\times \Theta$ and \eqref{eq:local_condition} holds for $\bar{\theta}=\theta^*$, the function $f^*:\Theta\rightarrow \reals$ is continuous at $\theta^*$ and the multifunction $X^*:\Theta\rightarrow X$ is \emph{upper-hemicontinuous} at $\theta^*$. Similarly, whenever there exists a polytope $\cP\subset \Theta$ such that ${\rm int} \cP\neq \emptyset$ and $\theta^*\in \cP$, we can conclude that $f^*(\cdot)$ is continuous on $\cP$ and $X^*(\cdot)$ is upper-hemicontinuous on $\cP$.}

\begin{remark} \sa{A natural question is under what conditions \eqref{eq:local_condition} is expected to hold. We provide one such condition next. By \cite[Lemma~1]{mackowiak2006some}, if there exists a polytope $\cP\subset \Theta$ such that ${\rm int} \cP\neq \emptyset$ and $\theta^*\in \cP$, then \eqref{eq:local_condition} holds for all $\bar{\theta}\in\cP$ and $\cS$ is lower-hemicontinuous on $\cP$  -- hence, at $\theta^*$ as well.}
\end{remark}

%\end{appendices}
%\begin{IEEEbiography}
%v{Hesam Ahmadi}
%\end{IEEEbiography}
%	\vspace{-,5in}
%	\begin{IEEEbiography}
%{Necdet Serhat Aybat} xxx.
%\end{IEEEbiography}
%	\vspace{-,5in}
%	\begin{IEEEbiography}%[{\includegraphics[scale=.48]{uday.png}}]
%{Uday V. Shanbhag} holds the Gary and Sheila Bello Chaired professorship
%in the Industrial and Manufacturing Engineering
%at Pennsylvania State University.
%\end{IEEEbiography}

%\bibitem{IEEEhowto:kopka}
%H.~Kopka and P.~W. Daly, \emph{A Guide to \LaTeX}, 3rd~ed.\hskip 1em plus
%  0.5em minus 0.4em\relax Harlow, England: Addison-Wesley, 1999.
%
%\end{thebibliography}

% that's all folks
\end{document}

%% file: al_TAC_shared.tex
\usepackage{algorithmic}
\usepackage{url}
\usepackage{verbatim}
\usepackage{amsmath} % assumes amsmath package installed
\usepackage{amssymb}  % assumes amsmath package installed
\usepackage{amsfonts}
\usepackage{acronym}
\usepackage{multirow}
\usepackage[normalem]{ulem}
\usepackage{graphicx}
\usepackage{epsfig}
\usepackage{cite}
\usepackage{algorithm}
\usepackage{verbatim}
\usepackage[dvipsnames,usenames]{xcolor}
\usepackage{multicol}% http://ctan.org/pkg/multicols

\input{defs-TAC.tex}

\newcommand{\sa}[1]{{\color{black}#1}}

\newcommand\rout{\bgroup\markoverwith{\textcolor{red}{\rule[0.5ex]{2pt}{0.4pt}}}\ULon}

%% file: defs-TAC.tex
\def\grad{\nabla}
\def\noi{\noindent}

\def\be{\mathbf{e}}

\def\bk{\mathbf{k}}

\def\bx{\mathbf{x}}  %{\mbox{\boldmath $\lambda$}}

\def\L{{\boldsymbol{\Lambda}}}

\def\cB{\mathcal{B}}
\def\cC{\mathcal{C}}
\def\cD{\mathcal{D}}
\def\cE{\mathcal{E}}

\def\cG{\mathcal{G}}
\def\cH{\mathcal{H}}
\def\cI{\mathcal{I}}

\def\cK{\mathcal{K}}
\def\cL{\mathcal{L}}

\def\cN{\mathcal{N}}
\def\cO{\mathcal{O}}
\def\cP{\mathcal{P}}

\def\cS{\mathcal{S}}
\def\cT{\mathcal{T}}

\def\cV{\mathcal{V}}

\def\smskip{\smallskip}

\def\texitem#1{\par\smskip\noindent\hangindent 25pt
               \hbox to 25pt {\hss #1 ~}\ignorespaces}

\def\norm#1{\left\|#1\right\|}

\newcommand{\BEAS}{\begin{eqnarray*}}
\newcommand{\EEAS}{\end{eqnarray*}}
\newcommand{\BEA}{\begin{eqnarray}}
\newcommand{\EEA}{\end{eqnarray}}
\newcommand{\BEQ}{\begin{eqnarray}}
\newcommand{\EEQ}{\end{eqnarray}}
\newcommand{\BIT}{\begin{itemize}}
\newcommand{\EIT}{\end{itemize}}
\newcommand{\BNUM}{\begin{enumerate}}
\newcommand{\ENUM}{\end{enumerate}}

\newcommand{\BA}{\begin{array}}
\newcommand{\EA}{\end{array}}

\newcommand{\reals}{\mathbb{R}}
\newcommand{\integers}{\mathbb{Z}}

\newcommand{\argmin}{\mathop{\rm argmin}}
\newcommand{\argmax}{\mathop{\rm argmax}}

\newcommand{\dom}{\mathop{\bf dom}}

\newif\ifpagenumbering
\pagenumberingtrue

\pagenumberingfalse

\newsavebox{\theorembox}
\newsavebox{\lemmabox}
\newsavebox{\corollarybox}
\newsavebox{\propositionbox}
\newsavebox{\remarkbox}
\newsavebox{\assbox}
\newsavebox{\condbox}
\savebox{\theorembox}{\noindent\bf Theorem}
\savebox{\lemmabox}{\noindent\bf Lemma}
\savebox{\corollarybox}{\noindent\bf Corollary}
\savebox{\propositionbox}{\noindent\bf Proposition}
\savebox{\remarkbox}{\noindent\bf Remark}
\savebox{\assbox}{\noindent\bf Assumption}
\savebox{\condbox}{\noindent\bf Condition}
\newtheorem{condition}{\usebox{\condbox}}
\newtheorem{assumption}{\usebox{\assbox}}
\newtheorem{theorem}{\usebox{\theorembox}}
\newtheorem{lemma}{\usebox{\lemmabox}}
\newtheorem{remark}{\usebox{\remarkbox}}%[section]
\newtheorem{corollary}{\usebox{\corollarybox}}
\newtheorem{proposition}[theorem]{\usebox{\propositionbox}}
\newcommand{\qed}{\mbox{}\hspace*{\fill}\nolinebreak\mbox{$\rule{0.7em}{0.7em}$}}

\newenvironment{proof}{\par{\noi \bf Proof: }}{\(\qed\) \par}
\def\fskip#1{}

\newcommand{\Real}{\ensuremath{\mathbb{R}}}
\def\half  {{\textstyle{1\over 2}}}

\def\Kscr{{\cal K}}

\newcommand{\thetahat}{\widehat{\theta}}

\def\be{\begin{enumerate}}
\def\ee{\end{enumerate}}

\newcommand{\pmat}[1]{\begin{pmatrix} #1 \end{pmatrix}}

        \def\bkE{{\rm I\kern-.17em E}}
		\def\bk1{{\rm 1\kern-.17em l}}
		\def\bkD{{\rm I\kern-.17em D}}
		\def\bkR{{\rm I\kern-.17em R}}
		\def\bkP{{\rm I\kern-.17em P}}
		\def\bkY{{\bf \kern-.17em Y}}
		\def\bkZ{{\bf \kern-.17em Z}}

\def\L{\mathcal{L}}

\def\k{\kappa}

\def\xbar{\bar{x}}

\def\R{\mathbb{R}}

\def\ineq{\preceq_{\Kscr}}
\def\fprod#1{\left\langle#1\right\rangle}